\documentclass[12pt]{amsart}

\usepackage{amssymb,amsmath,enumerate,epsfig}
\usepackage{amsfonts,color}
\usepackage{a4,latexsym}

\usepackage{hyperref}
\hypersetup{pdfauthor=MS+TS}
\hypersetup{pdftitle=elliptic}

\newcommand{\C} {\mathbb{C}}
\newcommand{\Q} {\mathbb{Q}}
\newcommand{\N}  {\mathbb{N}}
\newcommand{\R} {\mathbb{R}}
\newcommand{\F}{\mathbb{F}}
\newcommand{\Z}{\mathbb{Z}}

\newcommand{\OO}{\mathcal{O}}
\newcommand{\NN}{\mathcal{N}}

\newcommand{\PP}{\mathbb{P}}
\newcommand{\NS}{\mbox{NS}}
\newcommand{\MW}{\mbox{MW}}
\newcommand{\MWL}{\mbox{MWL}}
\newcommand{\Km}{\mbox{Km}}
\newcommand{\A}{\mathbb{A}}

\newcommand{\T}{\mathbb{T}}
\newcommand{\Br}{\mathop{\rm Br}\nolimits}

\newcommand{\I}{{\mathop{\rm I}}}
\newcommand{\II}{{\mathop{\rm II}}}
\newcommand{\III}{{\mathop{\rm III}}}
\newcommand{\IV}{{\mathop{\rm IV}}}

\newtheorem{Theorem}{Theorem}[section]
\newtheorem{Proposition}[Theorem]{Proposition}
\newtheorem{Lemma}[Theorem]{Lemma}

\newtheorem{Corollary}[Theorem]{Corollary}

\theoremstyle{remark}
\newtheorem{Remark}[Theorem]{Remark}

\newtheorem{Example}[Theorem]{Example}

\theoremstyle{definition}

\newtheorem{Definition}[Theorem]{Definition}
\newtheorem{Question}[Theorem]{Question}

\begin{document}

\title{Elliptic surfaces}

\author{Matthias Sch\"utt}
\address{Institut f\"ur Algebraische Geometrie, Leibniz Universit\"at Hannover, Welfengarten 1, 30167 Hannover, Germany}
\email{schuett@math.uni-hannover.de}
\urladdr{http://www.iag.uni-hannover.de/\~{}schuett/}

\author{Tetsuji Shioda}
\address{Department of Mathematics, Rikkyo University, Tokyo 171-8501, Japan}
\email{shioda@rikkyo.ac.jp}
\urladdr{http://www.rkmath.rikkyo.ac.jp/math/shioda/}

\address{Research Institute,
for Mathematical Sciences,
Kyoto University,
Kyoto 606-8502,
Japan}
\email{shioda@kurims.kyoto-u.ac.jp}

\thanks{Partial funding from DFG under grant Schu 2266/2-2 and JSPS under Grant-in-Aid for Scientific Research (C) No.~20540051 is gratefully acknowledged.}

\date{February 22, 2010}

\begin{abstract}
This survey paper concerns elliptic surfaces with section.
We give a detailed overview of the theory including many examples.
Emphasis is placed on rational elliptic surfaces and elliptic K3 surfaces.
To this end, we particularly review the theory of Mordell-Weil lattices and address arithmetic questions.
\end{abstract}

\maketitle

\tableofcontents

\section{Introduction}

Elliptic surfaces are ubiquitous in the theory of algebraic surfaces.
They play a key role for many arithmetic and geometric considerations.
While this feature has become ever more clear during the last two decades, extensive survey papers and monographs seem to date back exclusively to the 80's and early 90's.


Hence we found it a good time to comprise a detailed overview of the theory of elliptic surfaces that includes the recent developments.
Since we are mainly aiming at algebraic and in particular arithmetic aspects,
it is natural to restrict to elliptic surfaces with sections.
Albeit this deters us from several interesting phenomena (Enriques surfaces to name but one example), 
elliptic surfaces with section are useful in general context as well since
one can always derive some basic information about a given elliptic surface without section from its Jacobian (cf.~\cite{Keum}; for the general theory see e.g.~\cite{LLR} and the references therein).

While we will develop the theory of elliptic surfaces with section in full generality, we will put particular emphasis on the following three related subjects:
\begin{itemize}
 \item 
rational elliptic surfaces,
\item
Mordell-Weil lattices,
\item
elliptic K3 surfaces and their arithmetic.
\end{itemize}
Throughout we will discuss examples whenever they become available.

\smallskip

Elliptic surfaces play a prominent role in the classification of algebraic surfaces and fit well with the general theory.
As we will see, they admit relatively simple foundations which in particular do not require too many prerequisites.
Yet elliptic surfaces are endowed with many different applications for several areas of research, such as lattice theory, singularities, group theory and modular curves.

To get an idea, consider height theory. 
On elliptic curves over number fields, height theory requires quite a bit of fairly technical machinery.
In contrast, elliptic surfaces naturally lead to a definition of height that very clear concepts and proofs underlie.
In particular, it is a key feature of elliptic surfaces among all algebraic surfaces that the structure of their N\'eron-Severi groups is well-understood; this property makes elliptic surfaces accessible to direct computations.

\smallskip

Throughout the paper, we will rarely give complete proofs;
most of the time, we will only briefly sketch the main ideas and concepts and include a reference for the reader interested in the details.
There are a few key references that should be mentioned separately:


\begin{center}
\begin{tabular}{ll}
textbooks on elliptic curves: &  Cassels \cite{Ca}, Silverman \cite{Si0}, \cite{Si}\\
textbooks on algebraic surfaces:\; & 
Mumford \cite{Mumford}, 
 Beauville \cite{Beau-surf}, Shafarevich \cite{Shafa-basic},\\
& Barth-Hulek-Peters-van de Ven \cite{BHPV}\\
textbook on lattices: & Conway-Sloane \cite{CS}\\
&\\
references for elliptic surfaces: &
Kodaira \cite{K},
Tate   \cite{Tate},
Shioda \cite{ShMW},\\
& Cossec-Dolgachev \cite{CD}, Miranda \cite{M1}
\end{tabular}
\end{center}

\smallskip

Elliptic surfaces form such a rich subject that we could not possibly have treated all aspects that we would have liked to include in this survey.
Along the same lines, it is possible, if not likely that at one point or another we might have missed an original reference, but we have always tried to give credit to the best of our knowledge, and our apologies go to those colleagues who might have evaded out attention.

\section{Elliptic curves}

We start by reviewing the basic theory of elliptic curves. Let $K$ denote a field; that could be a number field, the real or complex numbers, a local field, a finite field or a function field over either of the former fields.

\begin{Definition}
An elliptic curve $E$ is a smooth projective curve of genus one with a point $O$ defined over $K$.
\end{Definition}

One crucial property of elliptic curves is that for any field $K'\supseteq K$, the set of $K'$-rational points $E(K')$ will form a group with origin $O$. Let us consider the case $K=\C$ where we have the following analytic derivation of the group law.

\subsection{Analytic description}

Every complex elliptic curve is a one-dimensional torus. There is a lattice
\[
 \Lambda \subset \C \;\;\; \text{of rank two such that}\;\;\; E\cong \C/\Lambda.
\]
Then the composition in the group is induced by addition in $\C$, and the neutral element is the equivalence class of the origin. We can always normalise to the situation where
\[
 \Lambda = \Lambda_\tau = \Z + \tau\,\Z,\;\;\;\;\; \mbox{im}(\tau)>0.
\]
For later reference, we denote the corresponding elliptic curve by 
\begin{eqnarray}\label{eq:E_tau}
 E_\tau = \C/\Lambda_\tau.
\end{eqnarray}
Two complex tori are isomorphic if and only if the corresponding lattices are homothetic.

\subsection{Algebraic description}

Any elliptic curve $E$ has a model as a smooth cubic in $\PP^2$. As such, it has a point of inflection. 
Usually a flex $O$ is chosen as the neutral element of the group law.
It was classically shown how the general case can be reduced to this situation (cf.~for instance \cite[\S 8]{Ca}).

By Bezout's theorem, any line $\ell$ intersects a cubic in $\PP^2$ in three points $P,Q,R$, possibly with multiplicity. 
This fact forms the basis for defining the group law with $O$ as neutral element:

\begin{Definition}
Any three collinear points $P,Q,R$ on an elliptic curve $E\subset \PP^2$ add up to zero:
\[
 P+Q+R=O.
\]
\end{Definition}
Hence $P+Q$ is obtained as the third intersection point of the line through $O$ and $R$ with $E$.
The hard part about the group structure is proving the associativity, but we will not go into the details here.

\subsection{Weierstrass equation}

Given an elliptic curve $E\subset \PP^2$, we can always find a linear transformation that takes the origin of the group law $O$ to $[0,1,0]$ and the flex tangent to $E$ at $O$ to the line $\ell=\{z=0\}$. 
In the affine chart $z=1$, the equation of $E$ then takes the \textbf{generalised Weierstrass form}
\begin{eqnarray}\label{eq:NF}
E:\;\;\; y^2 + a_1 \,x\,y+ a_3\,y \, = \, x^3 + a_2 \,x^2 + a_4 \,x+a_6.
\end{eqnarray}
Outside this chart, $E$ only has the point $O=[0,1,0]$. The lines through $O$ are exactly  $\ell$ and the vertical lines $\{x=\mbox{const}\}$ in the usual coordinate system. 
The indices refer to some weighting (cf.~\ref{ss:j}).

Except for small characteristics, there are further simplifications: If char$(K)\neq 2$, then we can complete the square on the left-hand side of  (\ref{eq:NF}) to achieve $a_1=a_3=0$. Similarly, if char$(K)\neq 3$, we can assume $a_2=0$ after completing the cube on the right-hand side of (\ref{eq:NF}). Hence if char$(K)\neq 2,3$, then we can transform $E$ to the \textbf{Weierstrass form}
\begin{eqnarray}\label{eq:WF}
E:\;\;\; y^2 \, = \, x^3 +  a_4 \,x+a_6.
\end{eqnarray}

The name stems from the fact that the Weierstrass $\wp$-function for a given lattice $\Lambda$ satisfies a differential equation of the above shape. Hence $(\wp, \wp')$ gives a map from the complex torus $E_\tau$ to its Weierstrass form with $a_4, a_6$ expressed through Eisenstein series in $\tau$.

\subsection{Group operations}

One can easily write down the group operations for an elliptic curve $E$ given in (generalised) Weierstrass form. They are given as rational functions in the coefficients of $K$-rational points $P=(x_1,y_1), Q=(x_2,y_2)$ on $E$:

$$
\begin{array}{ccccc}
\text{operation} && \text{generalised Weierstrass form} && \text{Weierstrass form}\\
\hline
-P && (x_1, -y_1-a_1\,x_1-a_3) && (x_1, -y_1)\\
&&&&\\
 P+Q
 && \begin{matrix}
x_3=\lambda^2+a_1\lambda-a_2-x_1-x_2\\
y_3 = -(\lambda+a_1)x_3-\nu-a_3
    \end{matrix}
 && 
\begin{matrix}
x_3=\lambda^2-x_1-x_2\\
 y_3 = -\lambda x_3-\nu
\end{matrix}
\end{array}
$$
Here $y=\lambda\,x+\nu$ parametrises the line through $P,Q$ if $P\neq Q$, resp.~the tangent to $E$ at $P$ if $P=Q$ (under the assumption that $Q\neq -P$).

\subsection{Discriminant}

Recall that we require an elliptic curve to be smooth. If $E$ is given in Weierstrass form (\ref{eq:WF}), then smoothness can be phrased in terms of the \textbf{discriminant} of $E$:
\begin{eqnarray}\label{eq:disc}
 \Delta = -16\,(4\,a_4^3 + 27\,a_6^2).
\end{eqnarray}
For the generalised Weierstrass form (\ref{eq:NF}), the discriminant can be obtained by running through the completions of square and cube in the reverse direction.
It is a polynomial in the $a_i$ with integral coefficients (cf.~\cite[\S 1]{Tate}).

\begin{Lemma}
\label{Lem:disc}
A cubic curve $E$ is smooth if and only if $\Delta\neq 0$.
\end{Lemma}

\subsection{j-invariant}
\label{ss:j}

We have already seen when two complex tori are isomorphic. In this context, it is important to note that the Weierstrass form (\ref{eq:WF}) only allows one kind of coordinate transformation:
\begin{eqnarray}\label{eq:scale}
 x \mapsto u^2\,x, \;\;\; y\mapsto u^3\,y.
\end{eqnarray}
Such coordinate transformations are often called admissible.
An admissible transformation affects the coefficients and discriminant as follows:
\[
 a_4 \mapsto a_4/u^4,\;\;\; a_6 \mapsto a_6/u^6,\;\;\; \Delta \mapsto \Delta/u^{12}.
\]
It becomes clear that the indices of the coefficients refer to the respective exponents of $u$ which we consider as weights. 
All such transformation yield the same so-called \textbf{j-invariant}:
\[
 j = -1728\, \dfrac{(4\,a_4)^3}{\Delta}.
\]
\begin{Theorem}
\label{Thm:j}
Let $E, E'$ be elliptic curves in Weierstrass form (\ref{eq:WF}) over $K$.
If $E\cong E'$, then $j=j'$. The converse holds if additionally $a_4/a_4'$ is a fourth power and $a_6/a_6'$ is a sixth power in $K$, in particular if $K$ is algebraically closed.
\end{Theorem}
The extra assumption of the Theorem is neccessary, since we have to consider \textbf{quadratic twists}. For each $d\in K^*$,
\begin{eqnarray}\label{eq:quadr_twist}
E_d:\;\;\; y^2 \, = \, x^3 + d^2\, a_4 \,x + d^3 \, a_6.
\end{eqnarray}
defines an elliptic curve with $j'=j$. However, the isomorphism $E\cong E_d$ involves $\sqrt{d}$, thus it is a priori only defined over $K(\sqrt{d})$.
The two elliptic curves with extra automorphisms in \ref{ss:normal_j} moreover admit cubic (and thus sextic) resp.~quartic twists.

\subsection{Normal form for given j-invariant}
\label{ss:normal_j}

We want to show that for any field $K$ and any fixed $j\in K$, there is an elliptic curve $E$ over $K$ with j-invariant $j$. For this purpose, we will cover $j=0$ and $12^3$ separately in the sequel.

\begin{Definition}[Normal form for j]
 Let $j\in K\setminus\{0, 12^3\}$. Then the following elliptic curve in generalised Weierstrass form has j-invariant $j$:
\begin{eqnarray}\label{eq:normal_j}
 E:\;\;\; y^2 + x\,y = x^3 - \frac{36}{j-12^3}\,x - \frac{1}{j-12^3}.
\end{eqnarray}
\end{Definition}

The remaining two elliptic curve are special because they admit extra automorphisms. They can be given as follows:
$$
\begin{array}{llcll}
j=0: & y^2 + y & = & x^3, & \mbox{char}(K)\neq 3,\\
j=12^3:\;\;\;\; & y^2 & = & x^3 + x, \;\;\; & \mbox{char}(K)\neq 2.
\end{array}
$$

\subsection{Group structure over number fields}

The results for elliptic curves over number fields motivated many developments in the study of elliptic surfaces. We give a brief review of some of the main results.

In view of Faltings' theorem \cite{Fa}, elliptic curves can be considered the most intriguing curves if it comes to rational points over number fields. Namely, a curve of genus zero has either no points at all, or it has infinitely many rational points and is isomorphic to $\PP^1$. On the other hand, curves of genus $g>1$ have only finitely many points over any number field. Elliptic curves meanwhile feature both cases:

\begin{Theorem}[Mordell-Weil]
 Let $E$ be an elliptic curve over a number field $K$. Then $E(K)$ is a finitely generated abelian group:
\begin{eqnarray}\label{eq:MW-thm}
 E(K) \cong E(K)_{\text{tor}} \oplus \Z^r.
\end{eqnarray}
\end{Theorem}

One of the main ingredients of the proof is a quadratic height function on the $K$-rational points. 
For elliptic surfaces, the definition of such a height function will feature prominently in the theory of Mordell-Weil lattices. 
In fact, the height arises naturally from geometric considerations (cf.~\ref{s:MWL}).

\smallskip

There are two basic quantities in the Mordell-Weil theorem: the \textbf{torsion subgroup} and the \textbf{rank} of $E(K)$. 
The possible torsion subgroups can be bounded in terms of the degree of $K$ by work of Mazur and Merel (cf.~\cite{Mazur}).
Meanwhile it is an open question whether the ranks are bounded for given number fields $K$. 
In fact, elliptic surfaces prove a convenient tool to find elliptic curves of high rank by specialisation,
as we will briefly discuss in section \ref{s:ranks}.

For fixed classes of elliptic surfaces, one can establish analogous classifications, bounding torsion and rank.
The case of rational elliptic surfaces will be treated in Cor.~\ref{Cor:MW-RES}.

\subsection{Integral points}

If we want to talk about integral points on a variety, we always have to refer to some affine model of it. For elliptic curves, we can use the Weierstrass form (\ref{eq:WF}) or the generalised Weierstrass form (\ref{eq:NF}) for instance. 

\begin{Theorem}[Siegel]
\label{Thm:Siegel}
 Let $E$ be an elliptic curve with some affine model over $\Q$. Then $\# E(\Z)<\infty$.
\end{Theorem}

We will see that the same result holds for elliptic surfaces under mild conditions. In fact, we can even allow bounded denominators without losing the finiteness property.

\subsection{Bad places}

On a similar note, one can consider the set of bad places $v$ for an elliptic curve over a global field $K$.
Here $E$ attains  a singularity or degenerates over the residue field $K_v$;
this behaviour is called bad reduction.
In case $K$ is a number fields, one usually refers to the bad primes.

Since the bad places divide the discriminant by Lemma \ref{Lem:disc}, any elliptic curve has good reduction outside a finite set of bad places.
The following finiteness result for elliptic curves over $\Q$ addresses the inverse problem:

\begin{Theorem}[Shafarevich]
\label{Thm:Shafa}
Let $\mathcal S$ be a finite set of primes in $\Z$. Then there are only finitely many elliptic curves over $\Q$ up to $\Q$-isomorphism with good reduction outside $\mathcal S$.
\end{Theorem}

This result also generalises to elliptic surfaces under mild conditions which will appear quite naturally in the following (cf.~\ref{ss:finite}).
In fact, Shafarevich's original idea of proof generalises to almost all cases.
The generalisation is based on the theory of Mordell-Weil lattices as we will sketch in \ref{ss:finite'}.

\section{Elliptic surfaces}
\label{s:ell_surf}

In this section, we introduce elliptic surfaces.
We explain the interplay with elliptic curves over one-dimensional function fields using the generic fibre and the Kodaira-N\'eron model.
The central ingredient is the identification of sections of the surface and points on the generic fibre.

\subsection{}

We shall define elliptic surfaces in a geometric way. Therefore we let $k=\bar k$ denote an algebraically closed field, and $C$ a smooth projective curve over $k$.
Later on, we will also consider elliptic surfaces over non-algebraically closed fields (such as number fields and finite fields); in that case we will require the following conditions to be valid for the base change to the algebraic closure.

\begin{Definition}
An \textbf{elliptic surface} $S$ over $C$ is a smooth projective surface $S$ with an elliptic fibration over $C$, i.e.~a surjective morphism
\[
 f:\;\;\; S \to C,
\]
such that
\begin{enumerate}
 \item 
almost all fibres are smooth curves of genus $1$;
\item
no fibre contains an exceptional curve of the first kind.
\end{enumerate}
\end{Definition}

The second condition stems from the classification of algebraic surfaces.
An exceptional curve of the first kind is a smooth rational curve of self-intersection $-1$ (also called $(-1)$-curve). Naturally, $(-1)$-curves occur as exceptional divisors of blow-ups of surfaces at smooth points. In fact, one can always successively blow-down $(-1)$-curves to reduce to a smooth minimal model.

In the context of elliptic surface, we ask for a relatively minimal model with respect to the elliptic fibration. The next example will show that an elliptic surface need not be minimal as an algebraic surface when we forget about the structure of the elliptic fibration.

\subsection{Example: Cubic pencil}
\label{ss:cubic_pencil}
Let $g, h$ denote homogeneous cubic polynomials in three variables. Unless $g$ and $h$ have a common factor or involve only two of the coordinates of $\PP^2$ in total, the cubic pencil
\[
 S=\{\lambda\,g+\mu\,h = 0\}\subset \PP^2\times \PP^1.
\]
defines an elliptic surface $S$ over $\PP^1$ with projection onto the second factor. This surface is isomorphic to $\PP^2$ blown up in the nine base points of the pencil. This is visible from the projection onto the first factor. The exceptional divisors are not contained in the fibres: they will be sections.

If the base points are not distinct, then $S$ has ordinary double points as singularities. In this case, the minimal desingularisation gives rise to the elliptic surface associated to the cubic pencil.

\smallskip

For instance, the normal form for given j-invariant from \ref{ss:normal_j} is seen
as a cubic pencil with parameter $j$ after multiplying its equation by $(j-12^3)$. Thus the associated elliptic surface $S$ is rational with function field $k(S)=k(x,y)$.
The pencil has a multiple base point at $[0,1,0]$; the multiplicity is seven if char$(k)\neq 2,3$, and nine otherwise.
We will check in \ref{ss:ex_normal} how these multiplicities account for a singular fibre with many components.

\subsection{Sections}
A section of an elliptic surface $f:\; S\to C$ is a morphism 
\[
\pi:\;C\to S\;\;\; \text{ such that }\;\;\; f\circ\pi=\mbox{id}_C.
\]
The existence of a section is very convenient since then we can work with a Weierstrass form (\ref{eq:WF}) where we regard the generic fibre $E$ as an elliptic curve over the function field $k(C)$.
In particular, this implies that the sections form an abelian group: $E(k(C))$.
Here we choose one section as the origin of the group law.
We call it zero section and denote it by $O$.

\smallskip

In order to exploit the analogies with the number field case, we shall employ the following conventions throughout this paper:

\smallskip

\textbf{Conventions:}
\begin{enumerate}
 \item 
Every elliptic surface has a section.
\item
Every elliptic surface $S$ has a singular fibre. In particular, $S$ is not isomorphic to a product $E\times C$.
\end{enumerate}


\begin{figure}[ht!]
\setlength{\unitlength}{.45in}
\begin{picture}(8.5,4.5)(0,0)

{
\put(0.5,1.2){\line(1,0){6}}}


\qbezier(1.2,2.5)(1.8,2.5)(1.8,4)
\qbezier(1.2,2.5)(1.8,2.5)(1.8,1)

\qbezier(5.2,2.5)(5.2,1)(5.8,4)
\qbezier(5.2,2.5)(5.2,4)(5.8,1)

\qbezier(3.2,2.5)(3.2,3.5)(3.5,3)
\qbezier(3.5,3)(3.8,2.5)(3.8,4)

\qbezier(3.2,2.5)(3.2,1.5)(3.5,2)
\qbezier(3.5,2)(3.8,2.5)(3.8,1)

\put(0,0.75){\framebox(7,3.5){}}

\put(0,0){\line(1,0){7}}
\put(8,0){\makebox(0,0)[l]{$C$}}
\put(8,2.5){\makebox(0,0)[l]{$X$}}
\put(8.1,1.75){\vector(0,-1){1}}

\put(6.4,1.5){\makebox(0,0)[l]{$O$}}
\put(3.8,0){\circle*{.1}}

\end{picture}
\caption{Elliptic surface with section and singular fibres}
\label{Fig:ES}
\end{figure}
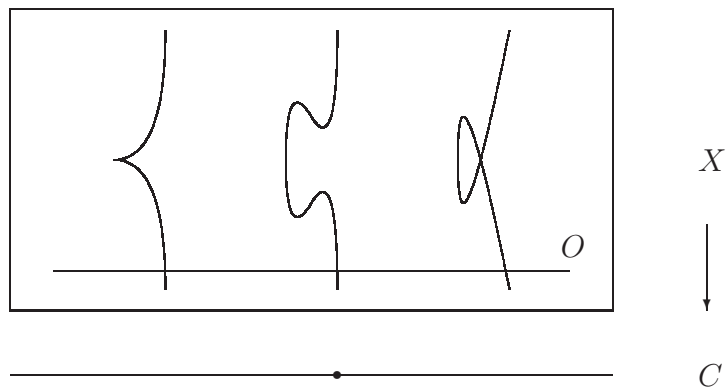

In terms of the Weierstrass form (\ref{eq:WF}), the second convention guarantees that the coefficients are not contained in the ground field $k$, but that they really involve the function field $k(C)$. We will discuss elliptic surfaces without singular fibres briefly in relation with minimal Weierstrass forms (cf.~\ref{ss:min}, \ref{ss:no}).

The assumption of a section is fairly strong. In fact, it rules out a number of surfaces that we could have studied otherwise.
For instance, every Enriques surface admits an elliptic fibration, albeit never with section.
However, we can always pass to the Jacobian elliptic surface. This process guarantees the existence of a section and preserves many properties, for instance the Picard number. The Jacobian of an elliptic fibration on an Enriques surface is a rational elliptic surface.

\subsection{Interplay between sections and points on the generic fibre}
\label{ss:s-p}

An elliptic surface $S$ over $C$ (with section) gives rise to an elliptic curve $E$ over the function field $k(C)$ by way of the generic fibre. Explicitly, a section $\pi:\;C\to S$ produces a $k(C)$ rational point $P$ on $E$ as follows: Let $\Gamma$ denote the image of $C$ under $\pi$ in $S$. Then $P=E\cap \Gamma$.

If we return to the example of a cubic pencil, then the base points will in general give nine sections, i.e.~points on the generic fibre $E$. Choosing one of them as the origin for the group law, we are led to ask whether the other sections give independent points. In general, this turns out to hold true (cf.~\ref{ss:MWL_RES}).

Conversely, let $P$ be a $k(C)$-rational point on the generic fibre  $E$. A priori, $P$ is only defined on the smooth fibres, but we can consider the closure $\Gamma$ of $P$ in $S$ (so that $\Gamma\cap E=P$).
Restricting the fibration to $\Gamma$, we obtain a birational morphism of $\Gamma$ onto the non-singular curve $C$.
\[
 f|_\Gamma:\; \Gamma \to C
\]
By Zariski's main theorem, $f|_\Gamma$ is an isomorphism. The inverse gives the unique section associated to the $k(C)$-rational point $P$.


\subsection{Kodaira-N\'eron model}

Suppose we are given an elliptic curve $E$ over the function field $k(C)$ of a curve $C$. Then the Kodaira-N\'eron model describes how to associate an elliptic surface $S\to C$ over $k$ to $E$ whose generic fibre returns exactly $E$.

At first, we can omit the singular fibres. Here we remove all those points from $C$ where the discriminant $\Delta$ vanishes (cf.~(\ref{eq:disc})). Denote the resulting punctured curve by $C^\circ$. Above every point of $C^\circ$, we read off the fibre -- a smooth elliptic curve -- from $E$. This gives a quasi-projective surface $S^\circ$ with a smooth elliptic fibration
\[ 
f^\circ:\;\;\; S^\circ \to C^\circ.
\]
Here one can simply think of the Weierstrass equation restricted to $C^\circ$ (after adding the point at $\infty$ to every smooth fibre).
It remains to fill in suitable singular fibres at the points omitted from $C$. 

For instance, if the Weierstrass form (\ref{eq:WF}) defines a smooth surface everywhere, then all fibres turn out to be irreducible. The singular fibres are either nodal or cuspidal rational curves:

\begin{figure}[ht!]
\setlength{\unitlength}{.45in}
\begin{picture}(10,2.8)(-1,-0.3)

%
%

\thinlines


%


\qbezier(0,1)(0,0)(2,2)
\qbezier(0,1)(0,2)(2,0)

\qbezier(5,1)(7.3,1)(8,2)
\qbezier(5,1)(7.3,1)(8,0)

\end{picture}
\caption{Irreducible singular fibres}
\label{Fig:I_1,II}
\end{figure}
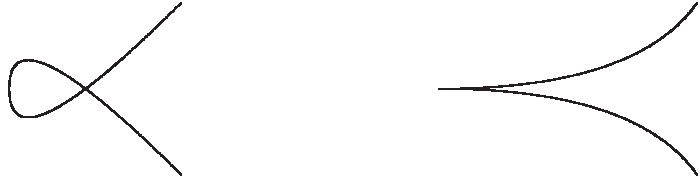

If the surface is not smooth somewhere, then we resolve singularities minimally. We will give an explicit description of the desingularisation process in the next section where we study Tate's algorithm in some detail. We will also recall Kodaira's classification of the possible singular fibres.

\smallskip

We conclude this section with a very rough discussion of the uniqueness of the Kodaira-N\'eron model. Assume we have two desingularisations $S, \tilde S$ which are relatively minimal with respect to the elliptic fibration over $C$:
$$
\begin{array}{ccccc}
S & \sim_{\mbox{bir}} & S^\circ & \sim_{\mbox{bir}} & \tilde S\\
& \searrow\; && \;\;\;\;\swarrow &\\
&& C &&
\end{array}
$$
By assumption there is a birational morphism 
\[
 S --\to \tilde S.
\]
The surface classification builds on the fact that every birational morphism is a succession of smooth blow-ups and blow-downs. 
By construction, the desingularisations $S$ and $\tilde S$ are isomorphic outside the singular fibres.
But then the fibres do not contain any $(-1)$-curves, since the surfaces are relatively minimal.
Hence $S\cong \tilde S$, and we have proven the uniqueness of the Kodaira-N\'eron model.

\section{Singular fibres}

We have seen that the main difference between an elliptic surface and its generic fibre consists in the singular fibres.
In this section, we discuss the classification of singular fibres.
This centers around Tate's algorithm.
We will mainly focus on the case outside characteristics $2, 3$, but also mention pathologies in  characteristic $p$ such as wild ramification.

\subsection{Classification}

For complex elliptic surfaces, the singular fibres were first classified by Kodaira \cite{K}. The fibre type is determined locally by the monodromy. 
In the general situation which includes positive characteristic, 
N\'eron pioneered  a canonical way to determine the singular fibres \cite{Neron}; 
later Tate exhibited a simplified algorithm that is valid over perfect fields \cite{Tate}. 
It turns out that there are no other types of singular fibres than already known to Kodaira. Except for characteristics 2 and 3, the fibre types are again determined by very simple local information.

The case of non-perfect base fields has recently been analysed by M.~Szydlo \cite{Szydlo}.
In characteristics $2$ and $3$, Tate's algorithm can fail as  there are new types of singular fibres appearing.
Here we will only consider perfect base fields and follow Tate's exposition \cite{Tate} (see also \cite{Si}).
It will turn out that all the irreducible components of a singular fibre are rational curves.
If a singular fibre is irreducible, then it is a rational curve with a node or a cusp. In Kodaira's 
classification of singular fibres, it is denoted by type $\I_1$ resp.~$\II$.

\smallskip

If a singular fibre is reducible, i.e. with more than one irreducible components, then 
each irreducible component is a smooth rational curve with self-intersection number -2
(a (-2)-curve) on the elliptic surface. Reducible singular fibres are classified into the following types: 

\begin{itemize}
 \item 
two infinite series  
$\I_n (n>1) $ and $\I_n^* (n \geq 0)$, and
\item 
five types 
$\III, \IV, \II^*, \III^*, \IV^*$. 
\end{itemize}

In general, an elliptic surface can also have multiple fibres.
For instance, this happens for Enriques surfaces.
In the present situation, however, multiple fibres are ruled out by the existence of a section.

\subsection{Tate's algorithm}

In the sequel, we shall sketch the first few steps in Tate's algorithm to give a flavour of the ideas involved. For simplicity, we shall only consider fields of characteristic different from two. Hence we can work with an extended Weierstrass form
\begin{eqnarray}\label{eq:ext_WF}
 y^2 = x^3 + a_2\,x^2+a_4\,x+a_6.
\end{eqnarray}
Here the discriminant is given by
\begin{eqnarray}\label{eq:disc_ext}
 \Delta = -27\,a_6^2+18\,a_2\,a_4\,a_6+a_2^2\,a_4^2-4\,a_2^3\,a_6-4\,a_4^3.
\end{eqnarray}
Recall that singular fibres are characterised by the vanishing of $\Delta$. Since we can work locally, we fix a local parameter $t$ on $C$ with normalised valuation $v$. (This set-up generalises to the number field case where one studies the reduction modulo some prime ideal and picks a uniformiser.) 

\smallskip

Assume that there is a singular fibre at $t$, that is $v(\Delta)>0$. By a translation in $x$,  we can move the singularity to $(0,0)$. Then the extended Weierstrass form transforms as
\begin{eqnarray}\label{eq:Tate-1}
  y^2 = x^3 + a_2\,x^2+t\,a_4'\,x+t\,a_6'.
\end{eqnarray}
If $\boldsymbol{t\nmid a_2}$, then the above equation at $t=0$ describes a nodal rational curve. We call the fibre \textbf{multiplicative}.
If $\boldsymbol{t\mid a_2}$, then (\ref{eq:Tate-1}) defines a cuspidal rational curve at $t=0$. We refer to \textbf{additive reduction}.
The terminology is motivated by the group structure of the rational points on the special fibre.
In either situation, the special point $(0,0)$ is a surface singularity if and only if $t|a_6'$.

\subsection{Multiplicative reduction}
Let $t\nmid a_2$ in (\ref{eq:Tate-1}). 
From the summand $a_2^3\,a_6$ of $\Delta$ in (\ref{eq:disc_ext}), it is immediate that $(0,0)$ is a surface singularity if and only if
\[
 v(\Delta)>1.
\]
If $v(\Delta)=1$, then $(0,0)$ is only a singularity of the fibre, but not of the surface.
Hence the singular fibre at $t=0$ is the irreducible nodal rational curve given by (\ref{eq:Tate-1}), cf.~Fig.~\ref{Fig:I_1,II}. Kodaira introduced the notation $\boldsymbol{\I_1}$ for this fibre type. 
So let us now assume that $v(\Delta)=n>1$. We have to resolve the singularity at $(0,0)$ of the surface given by (\ref{eq:Tate-1}).

This can be achieved as follows: Let $m=\lfloor \frac n2\rfloor$ be the greatest integer not exceeding $n/2$.
Then translate $x$ such that $t^{m+1}|a_4:$
\begin{eqnarray}\label{eq:Tate-m}
  y^2 = x^3 + a_2\,x^2+t^{m+1}\,a_4'\,x+a_6.
\end{eqnarray}
Then $v(\Delta)=n$ is equivalent to $v(a_6)=n$. Now we blow-up $m$ times at the point $(0,0)$. The first $(m-1)$ blow-ups introduce two $\PP^1$'s each. These exceptional divisors are locally given by
\[
 y'^2 = a_2(0)\,x'^2,
\]
so they will be conjugate over $k(\sqrt{a_2(0)})/k$. 
This extension appears on the nodal curve in (\ref{eq:Tate-1}) as the splitting field of the tangents to the singular point.
According to whether $\sqrt{a_2(0)}\in k$, one refers to split and non-split multiplicative reduction.

After each blow-up, we continue blowing up at $(0,0)$, the intersection of the exceptional divisors.
At the final step, the local equation of the special fibre is
\[
 y^2 = a_2(0)\,x^2 + (a_6/t^{2m})(0).
\]
This describes a single rational component, if $n=2m$ is even, or again two components if $n=2m+1$ is odd. At any rate, the blown-up surface is smooth locally around $t=0$. In total, the desingularisation has added $(n-1)$ rational curves of self-intersection $-2$. Hence the singular fibre consists of a cycle of $n$ rational curves, meeting transversally. Following Kodaira, we refer to this fibre type as $\boldsymbol{\I_n}$.

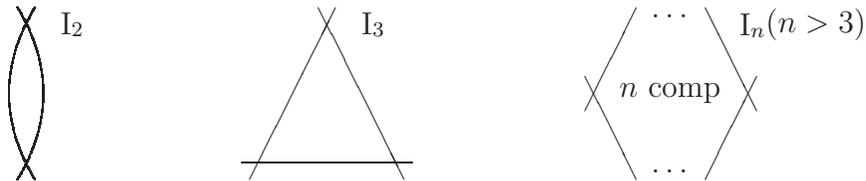
\begin{figure}[ht!]
\setlength{\unitlength}{.45in}
\begin{picture}(10,2.7)(0,-0.3)

\thinlines

\qbezier(0.6,0)(0,1)(0.6,2)
\qbezier(0.4,0)(1,1)(0.4,2)
\put(0.9,1.8){\makebox(0,0)[l]{$\I_2$}}

\put(3,0.2){\line(1,0){2}}
\put(3.1,0){\line(1,2){1}}
\put(3.9,2){\line(1,-2){1}}
\put(4.4,1.8){\makebox(0,0)[l]{$\I_3$}}

\put(7,0.8){\line(1,2){0.6}}
\put(7,1.2){\line(1,-2){0.6}}
\put(8.4,0){\line(1,2){0.6}}
\put(8.4,2){\line(1,-2){0.6}}
\put(7.77,1.9){\makebox(0,0)[l]{$\hdots$}}
\put(7.77,0.1){\makebox(0,0)[l]{$\hdots$}}
\put(7.4,1){\makebox(0,0)[l]{$n$ comp}}
\put(8.8,1.8){\makebox(0,0)[l]{$\I_n (n>3)$}}

\end{picture}
\caption{Singular fibre of type $\I_n\; (n>1)$}
\label{Fig:In}
\end{figure}

\subsection{Additive reduction}
\label{ss:add}

We now let $t\mid a_2$ in (\ref{eq:Tate-1}).
We have to distinguish whether $(0,0)$ is a surface singularity, i.e.~whether $t\mid a_6'$. If the characteristic is different from 2 and 3, then this is equivalent to 
\[
 v(\Delta)>2.
\]
In characteristics 2 and 3, however, the picture is not as uniform. It is visible from (\ref{eq:disc_ext}) that $v(\Delta)\geq 3$ if char$(k)=3$. In characteristic 2, one even has $v(\Delta)\geq 4$ in case of additive reduction. This is due to the existence of wild ramification which we will briefly comment on in section \ref{ss:wild}.
In any characteristic, if $(0,0)$ is a smooth surface point, then the singular fibre is a cuspidal rational curve -- type $\mathbf{\II}$, cf.~Fig.~\ref{Fig:I_1,II}.

If $(0,0)$ is a surface singularity, we sketch the possible singular fibres resulting from the desingularisation process.
There are three possibilities for the exceptional divisor of the first blow-up: 
\begin{enumerate}
 \item a rational curve of degree two, meeting the strict transform of the cuspidal curve tangentially in one point;
\item
Two lines, possibly conjugate in a quadratic extension of $k$, meeting the strict transform of the cuspidal curve in one point;
\item a double line.
\end{enumerate}
In the first two cases, we have reached the desingularisation and refer to type $\boldsymbol{\III}$ resp.~$\boldsymbol{\IV}$. 
The third case requires further blow-ups, each introducing lines of multiplicity up to six and self-intersection $-2$.
The resulting non-reduced fibres are visualised in the following figure. Multiple components appear in thick print.
Unless they have a label, their multiplicity will be two.

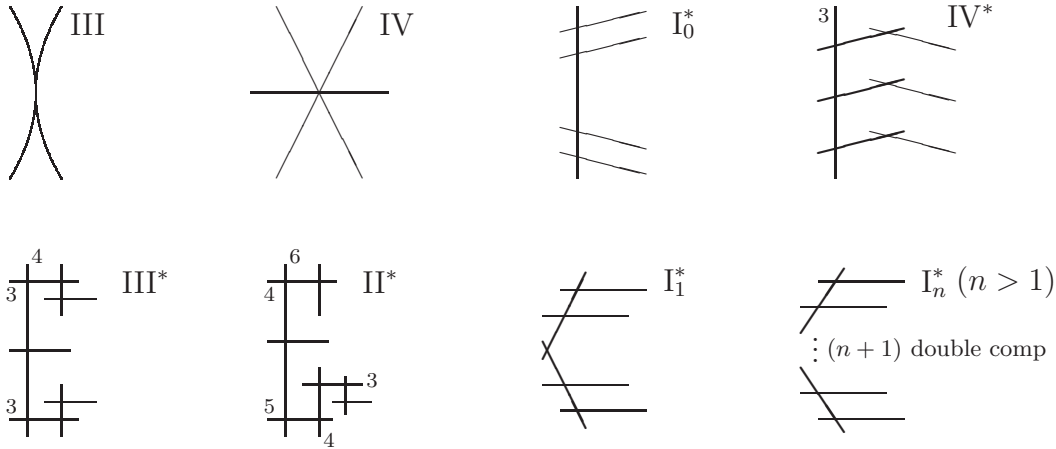
\begin{figure}[ht!]
\setlength{\unitlength}{.45in}
\begin{picture}(10,5.8)(1,-0.3)

\thinlines

\qbezier(0.6,3)(0,4)(0.6,5)
\qbezier(0,3)(0.6,4)(0,5)
\put(0.7,4.8){\makebox(0,0)[l]{$\III$}}

\thinlines
\put(2.8,4){\line(1,0){1.6}}
\put(3.1,3){\line(1,2){1}}
\put(3.1,5){\line(1,-2){1}}
\put(4.3,4.8){\makebox(0,0)[l]{$\IV$}}

\thinlines

\put(6.4,4.7){\line(4,1){1}}
\put(6.4,4.4){\line(4,1){1}}
\put(6.4,3.3){\line(4,-1){1}}
\put(6.4,3.6){\line(4,-1){1}}

\thicklines
\put(6.6,5){\line(0,-1){2}}
\put(7.7,4.8){\makebox(0,0)[l]{$\I_0^*$}}

\thinlines
\put(10,4.75){\line(4,-1){1}}
\put(10,4.15){\line(4,-1){1}}
\put(10,3.55){\line(4,-1){1}}

\thicklines
\put(9.6,5){\line(0,-1){2}}
\put(9.4,4.5){\line(4,1){1}}
\put(9.4,3.9){\line(4,1){1}}
\put(9.4,3.3){\line(4,1){1}}

\put(10.9,4.9){\makebox(0,0)[l]{$\IV^*$}}
\put(9.4,4.9){\makebox(0,0)[l]{{\tiny 3}}}

\thinlines

\put(0.4,1.6){\line(1,0){0.6}}
\put(0.4,0.4){\line(1,0){0.6}}

\thicklines
\put(0.2,2){\line(0,-1){2}}
\put(0.25,2.1){\makebox(0,0)[l]{{\tiny 4}}}

\put(0,1.8){\line(1,0){0.8}}
\put(-0.05,1.65){\makebox(0,0)[l]{{\tiny 3}}}

\put(0,1){\line(1,0){0.7}}
\put(0,0.2){\line(1,0){0.8}}
\put(-0.05,0.35){\makebox(0,0)[l]{{\tiny 3}}}

\put(0.6,2){\line(0,-1){0.6}}
\put(0.6,0){\line(0,1){0.6}}

\put(1.3,1.8){\makebox(0,0)[l]{$\III^*$}}

\thinlines
\put(3.75,0.4){\line(1,0){0.45}}

\thicklines
\put(3.2,2){\line(0,-1){2}}
\put(3.25,2.1){\makebox(0,0)[l]{{\tiny 6}}}

\put(3,1.8){\line(1,0){0.8}}
\put(2.95,1.65){\makebox(0,0)[l]{{\tiny 4}}}

\put(3,1.1){\line(1,0){0.7}}
\put(3,0.2){\line(1,0){0.75}}
\put(2.95,0.35){\makebox(0,0)[l]{{\tiny 5}}}

\put(3.6,2){\line(0,-1){0.6}}

\put(3.6,0){\line(0,1){0.8}}
\put(3.65,-0.05){\makebox(0,0)[l]{{\tiny 4}}}

\put(3.4,0.6){\line(1,0){0.7}}
\put(4.15,0.65){\makebox(0,0)[l]{{\tiny 3}}}

\put(3.9,0.7){\line(0,-1){0.45}}

\put(4.1,1.8){\makebox(0,0)[l]{$\II^*$}}

\thinlines
\put(6.4,1.7){\line(1,0){1}}
\put(6.2,1.4){\line(1,0){1}}
\put(6.4,0.3){\line(1,0){1}}
\put(6.2,0.6){\line(1,0){1}}

\thicklines
\put(6.2,0.9){\line(1,2){0.5}}
\put(6.2,1.1){\line(1,-2){0.5}}

\put(7.6,1.8){\makebox(0,0)[l]{$\I_1^*$}}

\thinlines
\put(9.4,1.8){\line(1,0){1}}
\put(9.2,1.5){\line(1,0){1}}
\put(9.4,0.2){\line(1,0){1}}
\put(9.2,0.5){\line(1,0){1}}

\thicklines
\put(9.2,1.2){\line(2,3){0.5}}
\put(9.2,0.8){\line(2,-3){0.5}}

\put(10.6,1.8){\makebox(0,0)[l]{$\I_n^*\;(n>1)$}}
\put(9.32,1.1){\makebox(0,0)[l]{$\vdots$}}
\put(9.5,1){\makebox(0,0)[l]{\scriptsize{$(n+1)$ double comp}}}

\end{picture}
\caption{Reducible additive singular fibres}
\label{Fig:add}
\end{figure}

\subsection{Additive fibres in characteristics different from 2,3}
\label{s:add_not_2,3}

If the characteristic differs from 2 and 3, then we can determine the type of the singular fibre directly from the discriminant and very few local information -- as opposed to having to  go through Tate's algorithm in full detail. In essence, this is due to the Weierstrass form (\ref{eq:WF}) that we can reduce to. Here we only have to consider the vanishing orders of $a_4$ and $a_6$ next to $\Delta$.

Assume that $v(\Delta)=n>0$. Then the type of the singular fibre can be read off from the following table. A single entry means equality, while otherwise we will indicate the possible range of valuations.

\begin{table}[ht!]
$$
\begin{array}{c|cc}
\hline
 \text{fibre type} & v(a_4) & v(a_6)\\
\hline
\I_0
 & 
\begin{cases}
\;\;0\\ \geq 0
\end{cases}
& 
\begin{matrix}
\geq 0\\ 0
\end{matrix}
\\
\I_n (n>0) & 0 & 0\\
\II & \geq 1 & 1\\
\III & 1 & \geq 2\\
\IV & \geq 2 & 2\\
\I_0^* &
\begin{cases}
 \;\;2 \\ \geq 2
\end{cases}
& 
\begin{matrix}
\geq 3\\ 3
\end{matrix}\\
\I_{n-6}^* (n>6) & 2 & 3\\
\IV^* & \geq 3 & 4\\
\III^* & 3 & \geq 5\\
\II^* & \geq 4 & 5\\
\hline
\end{array}
$$
\caption{Vanishing orders of Weierstrass form coefficients}
\end{table}

The following relation is worth noticing:
\begin{eqnarray}\label{eq:number}
 \# \begin{pmatrix}
 \text{components of}\\
 \text{singular fibre}
 \end{pmatrix}
  =
\begin{cases}
 v(\Delta) & \text{ in the multiplicative case}; \\
v(\Delta)-1 & \text{ in the additive case};
\end{cases}
\end{eqnarray}

\subsection{Wild ramification}
\label{ss:wild}

The notion of wild ramification was introduced to account for a possible discrepancy in the additive case in (\ref{eq:number}).
Following Ogg's foundational paper \cite{Ogg} (see also \cite{T-Saito}), we denote the index of wild ramification at an additive fibre by $\delta_v$:
\[
\delta_v = v(\Delta) - 1 - \# (\text{components of an additive fibre}).
\]
In characteristics 2 and 3, some additive fibre types imply higher vanishing order of $\Delta$ than predicted by the number of components. 
We have seen this happen for the cuspidal rational curve (type $\II$) in \ref{ss:add}. 
In order to determine the fibre type, one thus has to go through Tate's algorithm -- as opposed to the other characteristics where we read off the fibre type from discriminant and Weierstrass form.

If a fibre type admits wild ramification, then the index of wild ramification is unbounded within all elliptic surfaces.
By a simple case-by-case analysis (see \cite[Prop.~16]{SS2}), one finds  the following information on the index of wild ramification where again a single entry indicates equality:

\begin{table}[ht!]
$$
\begin{array}{c|c|c}
\hline
\text{fibre type} & p=2 & p=3\\
\hline

\I_n (n\geq 0) & 0 & 0\\

\II & \geq 2 & \geq 1\\

\III & \geq 1 & 0\\

\IV & 0 & \geq 1\\

\I_n^* \;
\begin{matrix}
(n=1) \\ (n\neq 1)
\end{matrix}
 & 
 \begin{matrix}
 1 \\
 \geq 2
 \end{matrix} & 
 \begin{matrix} 
 0\\ 0
 \end{matrix}\\

\IV^* & 0 & \geq 1\\

\III^* & \geq 1 & 0\\

\II^* & \geq 1 & \geq 1 \\
\hline
\end{array}
$$
\caption{Wild ramification indices}
\end{table}

\subsection{Example: Normal form for j}
\label{ss:ex_normal}

Recall the normal form for an elliptic surface of given j-invariant from \ref{ss:normal_j}. Here we want to interpret the normal form itself as an elliptic surface over $\PP^1$ with parameter $j=t$. Hence the generic fibre is defined over $k(t)$. An integral model is obtain from (\ref{eq:normal_j}) by a simple scaling of $x$ and $y$ as in (\ref{eq:scale}):
\begin{eqnarray}\label{eq:normal_int}
 S:\;\;\;  y^2 + (t-12^3)\,x\,y = x^3 - 36\,(t-12^3)^3\,x - (t-12^3)^5.
\end{eqnarray}
We want to determine the singular fibres of $S$. We first compute the discriminant
\[
 \Delta = t^2\,(t-12^3)^9.
\]
The fibres at $0$ and $12^3$ are easily seen to be additive. If char$(k)\neq 2, 3$, then translations in $y$ and subsequently $x$ yield the Weierstrass form
\[
 y^2 = x^3 - \frac 1{48}\,t\,(t-12^3)^3\,x +\frac 1{864}\,t\,(t-12^3)^5.
\]
By the above table, the singular fibres have type $\II$ at $t=0$ and $\III^*$ at $t=12^3$. If char$(k)=2$ or $3$, then the two additive fibres collapse so that $v_{t=0}(\Delta)=11$. Tate's algorithm only terminates at fibre type $\II^*$. Hence there is wild ramification of index $1$.
The extra fibre component is due to the collision of all base points mod $2$ and $3$ (cf.~\ref{ss:cubic_pencil}).

In any characteristic, we obtain an integral model at $\infty$ by the variable change
\begin{eqnarray}\label{eq:normal-oo}
 t\mapsto 1/s,\;\;\; x\mapsto x/s^2,\;\;\; y\mapsto y/s^3.
\end{eqnarray}
This results in the discriminant $\Delta=s\,(1-12^3\,s)^9$. Since the vanishing order of $\Delta$ at $s=0$ is one, there has to be a singular fibre of type $\I_1$, a nodal rational curve.

\subsection{Minimal Weierstrass form}
\label{ss:min}

From the table in section \ref{s:add_not_2,3} we also see why Tate's algorithm terminates: If it returns none of the above fibre types, then 
\[
 v(a_4)\geq 4,\;\;\;  v(a_6)\geq 6.
\]
But then we can apply the change of variables 
\[
x \mapsto t^2\,x,\;\;\; y\mapsto t^3\,y 
\]
to obtain an isomorphic equation that is integral at $t$. (For the generalised Weierstrass form (\ref{eq:NF}), this requires $v(a_i)\geq i$ after suitable translations of $x, y$.)
Such a transformation lets $v(\Delta)$ drop by 12, hence it can only happen a finite number of times. We call this process \textbf{minimalising} and the resulting equation the \textbf{minimal Weierstrass form} (locally at $t=0$). 

\smallskip

The notion of minimality involves a little subtlety. 
Namely, we work in the coordinate ring $k[C]$ of a chosen affine open of the base curve $C$. 
For instance, if the base curve $C$ is isomorphic to $\PP^1$, then $k[C]=k[t]$ and every ideal in $k[C]$ is principal. 
In consequence, an elliptic surface over $\PP^1$ admits a globally minimal Weierstrass form. In general, however, $k[C]$ need not be a principal ideal domain, and thus there might not be a globally minimal Weierstrass form.

As a strange twist, the Tate algorithm itself can also have the converse effect and make a minimal Weierstrass form non-minimal.
For instance, in the multiplicative case the translation in $x$ yielding (\ref{eq:Tate-m}) may well be non-minimal at $\infty$ if $m$ is big compared to the degrees of the coefficients $a_i$ (cf.~\ref{ss:IP^1}).

\subsection{Application: Elliptic surfaces without singular fibres}
\label{ss:no}

We give an example of an elliptic surface without singular fibres which is nonetheless not isomorphic to a product of curves.
The example builds on the local nature of a minimal Weierstrass form.

Let $C$ be a hyperelliptic curve of genus $g(C)>0$. We can give $C$ by a polynomial $h \in k[u]$ of degree at least three without multiple factors:
\[
 C:\;\;\; v^2 = h(u).
\]
Let $u_0$ denote any ramification point of the double cover $C\to \PP^1$, so $h(u) = (u-u_0)\,h'$.
Consider the following elliptic surface $S$ over $C$:
\[
 S:\;\;\; y^2 = x^3 + (u-u_0)^2\,x + (u-u_0)^3.
\]
Then $S$ has discriminant $\Delta = c\,(u-u_0)^6$, but $S$ does not have a singular fibre at $(u_0,0)$. The above Weierstrass form for $S$ is not minimal at $u_0$, since locally $h'$ is invertible. Hence we can write
\[
 S:\;\;\; y^2 = x^3 + \left(\frac{v^2}{h'}\right)^2\,x + \left(\frac{v^2}{h'}\right)^3.
\]
Then rescaling $x, y$ eliminates the local parameter $v$. The resulting surface is smooth and thus locally minimal at $(u_0,0)$.

\subsection{Elliptic surfaces over the projective line}
\label{ss:IP^1}

In this survey, we will mostly consider elliptic surfaces with base curve $\PP^1$. Such surfaces have a globally minimal generalised Weierstrass form. In terms of the Weierstrass form with polynomial coefficients, minimality requires 
\[
 v(a_4)< 4 \;\;\text{ or }\;\; v(a_6)<6\;\;\; \text{ at each place of } \PP^1.
\]
After minimalising at all finite places, we fix the smallest integer $n$ such that $\deg(a_i)\leq ni$. As in \ref{ss:ex_normal}, we derive a local equation at $\infty$ with coefficients 
\[
a_i' = s^{ni}\,a_i(1/s).
\]
Alternatively, we can homogenise the coefficients $a_i(t)$ as polynomials in two variables $s,t$ of degree $n\,i$. Then the discriminant is a homogeneous polynomial of degree $12 n$. This approach leads to the common interpretation of the minimal Weierstrass form of an elliptic surface over $\PP^1$ with section as a hypersurface in weighted projective space $\PP[1,1,2n,3n]$. Here the non-trivial weights refer to $x$ and $y$ respectively.
This interpretation allows one to apply several results from the general theory of hypersurfaces in weighted projective spaces to our situation (cf.~\ref{ss:N-L}).

\smallskip

The integer $n$ has an important property: It determines how an elliptic surface over $\PP^1$ with section fits into the classification of projective surfaces. We let $\kappa$ denote the Kodaira dimension.

\begin{center}
\begin{tabular}{ccll}
 $n=1$ && rational elliptic surface, & $\kappa=-\infty$,\\
$n=2$ && K3 surface, & $\kappa=0$,\\
$n>2$ && honestly elliptic surface, & $\kappa=1$.
\end{tabular}
\end{center}
The invariant $n$ is equal to the arithmetic genus $\chi$ which will be introduced in \ref{ss:Euler}.

In general, elliptic surfaces over curves of positive genus do also have Kodaira dimension one.
We will cover rational elliptic surfaces and elliptic K3 surfaces with section in some detail in sections \ref{s:RES} and \ref{s:K3}.

\section{Base change and quadratic twisting}

This section is devoted to the study of base change as one of the most fundamental operations in the context of elliptic surfaces.
There are immediate applications to the problem of constructing specific elliptic surfaces that will recur throughout this survey.
Base change also motivates other considerations such as quadratic twisting or inseparability.

\subsection{}

Base change provides a convenient method to produce new elliptic surfaces from old ones.
The set-up is as follows:
Suppose we have an elliptic surface
\[
 S \stackrel{f}{\longrightarrow} C.
\]
In order to apply a base change, we only need another projective curve $B$ mapping surjectively to $C$.
Formally the base change of $S$ from $C$ to $B$ then is defined as a fibre product:
$$
\begin{array}{ccc}
 S\times_C B & \longrightarrow & B\\
\downarrow && \downarrow\\
S & \stackrel{f}{\longrightarrow} & C
\end{array}
$$
In practice, we simply pull-back the Weierstrass form (or in general the equation) of $S$ over $C$ via the morphism $B\to C$. Clearly this replaces smooth fibres by smooth fibres.

\subsection{Singular fibres under base change}
\label{ss:base_change}

The effect of a base change on the singular fibres depends on the local ramification of the morphism $B\to C$. Singular fibres at unramified points are replaced by a fixed number of copies in the base change
(the number being the degree of the morphism).
If there is ramification, we have to be more careful. Of course the vanishing orders of the polynomials of the Weierstrass form and of the discriminant multiply by the ramification index. 
This suffices to solve the base change problem for multiplicative fibres:
a base change of local ramification index $d$ replaces a fibre of type $\I_n (n\geq 0)$ by a fibre of type $\I_{nd}$.
For additive fibres, however, the pull-back of the Weierstrass form might become non-minimal.

\smallskip

In the absence of wild ramification, the singular fibre resulting from the base change can be predicted exactly. For instance, if we have a Weierstrass form, then this can be read off from the vanishing orders of $a_4, a_6$ and $\Delta$. In particular it is also evident how often we have to minimalise. 

Depending on the index $d$ of ramification modulo some small integer, Table \ref{Tab:ram} collects the behaviour of singular fibres without wild ramification:

\begin{table}[ht!]
$$
\begin{array}{|c|c|c|}
\hline
\II, \II^* & d\mod 6 & \text{fibre}\\
\hline
& 1 & \II\\
& 2 & \IV\\
&3 & \I_0^*\\
&4 & \IV^*\\
&5 & \II^*\\
&6 & \I_0\\
\hline
\end{array}
\;\;\;\;
\begin{array}{|c|c|c|}
\hline
\III, \III^* & d\mod 4 & \text{fibre}\\
\hline
&1 & \III\\
&2 & \I_0^*\\
&3 & \III^*\\
&4 & \I_0\\
\hline
\multicolumn{2}{c}{}\\
\multicolumn{2}{c}{}\\
\end{array}
\;\;\;\;
\begin{array}{|c|c|c|}
\hline
\IV, \IV^* & d\mod 3 & \text{fibre}\\
\hline
&1 & \IV\\
&2 & \IV^*\\
&3 & \I_0\\
\hline
\hline
\I_n^* & d\mod 2 & \text{fibre}\\
\hline
&1 & \I_{nd}^*\\
&2 & \I_{nd}\\
\hline
\end{array}
$$
\caption{Singular fibres under ramified base change}
\label{Tab:ram}
\end{table}

In the presence of wild ramification, the local analysis can be much more complicated. We will see an example in \ref{ss:ex_quad}. Confer \cite{MT} for a general account of that situation.

By inspection of the tables, additive fibres are \textbf{potentially semi-stable}: after a suitable base change, they are replaced by semi-stable fibres, i.e.~multiplicative fibres of type $\I_n ~(n\geq 0)$.

\subsection{Base change by the j-map}

Often it is convenient to consider the j-invariant of an elliptic surface $S$ over $C$ as a morphism from $C$ to $\PP^1$.  
To emphasise this situation, we will also refer to the j-map.  Assume that $j \neq const.$
If $j: C\to\PP^1$ is separable, then it is unramified outside a finite sets of points.
Usually the ramification locus includes $0$ and $12^3$ and often also $\infty$.
Recall that these are exactly the places of the singular fibres of the normal form for the j-invariant.

Thus the question arises to which extent elliptic surfaces admit an interpretation as base change from the normal form via the j-map.
In this context, it is crucial that the j-invariant determines a unique elliptic curve up to isomorphism only under the assumption that the ground field $K$ is algebraically closed (Thm.~\ref{Thm:j}). 
Since here we work over function fields, there might in general be some discrepancy between the given elliptic surface with j-map $j$ and the base change of the normal form by $j$. By \ref{ss:j}, this discrepancy is related to quadratic twisting.

\subsection{Quadratic twisting}

Two elliptic surfaces with the same j-map have the same singular fibres up to some quadratic twist. In particular, the smooth fibres are isomorphic over $\bar k$, the algebraic closure of the ground field.

If char$(k)\neq 2$, then any elliptic surface with section admits an extended Weierstrass form (\ref{eq:ext_WF}). Hence quadratic twisting can be understood in analogy with (\ref{eq:quadr_twist}):
\[
S_d:\;\;\; y^2 = x^3 + da_2x^2+d^2a_4x+d^3a_6. 
\]
Here we point out the effect of a quadratic twist on the singular fibres. 
Again, this can be read off directly from Tate's algorithm:
\begin{eqnarray}\label{eq:twist-fibre}
&& \I_n \leftrightarrow \I_n^*,\;\;\; \II \leftrightarrow \IV^*,\;\;\; \III \leftrightarrow \III^*,\;\;\; \IV \leftrightarrow \II^*.
\end{eqnarray}
Note that any two elliptic surfaces that are quadratic twists of each other become isomorphic after a suitable finite base change. 
From \ref{ss:base_change} and (\ref{eq:twist-fibre}), it is immediate that this only requires even ramification order at all twisted fibres. 

In characteristic $2$, there is a different notion of quadratic twisting, related to the Artin-Schreier map $t\mapsto t^2+t$ (see~for instance \cite[p.~10]{SS2}).

\subsection{Example: Quadratic base change of normal form for j}
\label{ss:ex_quad}

As in \ref{ss:ex_normal}, we consider the normal form for given j-invariant as an elliptic surface over $\PP^1$ with coordinate $t$.
To look into an example we want to apply the following quadratic base change
\begin{eqnarray*}
\pi:\;\; \PP^1 & \to & \;\;\;\PP^1\\
t\;\; & \mapsto & t^2+12^3.
\end{eqnarray*}
The base change ramifies at $12^3$ and $\infty$. 
By \ref{ss:base_change}, the base changed surface $S'$ has singular fibres $\II, \II, \I_0^*, \I_2$ unless char$(k)=2,3$. To verify this, we pull back the integral model (\ref{eq:normal_int}) via $\pi^*$:
\[
 S':\;\;\;  y^2 + t^2\,x\,y = x^3 - 36\,t^6\,x - t^{10}.
\]
This elliptic curve is not minimal at $t=0$. A change of variables $x\mapsto x\,t^2, y\mapsto y\,t^3$ yields the minimal form
\[
 S':\;\;\;  y^2 + t\,x\,y = x^3 - 36\,t^2\,x - t^{4}
\]
with $\I_0^*$ fibre at $t=0$ and $\I_2$ at $\infty$.
If char$(k)\neq 2$, then we can complete the square on the left-hand side:
\[
 S':\;\;\;  y^2  = x^3 + \frac 14\, t^2\,x^2 - 36\,t^2\,x - t^{4}.
\]
Now we can apply quadratic twists by choosing an even number of places in $\PP^1$. For instance, twisting at $0$ and $\alpha$ in effect moves the singular fibre of type $\I_0^*$ from $t=0$ to $\alpha$:
\[
 \tilde S:\;\;\;  y^2  = x^3 + \frac 14\, t\,(t-\alpha)\,x^2 - 36\,(t-\alpha)^2\,x - t\,(t-\alpha)^{3}.
\]
By choosing $\alpha=\pm 24\sqrt{-3}$, we can also merge fibres of type $\II$ and $\I_0^*$ to obtain type $\IV^*$. Similarly, twisting at $0$ and $\infty$ results in a fibre of type $\I_2^*$ at $\infty$.

We conclude by noting the singular fibres and ramification indices of $S'$ in characteristics $2$ and $3$. In either characteristic, $S'$ has two singular fibres, one of which at $\infty$ has type $\I_2$. The other singular fibre at $0$ has type $\III^*$ with ramification index one if char$(k)=2$, and $\IV^*$ with ramification index two if char$(k)=3$. Both claims are easily verified with Tate's algorithm.

\subsection{Construction of elliptic surfaces}

The ideas from this section can also be used in the reversed direction to construct elliptic surfaces over a given curve $C$ with prescribed singular fibres. Namely one has to find a polynomial map
\[
 C \to \PP^1
\]
whose ramification at  $0, 12^3, \infty$ is compatible with the singular fibres.
Here compatibility might involve a quadratic twist.
For instance, an elliptic surface over $\PP^1$ with singular fibres of types $\II, \II, \I_2^*$ comes from the normal form after the base change in \ref{ss:ex_quad}, but one has to apply the quadratic twist at $0$ and $\infty$ after the base change.

As a special case, we are led to consider polynomial maps to $\PP^1$ which are unramified outside three points (usually normalised as $0,1,\infty$).
These maps are often called Belyi maps. 
Belyi proved that a projective curve can be defined over a number field if and only if it admits a Belyi map. 
We will see in \ref{ss:mono} how this problem can be translated into a purely combinatorial question.
Thus we will be able to answer existence questions by relatively easy means.

\subsection{Uniqueness}

By \ref{ss:j}, an elliptic curve is determined by its j-invariant if it is defined over an algebraically closed field.
With elliptic surfaces resp.~curves over function fields $k(C)$, we have seen that quadratic twisting causes some serious trouble.
Obviously, another obstruction occurs when $j$ is constant -- despite the surface not being isomorphic to a product.
Such elliptic surfaces are called isotrivial, see \ref{ss:isotrivial} for details.

If on the other hand $j$ is not constant, we deduce that the elliptic surface is uniquely determined (up to $\bar k$-isomorphism) by the pair $(j, \Delta)$.
For semi-stable surfaces (i.e.~without additive fibres), this criterion can be improved as follows:

\begin{Theorem}
\label{Thm:isom-ss}
Let $S, S'$ denote semi-stable elliptic surfaces with section over the same projective curve $C$. Then $S$ and $S'$ are isomorphic as elliptic surfaces if and only if $\Delta=u^{12} \Delta'$ for some $u\in k^*$.
\end{Theorem}

Over $\bar k$, the theorem can be rephrased as an equality of divisors on the curve $C$ (again only valid for semi-stable elliptic surfaces):
\[
 (\Delta)=(\Delta').
\]
Here we have to distinguish from isomorphisms of the  underlying surfaces which forget about the structure of an elliptic fibration.
For instance, K3 surfaces may very well admit several non-isomorphic elliptic fibrations (cf.~12.2).

\subsection{Inseparable base change}

Since the discriminant encodes the bad places of an elliptic curve over a global field, Thm.~\ref{Thm:isom-ss} brings us back to the theme of Shafarevich's finiteness result for elliptic curves over $\Q$ (Thm.~\ref{Thm:Shafa}).
A proof of the function field analogue would only require to show that there are only finitely many discriminant with given zeroes on the curve $C$.
However, such a general statement is bound to fail in positive characteristic.
This failure is due to the existence of inseparable base change.

Namely, let the elliptic surface $S\to C$ be defined over the finite field $\F_q$. 
Then we can base change $S$ by the Frobenius morphism $C\to C$ which raises coordinates to their $q$-th powers.
Clearly this does not change the set of bad places, but only the vanishing orders of $\Delta$.
For instance, at the $\I_n$ fibres the vanishing orders are multiplied by $q$.
Hence the inseparable base change results in a fibre of type $\I_{qn}$.

\subsection{Conductor}

We explain one possibility to take inseparability and wild ramification into account.
We associate to the discriminant $\Delta$ a divisor $\mathcal{N}$ on the base curve $C$ which we call the \textbf{conductor}:
\[
 \mathcal N = \sum_{v\in C} u_v\,v\;\;\; \text{ with }\;\;\; u_v=
\begin{cases}
 0, & \text{ if the fibre at $v$ is smooth};\\
1, & \text{ if the fibre at $v$ is multiplicative};\\
2+\delta_v, & \text{ if the fibre at $v$ is additive}.\\
\end{cases}
\]
By definition $(\Delta)-\mathcal N$ is effective, in particular~$\deg(\Delta)\geq \deg(\mathcal N)$.
The two degree are related by the following result:

\begin{Theorem}[Pesenti-Szpiro {\cite[Thm.~0.1]{P-S}}]
\label{Thm:PS}
Let $S\to C$ be a non-isotrivial elliptic surface over $\bar\F_p$. Let the j-map have inseparability degree $e$. Then
\[
\deg(\Delta)\leq 6 \,p^e (\deg(\NN) +2g(C)-2).
\]
\end{Theorem}
In characteristic zero the above result has been known to hold after removing the factor $p^e$
 (see \cite{ShEMS} and also \cite[p.~287]{Si}).
Once we understand the structure of the N\'eron-Severi group of elliptic surfaces, we will apply the theorem to prove a function field anlogue of Shafarevich's finiteness result over $\Q$ in \ref{ss:finite}.

\section{Mordell-Weil group and N\'eron-Severi lattice}
\label{s:NS}

In \ref{ss:s-p} we have seen how a rational point on the generic fibre gives rise to a section on the corresponding elliptic surface and vice versa. 
In this section we investigate this relation in precise detail.
Along the way, we will also compute many fundamental invariants of elliptic surfaces such as canonical divisor, arithmetic genus, Euler number, Betti and Hodge numbers.

\subsection{}

We fix some notation and terminology: 
As usual, we let $S\to C$ be an elliptic surface over $k$ with generic fibre $E$ over $k(C)$.
The $K$-rational points $E(K)$ form a group which is traditionally called \textbf{Mordell-Weil group}. We will usually denote the points on $E$ by $P, Q$ etc. 

Each point $P$ determines a section $C\to S$ which we interprete as a divisor on $S$. To avoid confusion, we shall denote this curve by $\bar P$.
Without further distinguishing, we will also consider $\bar P$ as an element in the \textbf{N\'eron-Severi group} of $S$
\[
 \NS(S) = \{\text{divisors}\}/\approx.
\]
Here $\approx$ denotes \textbf{algebraic equivalence}. For instance, all fibres of the elliptic fibration $S\to C$ are algebraically equivalent.
The rank of $\NS(S)$ is called the \text{Picard number}:
\[
 \rho(S) = \mbox{rank}(\NS(S)).
\]
By the Hodge index theorem, $\NS(S)$ has signature $(1,\rho(S)-1)$.

\subsection{Mordell-Weil group vs.~N\'eron-Severi group}

We now state the three fundamental results that relate the Mordell-Weil group and the N\'eron-Severi group of an elliptic surface with section.
All theorems require our assumption that the elliptic surface has a singular fibre.

\begin{Theorem}
\label{Thm:MW}
 $E(K)$ is a finitely generated group.
\end{Theorem}

This result is a  special case of the Mordell-Weil theorem, in generality for abelian varieties over suitable global fields (cf.~\cite[\S 6]{Lang}, \cite[\S 4]{Serre}).
Here we will sketch the geometric argument from \cite{ShMW}. 
The first step is to prove the corresponding result for the N\'eron-Severi group.

\begin{Theorem}
\label{Thm:NS}
$\NS(S)$ is finitely generated and torsion-free.
\end{Theorem}

The finiteness part is again valid in more generality for projective varieties as a special case of the theorem of the base \cite[\S 5]{Lang}.
On an elliptic surface, one can use intersection theory to prove both claims (cf.~Thm.~\ref{Thm:a-n}).
The connection between Thm.~\ref{Thm:MW} and \ref{Thm:NS} is provided by a third result:

\begin{Theorem}
\label{Thm:E-NS}
Let $T$ denote the subgroup of $\NS(S)$ generated by the zero section and fibre components. Then the map $P\mapsto \bar P \mod T$ gives an isomorphism
\[
 E(K) \cong \NS(S)/T.
\]
\end{Theorem}
The following terminology has established itself for elliptic surfaces: fibre components are called \textbf{vertical} divisors, sections \textbf{horizontal}. In this terminology, the above theorem states that $\NS(S)$ is generated by vertical and horizontal divisors
(as opposed to involving multisections).


\begin{figure}[ht!]
\setlength{\unitlength}{.45in}
\begin{picture}(9,4.5)(-0.3,0)
\thicklines

\put(0.5,1.2){\line(1,0){6}}

{\color{blue}
\put(0.5,2.7){\line(1,0){0.48}}
\put(5.06,2.7){\line(1,0){0.4}}
\put(5.62,2.7){\line(1,0){0.88}}
\put(1.13,2.7){\line(1,0){0.6}} 
\put(1.88,2.7){\line(1,0){3.05}}}
\thinlines

%

\put(1.8,4){\line(0,-1){3}}
\put(1,2.8){\line(2,-3){1}}
\put(1,2.2){\line(2,3){1}}

\qbezier(5,2.5)(5,1)(6,4)
\qbezier(5,2.5)(5,4)(6,1)

\qbezier(3,2.5)(3,3.5)(3.5,3)
\qbezier(3.5,3)(4,2.5)(4,4)

\qbezier(3,2.5)(3,1.5)(3.5,2)
\qbezier(3.5,2)(4,2.5)(4,1)

\thinlines
\put(0,0.75){\framebox(7,3.5){}}

\put(0,0){\line(1,0){7}}
\put(8,0){\makebox(0,0)[l]{$C$}}
\put(8,2.5){\makebox(0,0)[l]{$X$}}
\put(8.1,1.75){\vector(0,-1){1}}

\put(6.4,1.5){\makebox(0,0)[l]{$O$}}
\put(4,0){\circle*{.1}}

\end{picture}
\caption{Horizontal and vertical divisors}
\label{Fig:HV}
\end{figure}
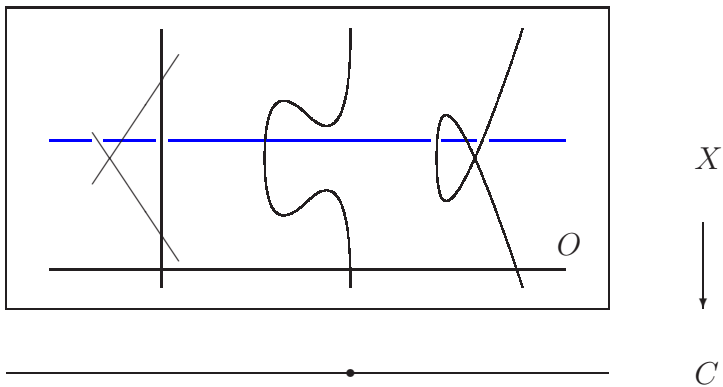

The idea here is to prove Thm.s \ref{Thm:NS} and \ref{Thm:E-NS}, since then Thm.~\ref{Thm:MW} follows immediately.
In the sequel, we sketch the main line of proof.
Details can be found in \cite{ShMW}.
We will restrict ourselves to those ideas which shed special light on elliptic surfaces or which are particularly relevant for our later issues.

\subsection{Numerical Equivalence}

On any projective surface, algebraic equivalence implies \textbf{numerical equivalence} $\equiv$ where we identify divisors with the same intersection behaviour.
Hence the intersection of divisors defines a symmetric bilinear pairing on $\NS(S)$.
It endows $\NS(S)$ (up to torsion) 
with the structure of an integral lattice, the \textbf{N\'eron-Severi lattice}.
The reader is referred to \cite{CS} for generalities on lattices; 
note that in this survey we are only concerned with \textbf{non-degenerate lattices} (which will sometimes be stated explicitly in order to stress the condition, but often be omitted).
In the following, we will use the phrases N\'eron-Severi group and N\'eron-Severi lattice interchangeably. 

\begin{Lemma}
 Modulo numerical equivalence, the N\'eron-Severi group is finitely generated.
\end{Lemma}

The lemma is a direct consequence of the existence of the cycle map
\[
 \gamma:\;\;\; \NS(S) \to H^2(S).
\]
Here one can work with $\ell$-adic \'etale cohomology in general ($\ell$ a prime different from the characteristic). 
In particular $H^2(S)$ is a finite-dimensional $\Q_\ell$-vectorspace, equipped with cup-product which is a non-degenerate pairing with values in $H^4(S)\cong\Q_\ell$. 
Notably, $\gamma$ preserves the pairing.
In consequence, the kernel of $\gamma$ comprises exactly the elements which are numerically equivalent to zero.
But then $\NS(S)/\equiv$ embeds into a finite-dimensional vector space.
Hence this quotient is finitely generated and torsion free.
Thus, Thm.~\ref{Thm:NS} will follow once we have proved the following result:

\begin{Theorem}
\label{Thm:a-n}
 On an elliptic surface, algebraic and numerical equivalence coincide.
\end{Theorem}

This result is very convenient for practical reasons, since we can now solve problems concerning divisor classes in $\NS(S)$ simply by calculating intersection numbers.

\subsection{The trivial lattice}

Thm.~\ref{Thm:E-NS} introduced the \textbf{trivial lattice} $T$, generated by the zero section and fibre components. Since any two fibres are algebraically equivalent, we only have to consider a general fibre and fibre components not met by the zero section. We set up some notation for the remainder of this paper. As usual it refers to an elliptic surface $S\stackrel{f}{\longrightarrow} C$ with zero section $O$.
\begin{center}
\begin{tabular}{ll}
 $F$ & a general fibre\\
$F_v$ & the fibre $f^{-1}(v)$ above $v\in C$\\
$m_v$ & the number of components of the fibre $F_v$\\
$\Sigma$ & the points of $C$ underneath singular fibres: $\Sigma=\{v\in C; F_v$ is singular$\}$\\
$R$ & the points of $C$ underneath reducible fibres: $R=\{v\in C; F_v$ is reducible$\}$\\
$\Theta_{v,0}$ & the component of $F_v$ met by the zero section -- the zero component\\
$\Theta_{v,i}$ & the other components of $F_v$ ($i=1,\hdots,m_v-1$)\\
$T_v$ & the lattice generated by fibre components in $F_v$ not meeting the zero section:\\
& $T_v=\langle \Theta_{v,i};\;\; 1\leq i\leq m_v-1\rangle.$
\end{tabular}
\end{center}
In this notation, the trivial lattice $T\subseteq\NS(S)$ is defined as the orthogonal sum
\begin{eqnarray}\label{eq:Triv}
 T = \langle \bar O, F \rangle \oplus \bigoplus_{v\in R} T_v.
\end{eqnarray}
\begin{Proposition}
\label{Prop:triv_lattice}
The divisor classes  of $\{\bar O, \;F, \;\Theta_{v,i};\;\; v\in R,\; 1\leq i\leq m_v-1\}$ form a $\Z$-basis of $T$. In particular,
\[
 \mbox{rank}(T) = 2 +\sum_{v\in R} (m_v-1).
\]
\end{Proposition}
To see that the given divisor classes form a basis, one can compute the intersection matrix $M$ of $T$ and verify that $\det(M)\neq 0$. 
Here it suffices to consider the separate summands from (\ref{eq:Triv}).
The first summand has intersection form
\[
 \begin{pmatrix}
  \bar O^2 & 1\\
1 & 0
 \end{pmatrix}
\]
This matrix has determinant $-1$ and signature $(1,1)$. We claim that all other summands $T_v$ are negative-definite even lattices, so that $T$ has signature $(1,1+\sum_{v\in R}(m_v-1))$.

\subsection{Dynkin diagrams}
\label{ss:Dynkin}

Recall that every component of a reducible fibre is a rational curve of self-intersection $-2$. We associate a graph to each fibre by drawing a vertex for each component and connecting two vertices by an edge for each intersection point of the corresponding components.
For instance, type $\I_n$ for $n>1$ gives a circle of $n$ vertices, and the same applies to $\III$ ($n=2$) and $\IV$ ($n=3$).

It is an elementary observation going through all fibre types that these graphs correspond exactly to the extended Dynkin diagrams $\tilde A_m, \tilde D_m, \tilde E_m$. 
One can recover the Dynkin diagrams $A_m, D_m, E_m$ by omitting the vertex corresponding to the zero component and the edges attached to it.
The following figure sketches the Dynkin diagrams.
The circles  indicate components intersecting the zero component  (as opposed to the balls).

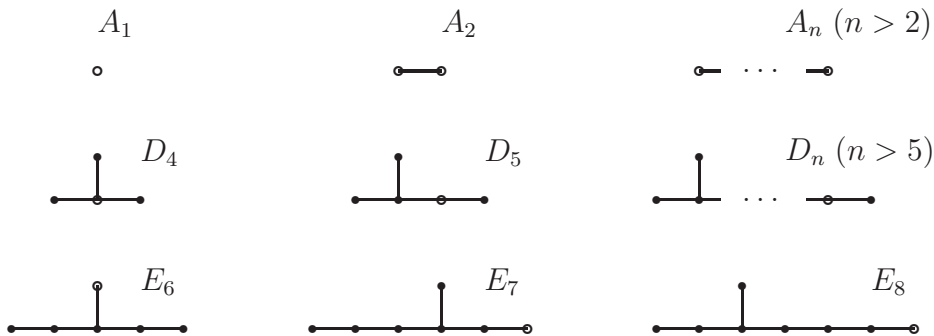
\begin{figure}[ht!]
\setlength{\unitlength}{.45in}
\begin{picture}(10,4.5)(0.4,-0.3)
\thicklines

  \put(1,3){\circle{.1}}
    \put(1,3.55){\makebox(0,0)[l]{${A}_1$}}

\put(4.5,3){\line(1,0){0.5}}
\multiput(4.5,3)(0.5,0){2}{\circle{.1}}
    \put(5,3.55){\makebox(0,0)[l]{${A}_2$}}

\put(8,3){\line(1,0){0.25}}
\multiput(8,3)(1.5,0){2}{\circle{.1}}
  \put(9.25,3){\line(1,0){0.25}}
    \put(9,3.55){\makebox(0,0)[l]{${A}_n\;(n>2)$}}
      \put(8.5,3){\makebox(0,0)[l]{$\hdots$}}

\put(0.5,1.5){\line(1,0){1}}
\multiput(0.5,1.5)(1,0){2}{\circle*{.1}}
\put(1,1.5){\line(0,1){0.5}}
  \put(1,2){\circle*{.1}}
  \put(1,1.5){\circle{.1}}
    \put(1.5,2.05){\makebox(0,0)[l]{${D}_4$}}

\put(4,1.5){\line(1,0){1.5}}
\multiput(4.5,1.5)(1,0){2}{\circle*{.1}}
\put(4.5,1.5){\line(0,1){0.5}}
  \put(4.5,2){\circle*{.1}}
    \put(4,1.5){\circle*{.1}}
  \put(5,1.5){\circle{.1}}
    \put(5.5,2.05){\makebox(0,0)[l]{${D}_5$}}

\put(7.5,1.5){\line(1,0){0.75}}
\multiput(7.5,1.5)(0.5,0){2}{\circle*{.1}}
\put(8,1.5){\line(0,1){0.5}}
  \put(8,2){\circle*{.1}}
  \put(9.5,1.5){\circle{.1}}
  \put(9.25,1.5){\line(1,0){0.75}}
    \put(10,1.5){\circle*{.1}}
    \put(9,2.05){\makebox(0,0)[l]{${D}_n\;(n>5)$}}
      \put(8.5,1.5){\makebox(0,0)[l]{$\hdots$}}

\put(0,0){\line(1,0){2}}
\multiput(0,0)(0.5,0){5}{\circle*{.1}}
\put(1,0){\line(0,1){0.5}}
  \put(1,0.5){\circle{.1}}
    \put(1.5,0.55){\makebox(0,0)[l]{${E}_6$}}

\put(3.5,0){\line(1,0){2.5}}
\multiput(3.5,0)(0.5,0){5}{\circle*{.1}}
\put(5,0){\line(0,1){0.5}}
  \put(5,0.5){\circle*{.1}}
  \put(6,0){\circle{.1}}
    \put(5.5,0.55){\makebox(0,0)[l]{${E}_7$}}

\put(7.5,0){\line(1,0){3}}
\multiput(7.5,0)(0.5,0){6}{\circle*{.1}}
\put(8.5,0){\line(0,1){0.5}}
  \put(8.5,0.5){\circle*{.1}}
  \put(10.5,0){\circle{.1}}
    \put(10,0.55){\makebox(0,0)[l]{${E}_8$}}
  

%

\end{picture}
\caption{Dynkin diagrams}
\label{Fig:Dynkin}
\end{figure}

\begin{Lemma}
\label{Lem:Dynkin}
 Each Dynkin diagram defines an even negative-definite lattice, denoted by the same symbol. 
  The determinants are:
\[
 \det(A_n) = (-1)^n\,(n+1),\;\;\; \det(D_m) = (-1)^m\,4\;\; (m\geq 4),
\]
\[
\det(E_6)=3,\;\;\;\det(E_7)=-2,\;\;\;\det(E_8)=1.
\]
\end{Lemma}

Note that in each case, the absolute value of the determinant equals the number of simple components of the singular fibre (including the zero component).
The multiplicities $n_i$ of the other components of the singular fibre (which one can compute with Tate's algorithm) are uniquely determined by the condition
\[
 F \approx \Theta_0 + \sum_{i=1}^{m_v-1} n_i\Theta_i,\;\;\;  F^2=F.\Theta_j=\left(\Theta_0 + \sum_{i=1}^{m_v-1} n_i\Theta_i\right)^2 =0.
\]
The lattices $A_n, D_n, E_n$ above are called the (negative) {\bf  root lattices} of type 
$A_n, \ldots$ (cf.~\ref{ss:root1}).
The positive-definite  lattices which are obtained from them 
by changing the sign of the intersection form  
are also called the (positive) {\bf  root lattices}, and they are often denoted by the same symbol (see e.g. \cite{CS}).  
It will be clear from the context which (negative or positive) is meant in the statements in this survey.
Most of the time (such as in this section), we will refer to the negative root lattices;
but specifically in section \ref{s:MWL}, we define Mordell-Weil lattices in such a way that they are positive-definite.

Definite lattices have been classified to some extent.
For instance, 
the root lattice $E_8$ with the two signs of the intersection form gives the unique positive-definite and negative-definite even unimodular lattice of rank 8. 
Such classifications will play an important role in two areas: the study of rational elliptic surfaces (cf.~section \ref{s:MWL}) and the classification of elliptic fibrations on a given K3 surface (cf.~\ref{ss:Nishi}).


\subsection{Canonical divisor}
\label{ss:can_bundle}

In order to prove that algebraic and numerical equivalence coincide on an elliptic surface $S$, it is useful to know the canonical divisor $K_S$.
The following formula goes back to Kodaira for complex elliptic surfaces. A characteristic-free proof was given in \cite{BM}.
The statement involves the Euler characteristic $\chi(S)=\chi(S,\OO_S)$ which we will investigate further in the next section.
In this survey, we shall refer to $\chi =\chi (S)$  as the \textbf{arithmetic genus} of $S$. 

\smallskip

\textbf{N.B.} 
 In the theory of algebraic surfaces, the invariant
 $p_a = \chi-1$ is usually called the arithmetic genus, but we find it more convenient 
 to adopt our definition in this survey.

\begin{Theorem}[Canonical bundle formula]
\label{Thm:can}
 The canonical bundle of an elliptic surface $S\stackrel{f}{\longrightarrow}C$ is given by
\[
 \omega_S = f^*(\omega_C\otimes\mathcal{L}^{-1})
\]
where $\mathcal L$ is a certain line bundle of degree $-\chi(S)$ on $C$. In particular, we have
\[
 K_S \approx (2g(C)-2+\chi(S))\,F,\;\;\; K_S^2=0.
\]
\end{Theorem}

The main idea of proof is that $K_S$ is vertical. This comes from the fact that $F.K_S=0$ by adjunction.
Then one applies Zariski's Lemma to show that $K_S$ is a fibre multiple.
Finally one computes the degree with the help of spectral sequences and the Riemann-Roch theorem.

\smallskip

As an application, we compute the self-intersection of any section by the adjunction formula:
\begin{Corollary}
\label{Cor:P^2}
For any $P\in E(K)$, we have $\bar P^2=-\chi(S)$.
\end{Corollary}

\subsection{Arithmetic genus and Euler number}
\label{ss:Euler}

We want to review a formula for the Euler number $e(S)$ of an elliptic surface $S$. Here one can think of the topological Euler number in case we are working over the complex numbers; in general we work with the alternating sum of the Betti numbers, the dimensions of the $\ell$-adic \'etale cohomology groups.

For the fibres of an elliptic surface, we thus obtain
\[
 e(F_v) = \begin{cases}
         0, & \text{ if $F_v$ is smooth};\\
m_v, & \text{ if $F_v$ is multiplicative};\\
m_v+1, & \text{ if $F_v$ is additive}.
        \end{cases}
\]
By (\ref{eq:number}), this local Euler number of the fibre agrees exactly with the vanishing order of the discriminant if there is no wild ramification. 
Otherwise it involves a contribution from the index of wild ramification $\delta_v$ at $v$.
Recall from \ref{ss:wild} that $\delta_v=0$ unless char$(k)=2,3$ and $F_v$ is additive.

\begin{Theorem}[{\cite[Prop.~5.16]{CD}}]
\label{Thm:Euler_number}
 For an elliptic surface $S$ over $C$, we have
\[
 e(S) = \sum_{v\in C} (e(F_v) + \delta_v).
\]
\end{Theorem}
Here the sum is actually finite, running over the singular fibres of $S\to C$, i.e.~$v\in \Sigma$.
By assumption, the elliptic surface has a singular fibre, hence $e(S)>0$.
As a corollary, we obtain the arithmetic genus $\chi(S)$ through Noether's formula
\[
 12\,\chi(S) = K_S^2 + e(S).
\]

\begin{Corollary}
\label{cor:ag}
 For an elliptic surface $S$, we have
\[
 \chi(S) = \frac 1{12}\, e(S) >0.
\]
\end{Corollary}

The proof of Thm.~\ref{Thm:a-n} now proceeds as follows. Assuming $D\equiv 0$, one shows that $K_S-D$ is vertical using Riemann-Roch with $\chi(S)>0$ and Serre duality. 
Then Lemma \ref{Lem:Dynkin} implies that  $K_S-D$ is algebraically equivalent to some fibre multiple.
By the canonical bundle formula, this also holds for $D$.
But then the degree has to be zero since $O.D=0$ by assumption.
Hence $D\approx 0$.

\smallskip

As we have seen, Thm.~\ref{Thm:a-n} implies Thm.~\ref{Thm:NS}.
It remains to prove Thm.~\ref{Thm:E-NS}.

\subsection{Sections vs.~horizontal components}

In order to prove Thm.~\ref{Thm:E-NS}, we shall exhibit the inverse of the map
\begin{eqnarray*}
E(K) & \to & \NS(S)/T\\
P & \mapsto & \bar P \mod T.
\end{eqnarray*}
For this purpose, it will be convenient to view the generic curve $E$ as a curve on $S$.
We start by defining a homomorphism
\[
\mbox{Div(S)} \to \mbox{Div(E)}
\]
as follows: Any divisor $D$ on $S$ decomposes into a horizontal part, consisting of sections and multisections, and a vertical divisor consisting of fibre components:
\[
D= D' + D'',\;\;\; D' \text{ horizontal},\;\; D'' \text{ vertical}.
\]
Then the horizontal part $D'$ and $E$ intersect properly, giving a divisor on $E$ of degree $D'.E$. This ($K$-rational) divisor is called the restriction of $D$ to $E$:
\[
D|_E := D' \cap E \in \mbox{Div}(E).
\]
In terms of \textbf{linear equivalence} $\sim$ (on $E$ resp.~$S$), it is easy to see that
\[
D|_E \sim_E 0 \Leftrightarrow D \sim_S D'' \text{ for some vertical } D''.
\]
By Abel's theorem for $E$ over $K$, the divisor $D$ thus determines a unique point $P\in E(K)$ by the following linear equivalence of degree zero divisors:
\[
D|_E - (D'.E)\, O \sim_E P-O.
\]
Writing $\psi(D)=P$, we obtain a homomorphism 
\[
\psi: \mbox{Div}(S) \to E(K).
\]
The kernel of $\psi$ can be seen to be
\[
\ker(\psi) = \langle D\in\mbox{Div}(S); D\approx 0\rangle + \Z \bar O + \langle D \in\mbox{Div}(S); D \text{ vertical}\rangle.
\]
Hence $\psi$ induces the claimed isomorphism
\[
\psi:\;\; \NS(S)/T \cong E(K).
\]
Recall that we introduced the terminology ``horizontal divisor'' for sections on an elliptic surface.
The above construction identifies any divisor $D$ (or its horizontal component $D'$) with the horizontal divisor $\bar P$.

All these considerations will be made more explicit in section \ref{s:MWL} in order to endow the Mordell-Weil group (modulo torsion) with a lattice structure.

\subsection{Picard variety}

A crucial step in the proof of Thm.~\ref{Thm:E-NS} concerns the Picard varieties of the elliptic surface $S$ and the base curve $C$.
Recall that generally on a projective variety $X$, the Picard variety Pic$^0(X)$ arises as the following quotient:
\[
\mbox{Pic}^0(X) =  \langle D\in\mbox{Div}(X); D\approx 0\rangle /  \langle D\in\mbox{Div}(X); D\sim 0\rangle.
\]
More specifically, if $X$ is a curve, then $D\approx 0 \Leftrightarrow \deg(D)=0$, and Pic$^0(X)$ is the Jacobian of $X$.

In the case of an elliptic surface $S\stackrel{f}{\to} C$, pull-back with $f^*$ defines an injection 
\[
f^*:\;\;\; \mbox{Pic}^0(C)\hookrightarrow \mbox{Pic}^0(S),
\]
since the section $\pi:\; C \to S$ provides a left inverse $\pi^*$. With the first bit of the above consideration, one can in fact show that $f^*$ is surjective:

\begin{Theorem}
\label{Thm:Pic-var}
For an elliptic surface $S$ over $C$ with section, Pic$^0(S)\cong$ Pic$^0(C)$.
\end{Theorem}

\subsection{Betti and Hodge numbers}

From Thm.~\ref{Thm:Pic-var} (or \cite[Cor.~5.2.2]{CD}), we deduce that an elliptic surface $S$ and its base curve $C$ have the same first Betti number:  
\[
b_1(S)= b_1(C).
\]
Then Poincar\'e duality allows us to express the second Betti number of $S$ through the Euler number:
\[
b_2(S) = e(S) - 2\,(1-b_1(C)).
\]
For a  complex elliptic surface $S$, we can describe the Hodge diamond explicitly.
We shall give its entries, the dimensions 
\[
h^{i,j}=\dim H^j(S, \Omega_S^i),
\]
in terms of the genus of the base curve $C$ and the arithmetic genus of $S$ (or equivalently the Euler number).
For this we abbreviate
\[
g=g(C) = q(S), \;\;\; \chi(S)=\chi,\;\;\; p_g=p_g(S) = h^{2,0}(S) = \chi-1+g.
\]
For the Euler number, we compare two expressions: $e(S)=12\chi$ and  the definition of $e(S)$ as alternating sum of Betti numbers to derive
\[
e(S) = 12\chi = 2 \chi - 2 g + h^{1,1}(S).
\]
Hence the Hodge diamond of a complex elliptic surface takes the shape
$$
\begin{array}{ccccc}
&& 1 &&\\
& g && g &\\
p_g && 10\chi+2g && p_g\\
& g && g &\\
&& 1 &&
\end{array}
$$

\subsection{Picard number}
\label{ss:Pic_number}

As a corollary of Thm.~\ref{Thm:E-NS}, we obtain a formula for the Picard number of an elliptic surface with section. Sometimes, this formula is referred to as Shioda-Tate formula.

\begin{Corollary}
\label{Cor:ST}
Let $S$ be an elliptic surface with section. Denote the generic fibre by $E$. Then
\[
\rho(S) = \mbox{rank } T + \mbox{rank } E(K) = 2 + \sum_{v\in R} (m_v-1) +  \mbox{rank } E(K).
\]
\end{Corollary}

Here we can also express the rank of the trivial lattice in terms of the Euler number. From Prop.~\ref{Prop:triv_lattice} and Thm.~\ref{Thm:Euler_number}, we deduce
\[
e(S) = \sum_{v\in \Sigma} e(F_v) = \sum_{v\in \Sigma} (m_v+\delta_v) + \#\{v\in \Sigma; F_v \text{ is additive}\}.
\]
Hence we obtain
\begin{eqnarray*}
\mbox{rank } T  & = &  e(S) - \#\{v\in \Sigma; F_v \text{ is multiplicative}\}\\
&& \;\;\;\;\;\;\;\;\;\;\;\; - 2\,\#\{v\in \Sigma; F_v \text{ is additive}\} - \sum_{v\in \Sigma} \delta_v.
\end{eqnarray*}

We conclude by recalling that over the complex numbers, Lefschetz' bound 
\[
\rho(S) \leq h^{1,1}(S)
\]
holds true, while in general, we only have Igusa's inequality
\[
\rho(S) \leq b_2(S).
\]
With Cor.~\ref{Cor:ST}, these inequalities translate into estimates for the number of fibre components that an elliptic surface of given Euler number may admit.

\subsection{Finiteness}
\label{ss:finite}

Combined with Thms \ref{Thm:isom-ss} and~\ref{Thm:PS},
the results from this section enable us to prove the following function field anlogue of Shafarevich's finiteness result over $\Q$ (Thm.~\ref{Thm:Shafa}):

\begin{Theorem}
\label{Thm:Shafa'}
Assume that the field $k$ is algebraically closed.
Let $C$ be a projective curve over $k$ and $\mathcal S$ a finite set of places on $C$.
Up to isomorphism, there are only finitely many elliptic surfaces $S\to C$ with section satisfying the following conditions:
\begin{enumerate}
 \item 
$S$ has good reduction, i.e.~smooth fibres outside $\mathcal S$;
\item
$S\to C$ is separable and not isotrivial: the j-map is not constant and separable;
\item
there is no wild ramification.
\end{enumerate}
\end{Theorem}

We shall only describe the proof in the semi-stable case.
Let $s=\#\mathcal S$.
Since $\Delta$ has support $\subset\mathcal S$, the conductor has degree $\deg(\mathcal N)\leq s$ by semi-stability.
Hence Thm.~\ref{Thm:PS} yields the bound
\[
 \deg(\Delta) = e(S) \leq 6 (s+2g(C)-2).
\]
This tells us how to bound the possible configurations of singular fibres -- either using the trivial bound from Thm.~\ref{Thm:Euler_number} (which is immanent in the above formula) or the improved bounds from the previous section depending on the characteristic.
At any rate, there are only finitely many possible configurations and thus also only finitely many possibilities for $\Delta$.
(The last fact relies on the specification of the set $\mathcal S$ containing all bad places.)
But by Thm.~\ref{Thm:isom-ss}, each admissible choice of $\Delta$ belongs to at most one semi-stable elliptic surface over $C$, if any.

For a conceptual partial approach in the spirit of Shafarevich's original proof that is based on the theory of Mordell-Weil lattices, see \ref{ss:finite'}.

\subsection{Example: Mordell-Weil group of the normal form}
\label{ss:MW-normal}

Let us return to the normal form $S$ for given j-invariant, considered as an elliptic surface over $\PP^1$ with coordinate $t$.
Since $S$ can be written as a cubic pencil (cf.~\ref{ss:cubic_pencil}), it follows that $S$ is rational.
At any rate, the singular fibres give $e(S)=12$ and thus $b_2(S)=10$ by the previous section. This gives an upper bound for the Picard number $\rho(S)$ which we will see immediately to be attained.

\smallskip

We start by investigating the trivial lattice.
By Cor.~\ref{Cor:P^2}, we have $\bar O^2=-\chi(S)=-1$.
Hence the sublattice of $T$ generated by $\bar O$ and $F$ is isomorphic to $\langle 1\rangle \oplus \langle -1\rangle$.
In consequence, the trivial lattice depends on the characteristic as follows:
\[
 T = \langle 1\rangle \oplus \langle -1\rangle \oplus
\begin{cases}
 E_7, & \text{ if char}(k)\neq 2,3;\\
E_8, & \text{ if char}(k)= 2,3.
\end{cases}
\]
If char$(k)=2,3$, then $T$ thus has rank ten and discriminant $-1$. Since $\NS(S)$ is an integral lattice, these conditions imply $\NS(S)=T$.

In case char$(k)\neq 2,3$, then $T$ has only rank nine and discriminant $2$. In addition, we find the following section on the integral model (\ref{eq:normal_int}) of $S$:
\[
 P = \left(- \frac 1{36}\, (t-12^3)^2,\;\, \dfrac{3+2\sqrt{2}}{216}\, (t-12^3)^3\right).
\]
This point is induced from one of the simple base point of the corresponding cubic pencil -- which has $x=-1/36$.
It is easily checked that the other simple base point induces $-P$.
Hence we claim that $E(K) = \langle P\rangle$. To prove this, we consider the sublattice of $\NS(S)$ generated by $\bar P$ and the trivial lattice.

\smallskip

By Thm.~\ref{Thm:a-n}, the study of sublattices of $\NS(S)$ basically amounts to the computation of intersection numbers.
In the affine chart (\ref{eq:normal_int}), $\bar P$ and $\bar O$ do not meet since $P$ is given by polynomials. 
On the other hand, one easily checks that also in the chart at $\infty$, the section is polynomial.
Indeed, on the integral model from (\ref{eq:normal-oo}), the section is given as
\[
 P = \left(- \frac 1{36}\, (1-12^3s)^2,\;\, \dfrac{3+2\sqrt{2}}{216}\, (1-12^3s)^3\right).
\]

Hence $\bar P.\bar O=0$.
Using Tate's algorithm, one finds that $P$ meets the non-zero simple component of the $\III^*$ fibre.
The resulting intersection matrix of size $10\times 10$ has determinant $-1$. 
As before, we deduce that 
\[
 \NS(S) = T + \Z\bar P.
\]
Thm.~\ref{Thm:E-NS} then gives the claim $E(K)=\langle P\rangle$.

\smallskip

The above example indicates that the Tate algorithm might actually be needed in order to find the fibre component met by a section.
Of course, this will always be a simple component, so for the present fibre type as well as $\I_2, \II,  \III, \II^*$, there is not much choice, but for the remaining fibre types the task becomes non-trivial.

\section{Torsion sections}
\label{s:torsion}

Obviously there is a substantial difference between torsion sections and so-called infinite sections.
On an elliptic surface, this can be detected lattice theoretically by the means of the previous chapter.
In the following we will exploit torsion sections in detail.
They admit a rich theory specifically due to the connection with universal elliptic curves,
but they will also be interesting to us for other reasons, for instance in the context of singular fibres or automorphisms of surfaces.

\subsection{}
\label{ss:torsion1}
In the previous example, we have seen an elliptic surface with a non-torsion section in characteristics other than $2, 3$.
We will also refer to non-torsion sections as infinite sections.
Here's how one can see that $P$ is not torsion: Otherwise, we would have $mP=O$ for some integer $m$.
But then $m\cdot \bar P\in T$ by Thm.~\ref{Thm:E-NS}.
In particular, adjoining $P$ to $T$ produces an overlattice of the same rank, i.e.~rank nine. 
Since the intersection matrix has rank ten, we obtain a contradiction.

\smallskip

Conversely, the existence of a  torsion section is equivalent to the property that $T$ does not embed primitively into $\NS(S)$. That is, the quotient $\NS(S)/T$ is not torsion-free. Denote by $T'$ the \textbf{primitive closure} of $T$ inside $\NS(S)$:
\[
 T' = (T\otimes\Q) \cap \NS(S).
\]
With this overlattice of $T$, we easily deduce the following statements:
\begin{Lemma}
\label{Lem:tor}
 Let $P\in E(K)$ be torsion. Then $\bar P\in T'$.
\end{Lemma}

\begin{Corollary}
\label{Cor:T-bar}
In the above notation, 
\[
E(K)_{\mbox{tors}}\cong T'/T.
\]
\end{Corollary}

\subsection{Group structure of singular fibres}

The structure of the N\'eron model induces a group structure on \emph{all} fibres, not only the smooth one.
On a singular fibre, the smooth points form an algebraic group scheme over the base curve.
Here the subgroup scheme of the identity component is
\[
 \mathbb{G}_m,\; \text{if $F_v$ is multiplicative};\;\;\; \mathbb{G}_a,\; \text{if $F_v$ is additive}.
\]
The quotient by this subgroup scheme is a finite abelian group which we denote by $G(F_v)$ depending on the fibre type.

\begin{Lemma}
\label{Lem:group_sing}
The singular fibres of elliptic surfaces admit the following group structure: 
$$
\begin{array}{rll}
\text{multiplicative} & \mathbb{G}_m \times G(F_v): & G(\I_n) \cong \Z/n\Z\\
\text{additive} & \mathbb{G}_a \times G(F_v): & G(\I_{2m}^*) \cong (\Z/2\Z)^2\\
&& G(\I_{2m+1}^*) \cong \Z/4\Z,\\
&& G(\II)\cong G(\II^*)\cong \{0\}\\
&& G(\III) \cong G(\III^*) \cong \Z/2\Z,\\
&& G(\IV)\cong G(\IV^*)\cong \Z/3\Z.
\end{array}
$$
\end{Lemma}

\subsection{Simple components vs.~sections}

By definition of the N\'eron model, the group structures of the generic fibre and the special fibres are compatible.
Here we only consider the simple fibre components met by a section. Since the simple components of a fibre exactly form the group $G(F_v)$, we derive the following property:

\begin{Lemma}
\label{Lem:sect-comp}
Consider the map $\psi: E(K) \to \prod_{v\in R} G(F_v)$, taking a section to the respective fibre components that it meets. Then $\psi$ is a group homomorphism. 
\end{Lemma}

\subsection{Simple components vs.~torsion sections}

Cor.~\ref{Cor:T-bar} indicates that a torsion section $P\neq O$ has to meet some fibre non-trivially.
By this, we mean that $P$ intersects a \textbf{non-zero component} (i.e.~different from the zero component) of some fibre.
Hence Lemma \ref{Lem:sect-comp} specialises as follows:
\begin{Corollary}
\label{Cor:tor-phi}
Restricted to the torsion subgroup of $E(K)$, the group homomorphism $\varphi$ is injective:
\[
 E(K)_{\text{tors}} \stackrel{\psi}{\hookrightarrow} \prod_{v\in R} G(F_v).
\]
\end{Corollary}

The corollary is best-known in characteristic zero.
A rigorous characteristic-free proof follows from the theory of Mordell-Weil lattices, see \ref{ss:appl}.

\begin{Remark}
\label{Rem:pec}
If the characteristic $p\geq 0$ does not divide the order of the torsion section $P$, then it follows from the theory of the N\'eron model \cite{Neron} that $\bar P$ and $\bar O$ are disjoint. 
Without the assumption, the statement is not valid
as supersingular elliptic curves in characteristic $p$ have no $p$-torsion. 
An example for an elliptic surface with peculiar $p$-torsion was given in~\cite[App.~2]{OS}.
\end{Remark}

\subsection{Narrow Mordell-Weil group}
\label{ss:narrow_MW}

The previous result yields a bound on the size of the torsion subgroup of $E(K)$. 
In particular, the order of a torsion section cannot exceed the least common multiple $m$ of the annihilators of the $G(F_v)$.

Generally, this tells us that upon multiplication by $m$ on $E$, no section meets the trivial lattice anymore. 
Hence all these multiples lie in the following subgroup of $E(K)$:

\begin{Definition}
The \textbf{narrow Mordell-Weil group} $E(K)^0$ consists of all those sections which meet the zero component of every fibre:
\[
 E(K)^0 = \ker(\psi) = \{P\in E(K); P \text{ meets every fibre at }\Theta_0\}.
\]
\end{Definition}

In terms of the narrow Mordell-Weil group, our previous observation can be rephrased as follows:
\[
 m E(K) \subseteq E(K)^0.
\]

\subsection{Sections vs.~automorphisms}

We will gain a much better understanding of torsion sections by the following interpretation:
Given a section $P\in E(K)$, we define an automorphism $t_P$ of the underlying surface $S$ by translation by $P$:
This is well-defined on the smooth fibres and can be extended to the singular fibres.
Note that translation by a section does not define an automorphism of the elliptic surface, since the zero-section is not preserved, but of the underlying surface where we forget the elliptic structure.
However, $t_P$ respects the elliptic fibration, i.e.~$f=f\circ t_P$.

\smallskip

Often, an elliptic fibration with an infinite section allows one to conclude that some projective surface has infinite automorphism group.
For instance, this holds true in characteristic $\neq 2,3$ 
for the normal form for given j-invariant
 by \ref{ss:MW-normal}.

\subsection{Quotient by torsion section}
\label{ss:quot-sect}

If $P$ is a torsion section, then $t_P$ defines an automorphism of finite order $m$ on the underlying algebraic surface $S$.
On the generic fibre, $t_P$ operates fixed point free.
This translates to all smooth fibres unless $p\mid m$, cf.~Remark \ref{Rem:pec}.
In any case, the quotient is an isogenous elliptic curve $E'$ over $k(C)$. 
If $k$ contains the $m$th roots of unity, then $E'$ is endowed with a $k$-rational $m$-torsion point
which induces the dual isogeny $E'\to E$.
The composition of the two isogenies is just multiplication by $m$ if $p\nmid m$, resp.~
by $p^{2e}m'$ if $m=p^em'$ with $p\nmid m'$.

The Kodaira-N\'eron model of $E'$ is another elliptic surface $S'$ over $C$ with section.
If $p\nmid m$, this surface is obtained from the quotient $S/\langle t_P\rangle$ by resolving the ordinary double point singularities resulting from the fixed points of $t_P$:
\[
 S' = \widetilde{S/\langle t_P\rangle}.
\]
Because of the isogeny $E\to E'$, the elliptic surfaces $S, S'$ are often also called isogenous.
Thanks to the dual isogeny, the surfaces share the same Betti numbers.
In fact, their Lefschetz numbers have to coincide as well;
in consequence $\rho(S)=\rho(S')$ (cf.~\cite{Inose} for a similar result for complex K3 surfaces).

\subsection{Relation between singular fibres}

The quotient $S/\langle t_P\rangle$ is best studied when the characteristic $p$ does not divide $m$, the order of the section $P$, since then the operation is separable.
One can study the action of $t_P$ on the singular fibres componentwise.
This is very convenient, since each component is a rational curve.
Since the quotient has to give rise to one of Kodaira's fibre types after the desingularisation, one can classify the possible actions on the singular fibres.

\subsubsection{Multiplicative fibres}

For simplicity, let us consider only the case where the order $m$ of $P$ is prime.
The general case can easily be deduced from this simplification.

If $P$ meets the zero component of an $\I_n$ fibre, then each component is fixed by $t_P$. Hence there are exactly $n$ fixed points at the intersections of the components. Each of them attains an $A_{m-1}$ singularity in the quotient. Hence the desingularisation results in a fibre of type $\I_{mn}$.

If $P$ does not meet the zero section, then $m\mid n$ by Cor.~\ref{Cor:tor-phi}. In particular, $t_P$ rotates the singular fibre (i.e.~the cycle of rational curves) by an angle of $2\pi/m$. Thus $t_P$ acts fixed-point free and  identifies $m$ components. The resulting fibre in the quotient surface $S'$ has type $\I_{n/m}$.

\subsubsection{Additive fibres}

If the torsion section $P$ were to meet the zero component of an additive fibre, then this would impose a second fixed point apart from the node (or cusp). It is easily checked that the resolution cannot result in one of Kodaira's types of singular fibres (under the assumption $p\nmid m$). Hence $P$ has to meet the singular fibre non-trivially. 

\begin{Lemma}
\label{Lem:torsion-add}
Let $E(K)'$ denote the prime to $p$-torsion in $E(K)$ where $p=$char$(K)$. Then any additive fibre $F_v$ gives an injection
\[
 E(K)' \hookrightarrow G(F_v).
\]
\end{Lemma}

The lemma is a direct consequence of Cor.~\ref{Cor:tor-phi}, since $\varphi$ is a group homomorphism.
In consequence, additive fibres can easily rule out torsion  (prime to $p$), like fibres of type $\II$ and $\II^*$ or combinations of other fibre types.
The condition $e(S)=e(S')$ imposes further restrictions on the singular fibres.
These restrictions apply particularly to the possible configurations of multiplicative fibres (cf.~section \ref{s:RES}).

\subsection{Universal elliptic curves}

In characteristic zero, torsion points give rise to universal elliptic curve.
Here the universal property refers to some level $N>3$ as follows:
$Y_1(N)$ is an elliptic curve over some function field $K$ over $\Q$ such that
\begin{itemize}
 \item $E(K)\cong \Z/N\Z$, and
\item conversely, any elliptic curve with  a rational $N$-torsion point is a member of the family.
\end{itemize}
The assumption $N>3$ guarantees that the universal elliptic curve $Y_1(N)$ forms a unique one-dimensional family over the modular curve $X_1(N)$.
By filling in suitable singular fibres through the Kodaira-N\'eron model, we can thus interpret $Y_1(N)$ as an elliptic surface with $N$-torsion in the Mordell-Weil group.
The universal property then guarantees that any other elliptic surface with an $N$-torsion section arises from $Y_1(N)$ through a base change.

These universal elliptic curves have important applications to arithmetic and in particular to modular forms.
Namely, one can relate all newforms of level $N$ and weight $>2$ to (desingularisations of) the self fibre products of $Y_1(N)$ by \cite{Deligne}.
In this arithmetic context, it becomes important that $Y_1(N)$ admits a model over $\Z[\frac 1N]$.

\subsection{Tate's normal form}
\label{ss:univ-few}

In order to explicitly parametrise universal elliptic curves, it is convenient to apply two normalisations. Starting from an elliptic curve in generalised Weierstrass form (\ref{eq:NF}) with section $P$, we first translate the section to $(0,0)$. 
Then $a_6\equiv 0$, and $P$ is two-torsion if and only if $a_3\equiv 0$. 
Otherwise a translation of $y$ yields $a_4\equiv 0$.
The resulting equation is Tate's normal form for a cubic with a rational point $P=(0,0)$:
\[
E:\;\; y^2 + a_1 \,x\,y+a_3\,y = x^3 + a_2\,x^2.
\]
As a reconfirmation, we check $-P=(0,-a_3)\neq P$, so $P$ is not 2-torsion.

\subsubsection{3-Torsion}
\label{ss:3-tor}

Since $-2P=(a_2,0)$, the point $P$ has order 3 if and only if $a_2\equiv 0$. 
But then there is still much freedom in choosing $a_1, a_3$ -- too much freedom actually in order to have a universal elliptic curve.
Moreover, there are two principal possibilities for the j-invariant: If $a_1\equiv 0$, then $j\equiv 0$, otherwise $j$ is not constant.
The following two curve can be regarded as minimal in some sense. (Namely, they are extremal, cf.~section \ref{s:RES-extr}.)
\begin{eqnarray*}
 \IV^*, \I_3, \I_1: & S: & y^2+xy+ty=x^3\\
\IV^*, \IV: & S': & y^2+t\,y=x^3
\end{eqnarray*}
Both surfaces are apparently rational since their function fields equal $k(x,y)$.
In either case, the trivial lattice $T$ has rank ten. Since $e(S)=e(S')=12$ and $\rho=10$, we can deduce from Cor.~\ref{Cor:ST} that the Mordell-Weil rank is zero. Since $T$ has discriminant $-9$, there can only be 3-torsion in either case:
\[
 E(K) = E'(K) = \{O, P, -P\}.
\]

\subsubsection{Higher torsion}
\label{ss:univ-tor}

If $a_2\not\equiv 0$, then we can rescale $x$ and $y$ in such a way that $a_2=a_3$. Writing
\[
 b=-a_2=-a_3,\;\;\; c=1-a_1,
\]
we can express the coefficients of $nP$ through $b,c$.
For the first few multiples, we obtain
$$
\begin{array}{rcccrcccrcc}
 P & = & (0,0) && 2P & = &  (b, bc) && 3P & = & (c,b-c)\\
-P & = & (0,b) && -2P & = & (b,0) && -3P & = & (c,c^2)
\end{array}
$$
From the coefficients, we obtain affine equations in $b,c$ for the modular curves parametrising elliptic curves with a torsion point.

For instance, for \textbf{4-torsion}, we obtain $c=0$. The resulting elliptic surface over the affine $b$-line has reducible singular fibres $\I_4$ at $0$ and $\I_1^*$ at $\infty$.

Similarly, \textbf{5-torsion} implies $b=c$. In consequence, there are two fibres of type $\I_5$ at $0$ and $\infty$.

Subsequently, \textbf{6-torsion} is parametrised by the conic $b=c+c^2$ while \textbf{7-torsion} yields a nodal cubic. The corresponding elliptic surface of level 7 is the first non-rational surface -- it is a K3 surface.

\smallskip

The above elliptic surfaces are well-defined in any characteristic, but the fibres degenerate in characteristics $p\mid N$. In fact, there the j-map becomes inseparable (cf.~\cite[\S 12]{KM}).

\subsection{2-Torsion and quadratic twisting}
\label{ss:univ-2}

The problems with 2-torsion are two-fold. On the one hand, we experience the same phenomenon with different j-invariants -- constant vs.~non-constant -- as in the 3-torsion case.
But there is also a similar structural difference related to quadratic twisting.

\smallskip

We work with fields of characteristic other than $2$. Hence we can reduce to an extended Weierstrass form (\ref{eq:ext_WF}).
Here the $2$-torsion points, if any, are given by the zeroes $\alpha_i$ of the right-hand side polynomial in $x$ -- with $y$-coordinate zero, since the lines through $O$ are exactly the vertical lines (and the line at $\infty$).

Now consider the quadratic twist by some $d\in K$ as in (\ref{eq:quadr_twist}).
It is immediate that the resulting elliptic curve has $2$-torsion points $(d\,\alpha_i,0)$.
In other words, quadratic twists preserve the $2$-torsion.
Thus the question for a universal elliptic curve with a $2$-torsion section or full $2$-torsion in the Mordell-Weil group is ill-posed, since we can always twist at arbitrary places.

\smallskip

Note that in general, quadratic twists of elliptic surface introduce $\I_0^*$ fibres.
Hence it makes sense to ask for elliptic surface with $\Z/2$- or $(\Z/2)^2$-torsion sections and the minimal number of $\I_0^*$ fibres.
All these surfaces are rational by inspection of their function fields.

For $(\Z/2)^2$, this leads to the so-called Legendre family on the one hand
\[
 S:\;\;\; y^2 = x\,(x-1)\,(x-t)
\]
with singular fibres $\I_2$ at $t=0,1$ and $\I_2^*$ at $\infty$. 

\smallskip

On the other hand, we can fix any elliptic curve $E$ over $k$. Since $k$ is algebraically closed, $E$ has full 2-torsion over $k$. Then consider the quadratic twist by $t$. This yields an elliptic surface over $\PP^1$ with coordinate $t$. It has two singular fibres of type $\I_0^*$ at the two places of ramification, $0$ and $\infty$, and constant j-invariant $j=j(E)$. In particular, these elliptic surface form a one-dimensional family parametrised by $j$.

For a single 2-torsion section, we derive two surfaces without $\I_0^*$ fibres:
\begin{eqnarray*}
S: & y^2 = x^3 + x^2 + t x & \III^*, \I_2, \I_1\\
S': & y^2 = x^3 + t\,x & \III^*, \III
\end{eqnarray*}
Both admit the section $(0,0)$. As before, one finds that $E(K)=E'(K)=\{O, (0,0)\}$. Note that $E'$ has j-invariant $j=12^3$.
As for $S$, we can still twist the generic fibre quadratically at two of the singular places without introducing any $\I_0^*$ fibres.

\subsection{}

We take this opportunity to note that with respect to non-2-torsion, quadratic twisting has exactly the opposite effect: no sections are preserved. This is related to quadratic base change as follows:

Let $S\to C$ be an elliptic surface and $C'\to C$ a double cover with set of ramification points $\mathcal S$.
Then the Mordell-Weil group of the base change $S'$ is generated (up to finite index, a $2$-power) by the independent Mordell-Weil groups of $S$ and its quadratic twist $\hat S$ over $\mathcal S$:
\[
2 E(k(C')) \hookrightarrow E(k(C)) \oplus \hat E(k(C)) \hookrightarrow E(k(C')).
\]
This decomposition holds true up to two-torsion
because the two given Mordell-Weil groups correspond exactly to the invariant and anti-invariant sections under the deck transformation.

\subsection{Elliptic modular surfaces}
\label{ss:EMS}

Over $\C$ there is direct construction of the elliptic surfaces associated to the universal elliptic curves.
This was introduced by the second author in \cite{ShEMS}.

In brief, let $\Gamma\subset SL_2(\Z)$ denote a subgroup of finite index.
If $-1\not\in\Gamma$, then one can associate a complex elliptic surface $S(\Gamma)$ to $\Gamma$ in a canonical way.
For the congruence subgroups $\Gamma_1(N)$, we recover the elliptic surface associated to $Y_1(N)$.

The common theme is a generalisation of the relation with modular forms of weight 3. Namely, there is an intrinsic isomorphism
\[
 H^{2,0}(S(\Gamma)) \cong S_3(\Gamma)
\]
between holomorphic 2-forms on $S(\Gamma)$ and cusp forms of weight 3 for $\Gamma$.
However, if $\Gamma$ is not a congruence subgroup, then there is no general way to translate this isomorphism into arithmetic properties, say using the $L$-function of $S(\Gamma)$ over some number field.
Instead, there are conjectural Atkin--Swinnerton-Dyer congruence relations, but these have only been proven in very special cases so far.

We remark that this construction was generalised by M.~Nori in \cite{Nori} to generalised modular elliptic surfaces.
Over $\C$ these have always finite Mordell-Weil group, thus yielding most extremal elliptic surfaces (see \ref{ss:extremal}).

\section{Rational elliptic surfaces}
\label{s:RES}

This section initiates our study of rational elliptic surfaces.
The previous sections allow us to develop a fairly precise picture.
Another aspect will be added by the introduction of Mordell-Weil lattices, see section \ref{s:MWL}.

\subsection{L\"uroth's theorem}

Let $S\stackrel{f}{\longrightarrow}C$ be an elliptic surface over $k$. The fibration induces an inclusion of function fields
\[
 k(C)\subseteq k(S).
\]
If $S$ is rational, L\"uroth's theorem implies that $C\cong \PP^1$.
Thus any rational elliptic surface is fibred over the projective line.
Alternatively, we could have argued with Thm.~\ref{Thm:Pic-var}.
Before we continue the investigation of rational elliptic surfaces, we derive a general fact for elliptic surfaces over $\PP^1$.

\subsection{Elliptic surfaces over the projective line cont'd}
\label{ss:IP^1'}

As a consequence of the Tate algorithm, we have seen in \ref{ss:IP^1} that every elliptic surface with a section over $\PP^1$ has a globally minimal generalised Weierstrass form (\ref{eq:NF}). Equivalently, there is $n\in\N$ such that 
\begin{itemize}
 \item $ \deg(a_i)\leq ni$ for all i, but
\item
there is some $i$ such that $\deg(a_i)\geq (n-1)i$ (minimality at $\infty$) and 
\item 
for any finite place on $\PP^1$ with valuation $v$, there is some $i$ such that $v(a_i)<i$ (minimality at finite places).
\end{itemize}
Since the indices refer to some weight, the discriminant of a globally minimal elliptic surface $S$ has degree 
\[
 \deg(\Delta)=12n.
\]
On the other hand, we have at any place
\[
 v(\Delta) = e(F_v) + \delta_v
\]
by definition of the wild ramification in \ref{ss:wild}. 
Hence Thm.~\ref{Thm:Euler_number} gives
\[
 \deg(\Delta) = e(S) = 12\chi(S).
\]
Thus we deduce the factor $n$ of the weights of the coefficients $a_i$ of the generalised Weierstrass form:
\[
  n=\chi(S).
\]

\subsection{Canonical bundle}
\label{ss:RES-can}

For an elliptic surface $S$ with section over $\PP^1$, the canonical bundle formula in \ref{ss:can_bundle} returns
\[
 K_S = (\chi(S)-2)\,F
\]
in $\NS(S)$. 
A rational surface has $K_S$ anti-ample. 
Recall that $\chi(S)>0$ since we assumed the existence of a singular fibre.
Hence any rational elliptic surface has 
\[
 \chi(S) = 1,\;\;\; e(S)=12.
\]

\subsection{Normal form of a rational elliptic surface}

Combining \ref{ss:IP^1'} and \ref{ss:RES-can}, we deduce that a rational elliptic surface $S$ with section is given by a generalised Weierstrass form (\ref{eq:NF})
\[
 S:\;\;\; y^2 + a_1 \,x\,y+ a_3\,y \, = \, x^3 + a_2 \,x^2 + a_4 \,x+a_6.
\]
with $\deg(a_i)\leq i$. 
Conversely, such a generalised Weierstrass form defines a rational elliptic surface if and only if it has a singular fibre.
That is, no admissible transformation makes each $a_i$ into an $i$th power.
In particular, this holds if the discriminant $\Delta$ is not a twelfth power.

\smallskip

The last characterisation is an equivalence if char$(k)\neq 2,3$ because then there is no wild ramification.
Since there would be only one singular fibre $F_\infty$, we would have $\deg(\Delta)=e(F_\infty)$.
Hence $F_\infty$ would have at least 11 components.
But then the trivial lattice would have rank$(T)\geq 12>\rho(S)=10$, contradiction.

In characteristics $2, 3$, however, there are (rational) elliptic surfaces with exactly one singular fibre.
Hence it is necessary to check the minimality at the twelve-fold zero of $\Delta$.

\subsection{Cubic pencils}

Our first examples for (rational) elliptic surfaces were cubic pencils. 
This characterisation is in fact complete:

\begin{Proposition}[{\cite[Thm.~5.6.1]{CD}}]
Any rational elliptic surface $S$ is isomorphic to the blow-up of $\PP^2$ in the base points of a cubic pencil.
\end{Proposition}

The proof is based on the consequence of the adjuction formula that any smooth rational curve $D$ on $S$ has self-intersection $D^2\geq -2$.
Blowing down $(-1)$-curves successively to a minimal model $X$ (in the sense of algebraic surfaces), 
one can show that $S$ also admits a morphism $\pi$ to $\PP^2$.
But then $\pi(K_S)=K_{\PP^2}\sim -3H$ for a hyperplane section $H$.
Since $F=-K_S$, the images of the fibres are plane cubic curves and generally smooth. (``The pencil descends from $S$ to $\PP^2$.'')

\subsection{Frame}

Even lattices have been classified to great extent. Here we investigate whether $\NS(S)$ is even for an elliptic surface -- or whether it admits meaningful even sublattices.

\smallskip

Recall the trivial lattice $T$ as a sublattice of $\NS(S)$. This lattice may vary greatly with the elliptic surface, but it always contains the  rank two lattice generated by zero section and general fibre: 
\[
 T\supseteq \langle \bar O, F\rangle = \begin{pmatrix}
                                        - \chi & 1\\ 1 & 0
                                       \end{pmatrix}
\]
Note that this sublattice $L$ is \textbf{even}, i.e.~$D^2\in 2\Z$ for all $D\in L$, if and only if $\chi(S)$ is even:
\[
 \langle \bar O, F\rangle = \begin{pmatrix}
                                        - \chi & 1\\ 1 & 0
                                       \end{pmatrix}
\cong \begin{cases}
       \langle 1\rangle \oplus \langle -1\rangle, & \text{ if $\chi(S)$ is odd};\\
U = \begin{pmatrix}
     0 & 1\\1 & 0
    \end{pmatrix}, &
\text{ if $\chi(S)$ is even}.
      \end{cases}
\]
The even lattice $U$ is often called the \textbf{hyperbolic plane}.
Thus $\NS(S)$ may not be even for an elliptic surface. However, $\NS(S)$ always contains an even sublattice of corank two:
the \textbf{frame} 
\[
 W(S)= \langle \bar O, F\rangle^\bot \subset \NS(S).
\]

\begin{Lemma}
\label{Lem:O,F_purp}
For any elliptic surface $S$ with section, the frame $W(S)$
is a negative-definite even lattice of rank $\rho(S)-2$. 
\end{Lemma}

\begin{proof}
By the algebraic index theorem, $\NS(S)$ has signature $(1,\rho(S)-1)$. But  $\langle \bar O, F\rangle$ is visibly indefinite, so its orthogonal complement has signature $(0,\rho(S)-2)$.

Letting $D\in  \langle \bar O, F\rangle^\bot$, we have $D.F=D.K_S=0$. The adjunction formula yields $D^2=2g(D)-2$.
\end{proof}

\subsection{}
\label{ss:ortho-frame}

Let $\varphi$ denote the orthogonal projection in $\NS(S)$ with respect to $U$
\[
 \varphi: \NS(S) \to W(S).
\]
Then the image of a section $P$ under $\varphi$ is
\[
 \varphi(\bar P) = \bar P - \bar O - (\chi(S) + \bar P.\bar O)\,F.
\]
In particular this element has self-intersection
\begin{eqnarray}\label{eq:height-ortho}
 \varphi(\bar P)^2 = - (2\,\chi(S) + 2 \bar P.\bar O).
\end{eqnarray}
This orthogonal projection should be compared to the definition of the Mordell-Weil lattices in \ref{ss:MWL-ortho};
the latter will be based on the orthogonal projection with respect to $T(S)$ in $\NS(S)\otimes\Q$.

\smallskip

As soon as $\chi(S)>1$, the relation (\ref{eq:height-ortho}) allows us to recover (the root lattices of) the singular fibres from the root lattice of the frame:
\[
 W(S)_{\mbox{root}} = \{x\in W(S);\; x^2=-2\}
\]

\begin{Lemma}
\label{Lem:frame-root}
Assume that $\chi(S)>1$. Then $\sum_{v\in R} T_v = W(S)_{\mbox{root}}$.
\end{Lemma}
In particular this shows that if $\chi(S)>1$, then the trivial lattice can be computed as 
\[
 T(S) = U  \oplus  W(S)_{\mbox{root}}.
\]
On a rational elliptic surface (i.e.~$\chi(S)=1$), the frame behaves completely differently as we will see in the next section.

\subsection{N\'eron-Severi lattice of a rational elliptic surface}
\label{ss:MWL_RES}

Since a rational elliptic surface $S$ is isomorphic to $\PP^2$ blown up in nine points, $\NS(S)$ has rank ten and discriminant $-1$. The latter property is called \textbf{unimodular}. Inside $\NS(S)$, we have the indefinite lattice generated by zero section and general fibre and its orthogonal complement, the frame $W(S)= \langle \bar O, F\rangle^\bot$. By definition, both sublattices are unimodular.
Moreover, $W(S)$ is even and negative definite by Lemma \ref{Lem:O,F_purp}.
It follows from lattice theory that there is only one such lattice up to isometry.
This lattice corresponds to the Dynkin diagram $E_8$.
Hence there is an isomorphism
\[
W(S) = \langle \bar O, F\rangle^\bot \cong E_8.
\]
We point out the special case where $\bar O$ and $F$ already generate the trivial lattice:
\begin{Corollary}
Let $S$ be a rational elliptic surface without reducible fibres. Then the orthogonal projection in $\NS(S)$ with respect to $\bar O, F$ induces an isomorphism
\[
 E(K) \cong E_8
\]
\end{Corollary}
For a general cubic pencil, the corollary tells us that the nine (disctinct) base points induce eight independent sections after choosing one of them as the zero section.

In the general situation where all fibres are irreducible, one can thus endow $E(K)$ with a lattice structure. 
Up to torsion, this approach can be pursued in full generality, 
leading to the notion of Mordell-Weil lattices which we introduce in section \ref{s:MWL}.

\section{Extremal rational elliptic surfaces}
\label{s:RES-extr}

Before continuing our investigation of rational elliptic surfaces, we turn to an important special class among these surfaces in detail -- extremal elliptic surfaces.
Their relevance lies in several different aspects such as classification questions and arithmetic properties.

\subsection{Maximal Picard number}

We first introduce a more general notion that governs extremality of elliptic surfaces.
Recall the bounds for the Picard number by Lefschetz and Igusa from \ref{ss:Pic_number}
\[
\rho(S) \leq h^{1,1}(S)\;\;\; \text{ resp. } \;\;\;
\rho(S) \leq b_2(S).
\]
If the respective bound is attained, we refer to the Picard number as \textbf{maximal}.
For elliptic surfaces,
one can see that both bounds are attained infinitely often by (purely inseprable in characteristic $p$) base change from rational elliptic surfaces.
For rational elliptic surfaces, both bounds coincide and are always attained.

\subsection{Extremal elliptic surfaces}
\label{ss:extremal}

Elliptic surfaces with maximal Picard number can be distinguished according to the contribution from the Mordell-Weil group. 
A  special case occurs when the trivial lattice generates the N\'eron-Severi lattice up to finite index.

\begin{Definition}
An elliptic surface $S$ with section is called \textbf{extremal} if $\rho(S)$ is maximal and $E(K)$ is finite.
\end{Definition}

One should think of extremal elliptic surfaces as very isolated and rare. 
Over the complex numbers, this holds with the exception of a few isotrivial families (cf.~\ref{ss:isotrivial}), but generally speaking there are quite a few examples. 
In postive characteristic, however, it follows from Thm.~\ref{Thm:PS} that essentially all extremal elliptic surfaces arise through inseparable base change.

\smallskip

By Thm.~\ref{Thm:E-NS}, the index of the trivial lattice $T(S)$ inside $\NS(S)$ is exactly the size of the Mordell-Weil group of the extremal elliptic surface $S$:
\begin{eqnarray}\label{eq:T-index}
 \mbox{disc}(T(S)) = (\# E(K))^2 \, \mbox{disc}(\NS(S)).
\end{eqnarray}

Often extremal elliptic surfaces are related to subgroups of $\mbox{SL}_2(\Z)$.
We have seen instances of this in Section \ref{s:torsion}, see in particular \ref{ss:EMS}.

\subsection{Extremal rational elliptic surfaces}

A rational elliptic surface $S$ has unimodular N\'eron-Severi group. If $S$ is extremal, (\ref{eq:T-index}) thus simplifies as
\[
 \mbox{disc}(T(S)) = -(\# E(K))^2
\]
Since the discriminants of the full trivial lattice $T(S)$ and the sublattice of the frame $W(S)$ which is the orthogonal complement of the unimodular sublattice generated by $\bar O, F$ agree up to sign, we obtain the following information about the singular fibres:
\begin{eqnarray}\label{eq:RES-extr}
\prod_{v\in R} \mbox{disc}(T_v) = (\# E(K))^2.
\end{eqnarray}
Here we employ the convention that the irreducible fibres have discriminant $1$.

\smallskip

On the other hand, we can compare the Euler number from Thm.~\ref{Thm:Euler_number}
\[
 12 = e(S) = \sum_v (e(F_v)+\delta_v)
\]
against the rank of the trivial lattice from Prop.~\ref{Prop:triv_lattice}:
\[
 10 = \mbox{rank}(T(S)) = 2 + \sum_v (m_v-1).
\]
Subtracting yields:
\begin{eqnarray}\label{eq:e(F)}
 4 = \sum_v \begin{cases}
            0, & \text{ if $F_v$ is smooth};\\
1, & \text{ if $F_v$ is multiplicative};\\
2+\delta_v, & \text{ if $F_v$ is additive}.
           \end{cases}
\end{eqnarray}
In particular, we deduce that there are at most four singular fibres.
If there is no wild ramification (e.g.~if char$(k)\neq 2,3)$, there are in fact exactly three cases:
\begin{enumerate}
 \item 4 multiplicative singular fibres;
\item 2 multiplicative and 1 additive singular fibres;
\item
2 additive singular fibres.
\end{enumerate}
For either case, we have seen examples in \ref{ss:3-tor} -- \ref{ss:univ-2}.
In the presence of wild ramification, all  combinations of fewer singular fibres (involving at least one additive fibre) occur as well.

\subsection{Classification}

One classifies extremal rational elliptic surfaces in terms of the configurations of singular fibres which we write as tuples 
\[
[n_1, n_2,\hdots]\;\;
\text{ with entries }\;\;
1,2,\hdots, 0^*, 1^*, \dots, \II, \III, \hdots, \II^*
\]
representing the fibre types.

\subsubsection{Existence}

Relation (\ref{eq:e(F)}) yields a finite list of possible configurations of singular fibres.
One can rule out many of these configurations already by (\ref{eq:RES-extr}).
Additionally we only need the injection from Lemma \ref{Lem:torsion-add}.
To prove the existence of the remaining configurations, one has to find explicit equations.
This can be accomplished with Weierstrass equations and base change \cite{MP1}.
It seems that the first semi-stable example was exhibited by the Hesse pencil in \cite{ShEMS}, realising the elliptic modular surface for $\Gamma(3)$.
The first complete list of semi-stable surfaces is due to Beauville using cubic pencils \cite{B}.
(For the close connection between extremal elliptic surfaces and the abc-conjecture, see \ref{ss:dessin} and \cite{Sh-abc}.)


\subsubsection{Mordell-Weil group}

In each case, the configuration determines the Mordell-Weil group.
One can use quotients by (translation by) the torsion sections to limit the possible orders.
It turns out that there is always only one abelian group of order given by the singular fibres respecting these bounds.
(This is not valid  in general; for instance, it fails for extremal elliptic K3 surfaces -- and so does the uniqueness.)

\subsubsection{Uniqueness}

To prove uniqueness, one can often argue with universal elliptic curves.
Here we obviously have to exclude the one-dimensional family with configuration $[0^*, 0^*]$.
A helpful fact is also that the quotient by a torsion section is an isogenous surface with a section of the same order.
On the other hand, upon deriving explicit Weierstrass equations, uniqueness up to M\"obius transformations and admissible transformation of $x$ and $y$ is always more or less immediate.

\begin{Theorem}[Miranda--Persson {\cite{MP1}}]
For any possible configuration except for $[0^*, 0^*]$, there exists a unique extremal rational elliptic surface with pre-determined Mordell-Weil group.
\end{Theorem}

In the sequel, we give the possible configurations without wild ramification and elaborate on a few of the details of the proof of the theorem.
The classification in characteristic $2$ or $3$ involving wild ramification
has been established by W.~Lang in \cite{Lang-I}, \cite{Lang-II}, but unfortunately there are some typos in the equations, so the interested reader is advised to compare with \cite{Ito} for instance.

\subsection{Case (1)}

In the multiplicative case, the criterion in (\ref{eq:RES-extr}) suffices to rule out all non-existing configurations. 
The next table lists the existing configurations.
We also give the Mordell-Weil group and describe the degeneration of the singular fibres in special characteristics (cf.~\ref{ss:univ-tor}).

$$
\begin{array}{cccc}
 \text{configuration} & MW & \mbox{char}(k) & \text{degeneration}\\
\hline
[1,1,1,9] & \Z/3 & \neq 3 & [\II,9]\\

[1,1,2,8] & \Z/4 & \neq 2 & [\III,8]\\

[1,2,3,6] & \Z/6 & \neq \begin{cases} 2\\3
                         \end{cases} &
\begin{matrix}[2,\IV,6]\\
 [\III,3,6]
\end{matrix}\\

[1,1,5,5] & \Z/5 & \neq 5 & [\II,5,5]\\

[2,2,4,4] & \Z/4\times \Z/2 & \neq 2 & --\\

[3,3,3,3] & (\Z/3)^2 & \neq 3 & --
\end{array}
$$

\subsection{Case (2)}

In the mixed case of additive and multiplicative fibres, Lemma~\ref{Lem:torsion-add} becomes important.
For instance, that lemma directly shows that the exceptional entries in the previous table are impossible outside the given characteristic.
Hence we shall not reproduce them in the following list.
For shortness, it only includes those configuration that exist in characteristic zero.
By reduction, the same holds outside chararacteristics $2,3$.
In the remaining characteristics, the existence of wild ramification complicates the classification.

$$
\begin{array}{cc}
 \text{configuration} & MW\\
\hline
[1,1,\II^*] & \{0\}\\

[1,2,\III^*] & \Z/2\\

[1,3,\IV^*] & \Z/3
\end{array}
\;\;\;\;\;\;
\begin{array}{cc}
 \text{configuration} & MW\\
\hline
[1,1,4^*] & \Z/2\\

[1,4,1^*] & \Z/4\\

[2,2,2^*] & (\Z/2)^2
\end{array}
$$

\subsection{Case (3)}

The additive configurations are treated exactly along the same lines. Outside characteristics $2,3$, we obtain

$$\begin{array}{cc}
 \text{configuration} & MW\\
\hline
[\II,\II^*] & \{0\}\\

[\III,\III^*] & \Z/2
\end{array}
\;\;\;\;\;\;
\begin{array}{cc}
 \text{configuration} & MW\\
\hline
[\IV,\IV^*] & \Z/3\\

[0^*, 0^*] & (\Z/2)^2
\end{array}
$$

\subsection{Isotrivial elliptic surfaces}
\label{ss:isotrivial}

All the rational elliptic surfaces in case (3) have constant j-invariant. 
Such fibrations are called \textbf{isotrivial}.
The name has its origin in the property that after a suitable finite base change, the surface becomes isomorphic to a product of the new base curve and the (constant) general fibre.

Among elliptic fibrations, the isotrivial ones play a special role.
We have seen in the discussion of universal elliptic curves that isotrivial fibrations can come in families even if it is unexpected.

In characteristic zero,
extremal elliptic surfaces over $\PP^1$ 
are known to have discrete moduli unless $e\leq 60$.
The proof combines the infinitesimal Torelli theorem for elliptic surfaces with non-constant j-invariant by M.-H.~Sait\=o \cite{Saito} with the classification of all extremal elliptic surfaces with constant j-invariant by Kloosterman \cite{Kl}.
The exceptional Euler numbers are due to families of isotrivial elliptic surfaces with j-invariant $0$ or $12^3$.

\section{Semi-stable elliptic surfaces}
\label{s:s-s}

We have classified all extremal rational elliptic surfaces in terms of the possible configurations of singular fibres.
Now we would like to know all possible configurations in general.
Letting go the condition extremal, there will clearly be no uniqueness anymore.

We will develop part of the theory in greater generality.
For simplicity, we shall only consider semi-stable elliptic surfaces (i.e.~without additive fibres).
In case of additive fibres, the arguments have to be modified slightly.
Moreover, we will restrict ourselves to the ground field $k=\C$. 
This will enable us to use monodromy arguments and deformations.
In positive characteristic, there are further issues with inseparable base change and wild ramification that need to be addressed,
but the results do not differ essentially.

We note that the first investigation of this kind seems to be due to U.~Schmickler-Hirezebruch
who classified elliptic surfaces over $\PP^1$ with section and at most three singular fibres \cite{S-H}.
Such a configuration forces some fibre to be additive, so these elliptic surfaces will not appear in this section
as we restrict to the semi-stable case.

\subsection{J-map of a semi-stable elliptic surface}

For a semi-stable elliptic surface $S$ over $C$, the j-map
\[
 j:\;\; C\to \PP^1
\]
is a morphism of degree $e(S)=12\chi(S)$.
By definition, $j$ is branched above $0$ and $12^3$ with ramification orders divisible by $3$ and $2$, respectively.
For simplicity, we will often apply a normalisation so that the third ramification point is $1$ instead of $12^3$.
According to \ref{ss:base_change} this guarantees that in the pull-back of the normal form surface via $j$, the additive fibres $\II$ and $\III^*$ are replaced by
$4\chi$ resp.~$6\chi$ fibres of type $\I_0^*$.
Then these can be eliminated by quadratic twisting to recover the semi-stable fibration $S\to C$.
Instead of the fibre $\I_1$ at $\infty$, $S$ acquires semi-stable fibres of types $[n_1,\hdots, n_s]$
where the $n_i$ denote the ramification indices at $\infty$.

If $C=\PP^1$, then we can phrase this through the following factorisation of the discriminant:
\[
 \Delta = \prod_{i=1}^s\; (t-\alpha_i)^{n_i}\;\;\; \text{ with pairwise distinct } \alpha_i.
\]

\subsection{Monodromy representation}
\label{ss:mono}

Let $\pi: \PP^1\to \PP^1$ be a map of degree $n$ with ramification points $0, 1, \infty$ (possibly among others). It is a general fact that any such $\pi$ is determined by its monodromy representation
\[
 (\sigma_0, \sigma_1, \sigma_\infty, \tau_1,\hdots,\tau_r).
\]
Here $\sigma_i, \tau_j\in S_n$ encode the monodromy around the branch points such that
\begin{itemize}
 \item 
$\sigma_0\circ\sigma_1\circ\sigma_\infty\circ\tau_1\circ\cdots\circ\tau_r=\mbox{id}$, and
\item
the $\sigma_i, \tau_j$ generate a transitive subgroup of $S_n$.
\end{itemize}

In the case of a semi-stable elliptic surface over $\PP^1$ with arithmetic genus $\chi$, we can specify the $\sigma_i$ as follows:
\begin{itemize}
 \item 
$\sigma_0$ is the product of $4\chi$ 3-cycles,
\item
$\sigma_1$ is the product of $6\chi$ 2-cycles, and
\item
$\sigma_\infty$ is the product of $s$ cycles $\gamma_i$ of lengths $n_i$.
\end{itemize}

\subsection{Deformation of j-map}
\label{ss:deformation}

In the sequel, we shall say that a tuple $[n_1,\hdots,n_s]$ exists if there is an elliptic surface over $\PP^1$ with exactly this configuration of singular fibres.
Over $\C$, we have the following deformation argument at hand:

\begin{Lemma}
\label{Lem:ex_conf}
 Let $[n_1,\hdots,n_s]$ exist. Let $i\leq s$ and $a,b\geq 1$ such that $a+b=n_i$. Then $[n_1,\hdots,n_{i-1},a,b,n_{i+1}\hdots,n_s]$ exists.
\end{Lemma}

Assume that $[n_1,\hdots,n_s]$ exists. 
The proof of the lemma builds on the fact that
\[
 ((1,\hdots,a)\,(a+1,\hdots,n))\,(a,n) = (1,\hdots,n).
\]
Hence we can write the $n_i$-cycle $\gamma_i$ as the product of two disjoint cycles $\gamma', \gamma''$ of lengths $a$ resp.~$b$ times a 2-cycle exchanging the ``end points'' of $\gamma'$ and $\gamma''$.
This extra 2-cycle accounts for an additional branch point.

It is easily checked that the resulting monodromy representation has product equalling the identity and generates a transitive group.
By \ref{ss:mono} there is a map $j: \PP^1\to \PP^1$ with the claimed ramification indices at $\infty$.


\subsection{Semi-stable rational elliptic surfaces}

Lemma \ref{Lem:ex_conf} is very useful to prove the existence of elliptic surfaces with given configuration of singular fibres. 
Starting with the six semi-stable extremal rational elliptic surfaces, we obtain through the deformation lemma:

\begin{Corollary}
\label{Cor:5-tuples}
For rational elliptic surfaces, the following 5-tuples exist
$$
\begin{array}{ccccc}
 [1,1,1,1,8], & [1,1,1,2,7], & [1,1,1,3,6], & [1,1,2,2,6], & [1,1,1,4,5],\\

[1,1,2,3,5], & [1,1,2,4,4], & [1,2,2,3,4], & [2,2,2,2,4], & [1,2,3,3,3].
\end{array}
$$
\end{Corollary}

Note that the tuple need not determine the torsion subgroup of the Mordell-Weil group. For instance, the first 5-tuple is realised by two different one-dimensional families: one has no torsion sections, while the other family arises through quadratic base change from the extremal rational elliptic surface with configuration $[1,1,4^*]$. In particular, it admits a 2-torsion section.

\begin{Corollary}
Every $s$-tuple $[n_1,\hdots,n_s]$ with $s\geq6$ and $\sum n_i=12$ exists on a rational elliptic surface.
\end{Corollary}

\subsection{}
It should be emphasised that the deformation construction is not algebraic in nature. Hence it does not provide any information about fields of definition or characteristic $p$.
For the extremal elliptic surfaces, it is immediately clear that they are defined over some number fields.
Compare Belyi's theorem which is actually the converse of a characterisation of algebraic curves over $\bar \Q$ due to Weil.

\subsection{Families}

It follows from the deformation argument that semi-stable rational elliptic surfaces with $s$ singular fibres occur in $(s-4)$-dimensional families. This statement does not generalise directly to elliptic surfaces with additive fibres, but it can be modified to work there as well as for arbitrary semi-stable elliptic surfaces.

In consequence, rational elliptic surfaces with only $\I_1$ fibres have 8 moduli. This expected number agrees with the heuristic based on the Weierstrass from (\ref{eq:WF}): the coefficients $a_4, a_6$ have $5+7=12$ coefficients. We have to subtract four degrees of freedom for applying M\"obius transformations and rescaling $x, y$ by an admissible transformation (\ref{eq:scale}).

\subsection{The remaining 5-tuples}

There are three 5-tuples of natural numbers adding up to $12$ (thus possibly realised by rational elliptic surfaces) that are not covered by Cor.~\ref{Cor:5-tuples}.
We claim that they are missing for a good reason:

\begin{Proposition}
\label{Prop:missing}
The 5-tuples $[1,2,2,2,5], [1,1,3,3,4], [2,2,2,3,3]$ do not exist.
\end{Proposition}

The proof of the non-existence goes like this:
First we prove that the above configurations do not allow torsion in the Mordell-Weil group.
To see this, consider the quotient by translation by a hypothetical torsion section.
By \ref{ss:quot-sect}, this results in an isogenous rational elliptic surface $S'$.
In particular, we obtain in the analogous notation (with $R=R'$)
\[
 12=e(S') = \sum_{v\in R} e(F_v').
\]
But in the present situation, one easily checks that any torsion section would meet some fibres in the zero component.
On the quotient surface, these fibres of type $\I_n$ result in ``bigger'' fibres $\I_{np}$ -- causing $e(S')$ to exceed $12$.

Since the Picard number is ten and the trivial lattice has rank nine, it follows from Thm.~\ref{Thm:E-NS} that 
\[
 E(K) = \{P\} \;\;\; \text{ and }\;\;\; \NS(S) = \langle \bar P, \bar O, F, \Theta_{v,i}\rangle.
\]
Here we have not used yet that $\NS(S)$ is unimodular. One could proceed with a case-by-case analysis, going though (enough of) the possible intersection behaviours of $P$ with $O$ and the singular fibres to conclude (by computing the intersection matrices) that it is an impossible task to set up $P$ in such a way that $\NS(S)$ has discriminant $-1$. However, this would be quite a lengthy and detailed undertaking.
Instead, we turn to the theory of Mordell-Weil lattices which endows the Mordell-Weil group (up to torsion) with a lattice structure.
This will enable us to read off the answer directly.
In \cite{OS} Oguiso and Shioda proved more generally that no rational elliptic surface admits the corresponding trivial lattices.

Yet another line of proof would involve the discriminant group, but we will only introduce this technique in the context of K3 surfaces where it becomes indispensable.

\section{Mordell-Weil lattice}
\label{s:MWL}

In this section we finally introduce the notion of Mordell-Weil lattices.
It goes back independently to Elkies \cite{Elki} and one of us \cite{ShMW}.
In practice, the idea of Mordell-Weil lattices makes the relations between the N\'eron-Severi group and the Mordell-Weil group from Section \ref{s:NS} much more explicit.
As a major advantage, the extent of this theory is very feasible to actual computations.

There are many important consequences to be noted.
First we can complete the discussion of rational elliptic surfaces and classify their Mordell-Weil lattices completely (following \cite{OS}).
We will also point out the functorial behaviour of Mordell-Weil lattices such as for base change and under Galois action.

\subsection{Essential lattice}

We aim at endowing the Mordell-Weil group (up to torsion) with the structure of a positive-definite lattice.
Recall that the sections are in some sense complimentary to the trivial lattice $T(S)$ by Thm.~\ref{Thm:E-NS}.
Thus the relevant part of $\NS(S)$ sits inside the orthogonal complement of $T(S)$.
This is how we define the \textbf{essential lattice}:
\[
 L(S) = T(S)^\bot \subset \NS(S).
\]
Since $L(S)$ is included in the frame of $S$ by definition, we obtain
\begin{Lemma}
 The essential lattice $L(S)$ is even and negative definite of rank $\rho(S)-2-\sum_v (m_v-1)$.
\end{Lemma}

By lattice theory, we can compute the discriminant of the essential lattice as
\[
 |\mbox{disc}(L(S))| = \left|\dfrac{\mbox{disc}(\NS(S))}{\mbox{disc}(T(S))}\right| [\NS(S):(L(S) \oplus T(S))]^2.
\]
Hence it is not necessarily true that $L(S) \oplus T(S)=\NS(S)$ (although this will hold true if $T(S)$ is unimodular, e.g.~if there are no reducible fibres). Thus it is unclear a priori whether there is a non-trivial group homomorphism $E(K)\to L(S)$.

\subsection{}

To circumvent subtleties of finite index embeddings as sketched above, we shall now tensor all lattices in consideration with $\Q$. 
Generally we shall abreviate this by a subscript $\Q$: 
\[
 \NS(S)_\Q = \NS(S) \otimes_\Z \Q,\;\; T(S)_\Q, L(S)_\Q, \hdots.
\]
Tensoring makes the lattices into $\Q$ vector spaces -- which does not do any harm since $\NS(S)$ is torsion-free by Thm.~\ref{Thm:NS}.
Note that the subgroup of torsion sections embeds into $T(S)_\Q$ by Lemma \ref{Lem:tor}.

\subsection{}
\label{ss:MWL-ortho}
Define the vector space homomorphism
\[
 \varphi:\;\; \NS(S)_\Q \to L(S)_\Q
\]
as the orthogonal projection with respect to the sub-vector space $T(S)_\Q$. 
(Compare the analogous construction involving $U$ and the frame $W(S)$ in \ref{ss:ortho-frame}.)
For an element $D\in\NS(S)$, the image $\varphi(D)$ is uniquely determined by the universal properties
\[
 \varphi(D)\,\bot\, T(S)_\Q,\;\;\; \varphi(D)\equiv D\mod T(S)_\Q.
\]
\begin{Lemma}
 The orthogonal projection $\varphi$ gives the unique map $E(K)\rightarrow L(S)_\Q$ with the above properties.
\end{Lemma}
The uniqueness follows from the universal property since (upon tensoring $E(K)$ by $\Q$) it identifies two vectorspaces of the same dimension.

\subsection{}
We can describe the map $\varphi$ explicitly.
For this purpose, we let $A_v$ denote the intersection matrix of the non-zero components of the (reducible) fibre $F_v$ at $v\in C$:
\[
 A_v = (\Theta_{v,i}, \Theta_{v,j})_{1\leq i,j\leq m_v-1}.
\]
Then a section $P\in E(K)$ is mapped to
\[
 \varphi(P) = \bar P - \bar O - (\bar P.\bar O + \chi(S))\,F - \sum_v (\Theta_{v,1},\hdots,\Theta_{v,m_v-1}) A_v^{-1} \begin{pmatrix}
                                                                                                                       \bar P.\Theta_{v,1}\\
\vdots\\
\bar P.\Theta_{v,m_v-1}
                                                                                                                      \end{pmatrix}
\]
Obviously $\varphi(P)\equiv \bar P\mod T(S)_\Q$. 
Therefore the precise shape of $\varphi(P)$ is verified by checking that $\varphi$ sits in the orthogonal complement of $T(S)_\Q$.

The inverse matrix $A_v^{-1}$ is in general not integral, but the coefficients are made integral by multiplying with disc$(T_v)$.
This shows that in general the definition of $\varphi$ requires tensoring with $\Q$ unless there are no reducible fibres.
Since the narrow Mordell-Weil group $E(K)^0$ does not meet any non-zero fibre components by definition, we deduce that its image under $\varphi$ is in fact integral:
\begin{Corollary}
 The orthogonal projection $\varphi$ defines an embedding $E(K)^0\hookrightarrow L(S)$.
\end{Corollary}

\subsection{}
We now consider the full Mordell-Weil group.
The following statement can easily be derived from the universal property of the orthogonal projection and from Abel's theorem:

\begin{Proposition}
 The map $\varphi:\; E(K)\to L(S)_\Q$ is a group homomorphism with kernel $\ker(\varphi)=E(K)_\text{tor}$.
\end{Proposition}

We have seen in \ref{ss:narrow_MW} that for $m=\prod_v\mbox{disc}(T_v)$
\[
 mE(K)\subseteq E(K)^0.
\]
Hence the image of $E(K)$ under $\varphi$ lies in fact in $\frac 1m L(S)$, and we obtain an embedding
\[
 E(K)/E(K)_\text{tor} \hookrightarrow \frac 1m L(S).
\]

\subsection{Height pairing}
Using the orthogonal projection $\varphi$, we define a pairing on $E(K)$, the so-called \textbf{height pairing}:
\[
 P, Q\in E(K)\;\; \Rightarrow\;\;\; <P,Q> = - \varphi(P).\varphi(Q).
\]
\begin{Theorem}
The height pairing is a symmetric bilinear pairing on $E(K)$.
It induces the structure of a positive-definite lattice on $E(K)/E(K)_\text{tor}$.
\end{Theorem}

Symmetry and bilinearity are clear. The rest follows from the fact that $L(S)$ is negative-definite, since 
\[
 <P,P>=0\;\; \Leftrightarrow \;\; \varphi(P)=0 \;\; \Leftrightarrow \;\; P\in E(K)_\text{tor}.
\]

\subsection{Mordell-Weil lattice}
We call the positive-definite lattice $E(K)/E(K)_{\text{tor}}$ with height pairing $<\cdot,\cdot>$ the \textbf{Mordell-Weil lattice} of the elliptic surface $S$. For ease of notation, we shall also denote the Mordell-Weil lattice by $\MWL(S)$ (where we suppress the elliptic fibration $S\to C$).

Note that the Mordell-Weil lattice will in general not be integral, but take values in $\Q$ (or sufficiently $\frac 1m \Z$ for the integer $m$ from \ref{ss:narrow_MW}).
In contrast, the \textbf{narrow Mordell-Weil lattice} $\MWL(S)^0$ is indeed always integral and even.

\subsection{Explicit formula}
\label{ss:height}

We shall now derive an explicit formula for the height pairing of two sections $P, Q\in E(K)$. Upon specialising to $P=Q$, we also obtain an expression for the \textbf{height} $h(P)$ of a section $P$:
\begin{eqnarray*}
 & <P,Q> = & \chi(S)+\bar P.\bar O+\bar Q.\bar O-\bar P.\bar Q - \sum_v \mbox{contr}_v(P,Q)\\
h(P) = & <P,P> = & 2\chi(S)+2 \bar P.\bar O - \sum_v \mbox{contr}_v(P).
\end{eqnarray*}
Here the correction terms contr$_v(P,Q)$ only depend on the fibre components of $F_v$ met. 
Specifically, let $P$ meet $\Theta_{v,i}$ and $Q$ meet $\Theta_{v,j}$ (with the convention that $O$ intersects $\Theta_{v,0}$ as usual). Then the local correction term at $v$ is given by
\[
 \mbox{contr}_v(P,Q) = 
\begin{cases}
 0, & \text{ if } ij=0\\
- (A_v^{-1})_{i,j}, & \text{ if } ij\neq 0.
\end{cases}
\]
The correction term for a single section $P$ is defined by contr$_v(P)=$contr$_v(P,P)$.
The inverse matrices $A_v^{-1}$ can be calculated in generality. 
We give their relevant entries (i.e.~those corresponding to the simple components) in the following table. This also verifies that the correction terms are always positive.
Throughout we employ the following conventions:
\begin{itemize}
 \item $0<i\leq j$ after exchanging $P$ and $Q$ if necessary;
\item
The components of an $\I_n$ fibre ($n>1$) are numbered cyclically according to their intersection behaviour $\Theta_0, \Theta_1,\hdots,\Theta_{n-1}$;
\item
At the additive fibres, we only number the simple components;
\item 
The simple components of an $\I_n^*$ fibre for $n>0$ are distinguished into the \textbf{near component} $\Theta_1$ which intersects the same double component as the zero component, and the \textbf{far components} $\Theta_2, \Theta_3$. 
\end{itemize}

\begin{table}[ht!]
$$
\begin{array}{cccccccc}
\text{fibre} & \;\; & \IV^* && \III^* && \I_n (n>1), \III, \IV & \I_n^*\\
\hline
\begin{matrix}
\text{Dynkin}\\
\text{diagram}
\end{matrix}
&& E_6 && E_7 && A_{n-1} & D_{n+4}\\
\hline
i=j && 4/3 && 3/2 && i(n-i)/n & \begin{cases}
                              1, & i=1\\
1+n/4, & i=2,3
                             \end{cases}\\
i<j && 2/3 && - && i(n-j)/n & 
\begin{cases}
                              1/2, & i=1\\
1/2+n/4, & i=2
                             \end{cases}
\end{array}
$$
\caption{Correction terms for the height pairing}
\end{table}
As claimed, we notice that the correction terms are always positive and become integral upon multiplication with the number of simple components of the fibre.

\subsection{First application: proof of Corollary \ref{Cor:tor-phi}}
\label{ss:appl}

With the above explicit formulas,
it is easy to prove Corollary \ref{Cor:tor-phi}.
Let $P, Q$ be torsion sections that meet every (singular) fibre at the same component.
In consequence
\[
\mbox{contr}_v(P) = \mbox{contr}_v(Q) = \mbox{contr}_v(P,Q) \;\;\; \forall\, v.
\]
Thus the torsion conditions $h(P)=h(Q)=\langle P,Q\rangle=0$ imply
\[
2\chi+2\bar P.\bar O = 2\chi+2\bar Q.\bar O=\chi+\bar P.\bar O+\bar Q.\bar O-\bar P.\bar Q.
\]
Whence $\bar P.\bar O=\bar Q.\bar O$ and $\bar P.\bar Q=-\chi$.
As $\chi>0$ by Corollary \ref{cor:ag} and $\bar P, \bar Q$ are irreducible curves on $S$, we deduce $\bar P=\bar Q$, i.e.~$P=Q$.

\subsection{}
Since the Mordell-Weil lattice is defined via the orthogonal projection in $\NS(S)_\Q$ with respect to $T(S)_\Q$, we can compute the discriminant of $\NS(S)$ as follows:
\begin{eqnarray}\label{eq:disc-NS-E}
\;\;\;\;\;\;\;\;\;\;\;\; \mbox{disc}(\NS(S)) = (-1)^{\text{rank} E(K)} \mbox{disc}(T(S))\cdot \mbox{disc}(\MWL(S))/(\# E(K)_\text{tor})^2.
\end{eqnarray}
We can apply this relation to give a simple proof of Prop.~\ref{Prop:missing}.

\smallskip

Recall that for those 5-tuples, the trivial lattice would have rank 9, so $E(K)$ would have rank 1.
Moreover, the configuration of singular fibres ruled out torsion sections.
Hence $E(K)$ would be generated by a single section $P$.
The above equality simplifies as
\begin{eqnarray}\label{eq:T-P}
  \mbox{disc}(\NS(S)) = -\mbox{disc}(T(S))\cdot h(P).
\end{eqnarray}
But then the denominator of the correction term contr$_v(P)$ does not suffice to produce an equality in (\ref{eq:disc-NS-E}): $\NS(S)$ is unimodular, but there are too many fibres with number of components divisible by $2$ or $3$. These fibres cause $T(S)$ to have highly divisible discriminant; hence, by inspection of the corresponding correction terms, the denominator of the height $h(P)$ cannot possibly be big enough to yield discriminant $-1$ in (\ref{eq:T-P}).

\subsection{}

Conversely, the above argument shows how a fibre configuration can force the existence of a section.
We illustrate this by considering the first few 5-tuples from Cor.~\ref{Cor:5-tuples}.

\subsubsection{[1,1,1,1,8]}
\label{sss:8}

In this case, there are actually two possibilities:
\begin{itemize}
 \item $E(K)$ torsion-free with generator $P$ disjoint from $O$, but meeting $\Theta_3$ of the $\I_8$ fibre: $h(P)=2-3\cdot 5/8=1/8$.
\item
$E(K)$ with 2-torsion and a section $P$ meeting $\Theta_2$, but not $O$: $h(P)=2-2\cdot 6/8=1/2$. In this case, the 2-isogenous rational elliptic surface has configuration [2,2,2,2,4].
\end{itemize}

\subsubsection{[1,1,1,2,7]}

The configuration rules out torsion sections, so $E(K)$ is generated by a single section $P$, disjoint from $O$ and intersecting the singular fibres $\I_2$ at $\Theta_1$ and $\I_7$ at $\Theta_2$: $h(P)=2-1/2-10/7=1/14$.

\subsubsection{[1,1,1,3,6]}

Since any section $P$ has height $P\in \frac 16\Z$, $\NS(S)$ can only be unimodular if there is a 3-torsion section.

\subsubsection{[1,1,2,2,6]}

Again any section $P$ has height $h(P)\in\frac 16\Z$. Hence there has to be a 2-torsion section.

\subsubsection{[1,1,2,4,4]}

With this configuration, any section has height $h(P)\in\frac 14\Z$. Hence unimodularity requires $4\mid \# E(K)_\text{tor}$ by (\ref{eq:disc-NS-E}).
However, there cannot be full 2-torsion in $E(K)$ since then there would be a 2-torsion section meeting an $\I_4$ fibre in the zero component.
Hence the quotient by the translation would have singular fibres $\I_2, \I_2, \I_8$ (from $\I_1,\I_1,\I_4$) plus two more, resulting in an Euler number exceeding 12.
We deduce $\Z/4\Z\subset E(K)$.

\smallskip

With the given tools, we can classify the Mordell-Weil lattices of all rational elliptic surfaces of Mordell-Weil rank zero or one, since they have at most one non-torsion generator. 
The full classification was established by the second author jointly with Oguiso in \cite{OS} -- we will briefly comment on the essential ingredients in the next few subsections.

For general elliptic surfaces, we can also employ the above techniques. 
But then we have also to take the transcendental part of $H^2$ into account, since $\NS(S)$ will in general not span $H^2$.
For the K3 case, see the discussion in \ref{ss:L1}, \ref{ss:Kondo}.

\subsection{}

We can relate the Mordell-Weil lattice $\MWL(S)$ and the essential lattice $L(S)$ more closely. 
For this purpose, we change the sign of the intersection form on $L(S)$ in order to make the lattice positive-definite;
the lattice with new pairing is abbreviated by $L(S)^-$.

\smallskip

In the following, we shall frequently need the dual lattice $M^\vee$  of a non-degenerate integral lattice $M$:
\[
 M^\vee = \{x\in M_\Q; (x.y) \in \Z\;\; \forall\; y\in M\}.
\]
By integrality, one clearly has $M\subseteq M^\vee$ as a sublattice of finite index $|\det(M)|$, with equality if and only if $M$ is unimodular.
Using the setup of Mordell-Weil lattices, one can show that the orthogonal projection $\varphi$ induces a commutative diagram
$$
\begin{array}{ccc}
\MWL(S)^0 & \cong & L(S)^-\\
\cap && \cap\\
\MWL(S) & \hookrightarrow & (L(S)^-)^\vee
\end{array}
$$
The bottom row embedding need not be an isomorphism; however, it is an isomorphism if $\NS(S)$ is unimodular. 
This relation proves very useful in the case of rational elliptic surfaces.

\subsection{Root lattice}
\label{ss:root1}

For an even integral negative-definite lattice $M$, the \textbf{roots} are all those elements $x\in M$ with maximal self-intersection $x^2=-2$.
The lattice $M$ is called a \textbf{root lattice} if it is generated by its roots.
For instance, the Dynkin diagrams $A_n, D_m, E_6, E_7, E_8$ encode root lattices.

To each fibre $F_v$, we can thus associate a root lattice $T_v$ spanned by the components not meeting the zero section.
For an elliptic surface $S$, we define the root lattice $R(S)$ of $S$ as the sublattice of $\NS(S)$ generated by the root lattices $T_v$:
\[
 R(S) = \oplus_v T_v.
\]
By definition, the root lattice $R(S)$ is exactly the intersection of the trivial lattice $T(S)$ and the frame $W(S)$.

\subsection{Mordell-Weil lattices of rational elliptic surfaces}

Recall that a rational elliptic surface has unimodular N\'eron-Severi group
\[
 \NS(S) \cong \langle \bar O,F\rangle \oplus E_8.
\]
Since the root lattice $R(S)$ sits imside the frame $W(S)$, we obtain an embedding of $R(S)$ into $E_8$.
From what we have seen before (in particular from $\MWL(S)\cong (L(S)^-)^\vee$), one can deduce the following relation between root lattice $R(S)$ and Mordell-Weil lattice $\MWL(S)$:
\begin{Theorem}[Shioda {\cite[Thm.~10.3]{ShMW}}]
Up to lattice automorphisms, the embedding
\[
 R(S)\hookrightarrow E_8
\]
determines both $E(K)$ as an abstract group and $\MWL(S)$ as a lattice.
\end{Theorem}

\subsection{}
The previous theorem creates the possibility to classify the Mordell-Weil lattices of rational elliptic surfaces by purely lattice-theoretic means. We sum up the relevant results.

\begin{Theorem}[Dynkin]
\label{Thm:Dynkin-root}
There are exactly 70 root lattices $R\neq 0, E_8$ embedding into $E_8$.
\end{Theorem}
Here the embedding $R\hookrightarrow E_8$ is unique up to lattice automorphisms unless
\[
 R= A_7, A_3^2, A_5 \oplus A_1, A_3 \oplus A_1^2, A_1^4
\]
For each ambiguous root lattice, there are two different families of (generally semi-stable) rational elliptic surfaces associated: one with torsion sections and one without.
We have already seen this for $R=A_7$ in \ref{sss:8}.

On the other hand, there are three root lattices embedded into $E_8$ as in Thm.~\ref{Thm:Dynkin-root} which are not compatible with the rational elliptic surfaces requirement that
\[
 e(R(S)) \leq e(S) = 12.
\]
These root lattices are $A_1^4 \oplus D_4, A_1^8, A_1^7$. In consequence, there are $(70+2+5-3)=74$ possible embeddings of root lattices $R(S)$ of rational elliptic surfaces $S$ into $E_8$ up to lattice automorphism. In each case, \cite{OS} determines the exact shape of $\MWL(S), \MWL(S)^0$ and $E(K)_\text{tor}$. 

\subsection{Corollaries}

Without details, we note the following corollaries of the classification of Mordell-Weil lattices of rational elliptic surfaces in \cite{OS}.

\begin{Corollary}
\label{Cor:MW-RES}
 For a rational elliptic surface, $E(K)$ takes the following shapes as an abstract group:
$$
\begin{array}{cc}
 E(K)_\text{tor} & r=\mbox{rank } E(K)\\
\hline
\{0\} & 0\leq r\leq 8\\
\Z/2\Z & 0\leq r\leq 4\\
\Z/3\Z & 0 \leq r\leq 2\\
(\Z/2\Z)^2 & 0\leq r\leq 2\\
\Z/4\Z & 0\leq r\leq 1\\
\Z/5\Z & 0\\
\Z/6\Z & 0\\
\Z/4\Z\times\Z/2\Z & 0\\
(\Z/3\Z)^2 & 0
\end{array}
$$
\end{Corollary}
Note that for all torsion groups (except for $(\Z/3\Z)^2$ where one can use the Hesse pencil as in \cite{ShEMS}) we have already seen examples in \ref{ss:3-tor} -- \ref{ss:univ-2} -- often even with the right number of parameters $r=$rank $E(K)$.

\begin{Corollary}
\label{Cor:11.8}
For a rational elliptic surface, $E(K)=\langle P; h(P)\leq 2\rangle$.
\end{Corollary}

Even with the inequality $h(P)\leq 2$ replaced by $h(P)\leq 2\chi(S)$, the above generation result does not hold true in general (cf.~\cite[Appendix 1]{OS}).

Actually, there is a stronger result than Corollary \ref{Cor:11.8}
that follows from the classification in \cite{OS} (cf.~\cite[Thm 3.4]{Sh-int}):


\begin{Proposition}
 For a rational elliptic surface, 
\[
E(K)=\langle P; \bar P.\bar O=0\rangle.
\]
\end{Proposition}

\subsection{Integral sections}
\label{ss:int-surf}

The set in the preceeding two corollaries are always finite: For the points with bounded height, this follows from the general property of definite lattices of finite rank.
For the second set, we only note that the sections disjoint from $O$ have bounded height $h\leq 2\chi(S)$.
It is this principle that allows us to generalise Siegel's theorem for elliptic curves over $\Z$.
Here we only have to define the right notion of integrality.

\smallskip

For an elliptic curve in Weierstrass form, the rational points other than $O$ are given by two rational functions 
\[
 P=(X,Y),\;\;\; X,Y\in k(C).
\]
Here we can always make the Weierstrass form integral, i.e.~after some transformations the coefficients $a_i$ live in the coordinate ring $k[C]$.
Then any notion of an integral section should certainly require $X,Y\in k[C]$.
We call such a section \textbf{polynomial}.
Note that 
\[
 X,Y\not\in k[C] \Rightarrow \bar P.\bar O>0,
\]
since $P$ and $O$ meet at the poles of $X, Y$ -- which in fact have to be common poles since the Weierstrass form is integral.
In other words, generally any section $P$ on an integral elliptic curve over $k(C)$ admits polynomials $X_0, Y_0, Z\in k[C]$ such that
\[
 P=(X_0/Z^2, Y_0/Z^3).
\]
However, this only reflects integrality in the chosen affine model. 
So $P$ could still have poles in the other affine charts -- or attain poles by minimalising.
Therefore it is convenient to write the coefficients $X, Y$ as homogeneous functions of degree $2\chi$ resp.~$3\chi$ over $C$.
On a globally minimal integral model, it thus becomes immediate that an affine polynomial section has a pole at $\infty$ if and only if 
\[
\deg(X)>2\chi, \;\;\deg(Y)>3\chi.
\]
Thus we call the section $P$ \textbf{integral} if and only if it is polynomial and respects the degree bounds after minimalising.
Equivalently, $\bar P.\bar O=0$.
Since integral sections have bounded height $h\leq 2\chi$ as we have seen, we deduce the analogue of Siegel's theorem:

\begin{Theorem}
Any integral elliptic curve admits only finitely many integral sections.
\end{Theorem}

Note that we can even allow the sections to have denominators of bounded degree.
More precisely, if the bound for the $x$-coordinate is $2d$ and for the $y$-coordinate $3d$, then there still is a height inequality $h\leq 2\chi+2d$.
Hence there are only finitely many such sections. 

\smallskip

Integral points are sometimes also called everywhere integral to stress integrality in all affine charts.
They reveal interesting relations to Gr\"obner bases and deformations of singularities.
The situation on rational elliptic surfaces was recently studied in \cite{Sh-int}.

For polynomial sections, finiteness does only hold over fields of characteristic zero (cf.~\cite[\S III.12]{Si0}).
In contrast, positive characteristic sees finiteness results for polynomial sections without further assumptions failing.
For an example in every characteristic $p>3$ (where polynomial sections are in fact called integral), see for instance \cite[Prop.~9.1]{SS}.

\subsection{Finiteness revisited}
\label{ss:finite'}

With the above finiteness results for the Mordell-Weil group, we can revisit Shafarevich's finiteness problem for elliptic curves.
Recall the set-up of Thm.~\ref{Thm:Shafa} and \ref{Thm:Shafa'} that predicted good reduction outside a finite set of places $\mathcal S$.
Over $\Q$ (or rather $\Z$ after switching to an integral model), Shafarevich's line of proof was to consider all possible discriminants $\Delta$ -- there are only finitely many.
Working with an integral Weierstrass forms
\[
S:\;\;\; y^2 = x^3 - 3A\,x-2B,\;\;\; A,B\in\Z
\]
we have $\Delta = 12^3\,(A^3-B^2)$. Regrouping terms, this reads
\begin{eqnarray}
\label{eq:Shafa}
S':\;\;\; B^2 = A^3 - \Delta/12^3
\end{eqnarray}
The latter describes an elliptic curve in the affine $A,B$-plane.
By Siegel's theorem, this curve only has finitely many points with bounded denominator.

\smallskip

For elliptic surfaces over a fixed curve $C$ in characteristics $\neq 2,3$, we can pursue exactly the same approach (with $\Z$ replaced by the coordinate ring $k[C]$).
Then the Weierstrass form (\ref{eq:Shafa}) defines an (isotrivial) elliptic surface $S'$ over $k[C]$.
As a little subtlety, this very Weierstrass form may not be minimal, since $\Delta$ may have roots of high multiplicity;
however, this issue does not cause any problems with integral points: 
they attain bounded denominators upon minimalising, but the height stays bounded.
In particular, any elliptic surface over $C$ with discriminant $\Delta$ can be considered as an (almost) integral section on $S'$.
By \ref{ss:int-surf}, these sections are finite in number.

\subsection{Galois invariance}

We conclude this investigation of Mordell-Weil lattices by noting two properties that are of relevance in arithmetic settings.
These properties refer to the two ways of extending the function field $k(C)$: by extension of the base field $k$ or base change $C'\to C$.

We retain the usual notation of an elliptic curve $E$ over the functions field $K=k(C)$ of a projective curve $C$.
Let $\bar k$ denote the algebraic closure of $k$.
Tacitly, we shall denote the base extension of $E$ from $k(C)$ to $\bar k(C)$ also by $E$ (distinguished by the respective sets of rational points). 

\begin{Proposition}
For any $\sigma\in\mbox{Gal}(\bar k/k)$ and for any points $P,Q\in E(\bar k(C))$, the height pairing is Galois-invariant:
\[
 P^\sigma, Q^\sigma\in E(\bar k(C)) \;\;\; \text{ with }\;\;\; <P^\sigma, Q^\sigma> = <P,Q>.
\] 
\end{Proposition}

The proof of the proposition builds on the associated elliptic surface $S\to C$ over $k$. The crucial property is that the orthogonal projection $\varphi$ commutes with the Galois action:
\[
 \varphi(P^\sigma) = \varphi(P)^\sigma \;\;\forall \; P\in E(\bar k(C)).
\]
The reason for this commutativity lies in the fact that the trivial lattice $T(S)$ is Galois invariant: the Galois action may permute singular fibres and/or fibre components, but it fixes $T(S)$ as a whole.

In consequence, the proposition follows from the Galois-invariance of intersection numbers on projective varieties.

\begin{Remark}
The assumption in \cite[Prop.~8.13]{ShMW} that the base field $k$ be perfect is not necessary.
\end{Remark}

\subsection{Base change}
\label{ss:MWL-base_change}

Assume we are given a morphism of projective curves $C'\to C$ of degree $d$ over some field $k$ (not necessarily algebraically closed). By base change, any elliptic curve $E$ over $k(C)$ can also be interpreted as an elliptic curve over $k(C')$.
In particular, we can consider $k(C)$-rational points on $E$ as elements in $E(k(C'))$.
How does this base change effect the height pairing (specified by the function field over which we regard $E$)?

\begin{Proposition}
\label{Prop:base}
In the above notation, let $P, Q\in E(k(C))$. Then
\[
 <P,Q>_{E(k(C'))} = d\,<P,Q>_{E(k(C))}.
\]
\end{Proposition}

In essence, this property goes back to the corresponding elliptic surfaces $S', S$ with rational map $g:S'\dashrightarrow S$ of degree $d$.
For any $D,E\in\NS(S)$, we have
\[
 g^*(D).g^*(E) = d\, D.E
\]
Hence we only have to prove that pull-back from $E(k(C))$ resp.~$S$ commutes with the orthogonal projection $\varphi$.
This can be seen by restricting to the narrow Mordell-Weil group $E(K)^0$ and then applying the universal property of $\varphi$ and bilinearity.
The precise details are given in the proof of \cite[Prop.~8.12]{ShMW}.

\section{Elliptic K3 surfaces}
\label{s:K3}

Elliptic fibrations are maybe most prominently featured in the theory of K3 surfaces.
For example, K3 surfaces are singled out by the property that the same surface may admit more than one elliptic fibration (Lem.~\ref{Lem:two-fibr}).
On the other hand, there are also several conjectures concerning algebraic surfaces which are not known in general, but for elliptic K3 surfaces with section.
Throughout this section, we will explore selected aspects of elliptic K3 surfaces, 
mainly concerning existence and classification questions.
We start by reviewing some of the general theory before going into elliptic K3 surfaces (cf.~\cite{BHPV} and the references therein).

\subsection{K3 surfaces}

\begin{Definition}
A smooth surface $X$ is called K3 if
\begin{itemize}
 \item 
the canonical bundle is trivial,
\item
$h^1(X,\OO_X)=0$.
\end{itemize}
\end{Definition}

Mostly we will consider projective K3 surfaces.
The name supposedly refers to K\"ahler, Kummer and Kodaira (the reasons for these references will become clear soon).
The original definition albeit goes back to A.~Weil who considered all smooth surfaces with the same differentiable structure as a quartic in $\PP^3$.
Over $\C$, all K3 surfaces are diffeomorphic and deformation equivalent.
In particular, they are simply connected,
as conjectured by Weil and Andreotti.
The conjecture was proved by Kodaira in \cite{Kodaira-structure},
using the theory of analytic elliptic surfaces.

\begin{Example}
\begin{enumerate}[(i)]
 \item
Quartics in $\PP^3$ -- either smooth or with isolated rational double point singularities after a minimal (crepant) desingularisation;
\item
smooth intersections of degree $(2,3)$ in $\PP^4$ or $(2,2,2)$ in $\PP^5$;
\item
double sextics: double coverings of $\PP^2$ branched along a sextic curve (smooth or after resolving isolated rational double point singularities);
\item
Kummer surfaces: starting from an abelian surface $A$, we consider the quotient by the involution $-1$ and resolve the 16 ordinary double points on $A/\langle-1\rangle$ (see Example \ref{Ex:Kummer}).
\end{enumerate}
\end{Example}

The Hodge diamond of a K3 surface is easily computed with Serre duality and Noether's formula:
$$
\begin{array}{ccccc}
 && 1 &&\\
& 0 && 0 &\\
1 && 20 && 1\\
& 0 && 0 &\\
&& 1 &&
\end{array}
$$

\subsection{K3 lattice}

We deduce that a K3 surface has second Betti number $b_2=22$. Cup-product equips
 $H^2(X,\Z)$ with the structure of an integral lattice of rank $22$. Often this lattice is called \textbf{K3 lattice} and denoted by $\Lambda$.
The following properties of $\Lambda$ are well-known:
\begin{itemize}
 \item $\Lambda$ is unimodular by Poincar\'e-duality;
\item
$\Lambda$ has signature $(3,19)$ by the topological index theorem;
\item
$\Lambda$ is even by Wu's formula since the first Chern class is even.
\end{itemize}
Hence the classification of even unimodular  lattices implies that
\[
 \Lambda \cong U^3  \oplus  E_8^2.
\]
A \textbf{marking} $\phi$ of a K3 surface $X$ is given by an explicit isomorphism
\[
 \phi: H^2(X,\Z) \stackrel{\cong}{\longrightarrow} \Lambda.
\]

\subsection{Torelli theorem}

Lattice theory is a crucial tool to study K3 surfaces. 
Its importance was pointed out in the pioneering work by Pjatetki-Shapiro and Shafarevich 
\cite{PSS}.
Recall that an isometry of lattices is an isomorphism of $\Z$-modules preserving the intersection form.
A Hodge isometry is moreover required to preserve the Hodge decomposition after $\C$-linear extension.

\begin{Theorem}[Weak Torelli Theorem]
Two complex K3 surfaces $X, X'$ are isomorphic if and only if there exists a Hodge isometry
\[
 \phi: H^2(X,\Z) \stackrel{\cong}{\longrightarrow} H^2(X',\Z).
\]
\end{Theorem}

A Hodge isometry is called effective if it preserves the K\"ahler cone (i.e.~if $X$ ia algebraic, then the isometry preserves ample divisors).
The strong Torelli Theorem extends the weak Torelli Theorem to the point that the isomorphism $f: X'\stackrel{\cong}{\longrightarrow} X$ is uniquely determined by the condition $f^*=\phi$ if $\phi$ is effective.

The statements made in this paragraph are valid for all K3 surfaces, also non-algebraic ones (which are nonetheless K\"ahler by a theorem of Siu \cite{Siu}).

\subsection{Moduli}
\label{ss:period}

K3 surfaces admit a beautiful moduli theory. Over $\C$, they can be studied through the \textbf{period map}.
Here one fixes a nowhere vanishing holomorphic $2$-form $\omega_X$ up to scalars
and chooses a basis $\{\gamma_1,\hdots,\gamma_{22}\}$ of $H^2(X,\Z)$.
Then the period point $p_X$ associated to the K3 surface $X$ is
\[
 p_X = \left[ \int_{\gamma_1} \omega_X,\hdots,\int_{\gamma_{22}}\omega_X \right]\in\PP^{21}.
\]
In fact, the image of the period map lies in a open subset of $\PP^{20}$
as 
\[
 \int \omega_X\wedge\omega_X =0,\;\;\; \int \omega_X\wedge\bar\omega_X >0.
\]
This 20-dimensional space is called the \textbf{period domain} $\Omega$ of K3 surfaces.
The nice moduli behaviour of K3 surfaces is illustrated by the following fact:
If $X$ varies in a local complete family of K3 surfaces, parametrised by a complex manifold $\mathcal M$, then there is a local isomorphism $\mathcal M \to \Omega$.
For a global description, the degeneration at the boundary has to be taken into account, cf.~\cite{Kuli}.


Now we specialise to projective K3 surfaces. Let $\gamma\in H^2(X,\Z)$ be algebraic. Then
\[
 \int_\gamma \omega_X=0,
\]
since the integral over $\omega_X$ vanishes along any irreducible curve in $X$.
In the moduli space, this yields the intersection of $\Omega$ with a hyperplane.
Thus algebraic K3 surfaces have 19-dimensional moduli, and we achieve a stratification of the moduli space in terms of the Picard number (note that $\rho\leq 20$ in characteristic zero).

\subsection{Polarisation}

On an algebraic K3 surface, one fixes a polarisation, i.e.~an ample divisor $H$ with self-intersection $H^2=2\,d$.
For instance, the hyperplane section on a quartic in $\PP^3$ is an ample divisor of self-intersection $4$.
Similarly, on a double sextic the hyperplane section on $\PP^2$ pulls back to an ample divisor of self-intersection $2$. 
Then one commonly considers moduli of polarised K3 surfaces.

\smallskip

One can extend polarisations to the notion of \textbf{lattice polarisations}.
Here we let $M$ denote a non-degenerate even integral lattice of signature $(1,r)$. 
A K3 surface $X$ is $M$-polarised if there is a primitive embedding
\[
 M \hookrightarrow \NS(X).
\]

\begin{Proposition}[{\cite[Cor.~1.9]{Mo}}]
\label{Prop:lp}
Let $M$ be an even non-degenerate integral lattice of signature $(1,r)$. If $M$ has a primitive embedding into $\Lambda$, then $M$-polarised complex K3 surfaces admit a local moduli space of dimension $19-r$.
\end{Proposition}

\subsection{Positive characteristic}

In positive characteristic, the moduli question has a completely different flavour since there can be \textbf{supersingular} K3 surfaces with $\rho=22$. 
Surprisingly, the supersingular K3 surfaces themselves admit a nice moduli theory which was discovered by Artin in \cite{Artin} and greatly extended by  Ogus \cite{Ogus}.
Over $\bar\F_p$, the \textbf{Tate conjecture} \cite{Tate-C} implies that any K3 surface should have even Picard number.
One reason (but certainly not the only one) to consider K3 surfaces admitting an elliptic fibration with section is that
  the Tate conjecture holds true for them by \cite{ASD}.

Remarkably there are also Torelli theorems for K3 surfaces in positive characteristic.
This has been proved in the supersingular case by Ogus \cite{Ogus-Torelli} and in the ordinary case by Nygaard \cite{Ny}.
Proposition \ref{Prop:lp} has an analogue for ordinary K3 surfaces in positive characteristic in the context of canonical liftings, see
\cite{De}, \cite{Nygaard}.

\subsection{Elliptic K3 surfaces}

An elliptic K3 surface $X$ always has base curve $\PP^1$ by Thm.~\ref{Thm:Pic-var}.
If $X$ admits a section, then by \ref{ss:IP^1} there is a globally minimal generalised Weierstrass form
\[
X:\;\;  y^2 + a_1 \,x\,y+ a_3\,y \, = \, x^3 + a_2 \,x^2 + a_4 \,x+a_6\;\;\; a_i\in k[t],\;\;\deg(a_i)\leq 2i.
\]
To give an impression of the possibilities, we discuss two instructive examples.

\begin{Example}
\label{Ex:RES-K3}
One can derive elliptic K3 surfaces $X$ from rational elliptic surfaces $S$ by quadratic base change over $\PP^1$. 
Here one only has to avoid ramification at non-reduced fibres.
The quadratic twist $S'$ of $S$ at the two ramification points is also K3.
Over $\C$, the K3 surfaces $X, S'$ are isogenous in a certain sense by \cite{Inose}. 
In particular, they have the same Picard number.
Note that $\rho(X)\geq 10$ by inspection of the trivial lattice and Prop.~\ref{Prop:base}.
\end{Example}

\begin{Example}
\label{Ex:Kummer}
Let $E, E'$ denote elliptic curves. 
Consider the Kummer surface $X$ for $E\times E'$ (outside characteristic two).
There are several simple models for $X$.
For instance, if $E, E'$ are given in Weierstrass form
\[
 E:\;\;\; y^2 = f(x),\;\;\; \text{ and }\;\;\; E':\;\;\; y^2 = g(x),
\]
then $X$ can be  defined as the desingularisation of the double sextic
\[
 X:\;\;\; y^2 = f(x)\,g(x').
\]
On the other hand, the projections onto the two factors induce elliptic fibrations of $X$. Projection onto $E'$ yields the  isotrivial elliptic fibration
\begin{eqnarray}\label{eq:Kummer}
 X:\;\;\; g(t)\,y^2 = f(x)
\end{eqnarray}
which is readily converted to an extended Weierstrass form.
This elliptic fibration has four singular fibres of type $\I_0^*$ at the zeroes of $g$ and $\infty$.
There is full two-torsion in the Mordell-Weil group given by $y=0$ and the zeroes of $f(x)$.

The Picard number of the Kummer surface $X$ is determined by $E, E'$ as follows:
\[
 \rho(X) =16+ 2 + \mbox{rank Hom}(E,E'). 
\]
Over $\C$, the Picard number thus only depends on the properties whether $E, E'$ are isogenous and have complex multiplication (CM):
\begin{eqnarray}\label{eq:rho}
\rho(X)=\begin{cases}
18, & \text{if $E, E'$ are not isogenous;}\\
19, & \text{if $E, E'$ are isogenous, but do not have CM};\\
20, & \text{if $E, E'$ are isogenous and have CM.}
             \end{cases}
\end{eqnarray}
In each case, the trivial lattice stays the same:
\[
 T(X) = U  \oplus  D_4^4.
\]
Comparing with Cor.~\ref{Cor:ST}, we deduce that the Mordell-Weil group of the given isotrivial fibrations has rank 0 resp.~1 resp.~2 in the above three cases.
In Example \ref{Ex:Kummer2}, we will investigate further elliptic fibrations on such Kummer surfaces.
\end{Example}


\subsection{Quasi-elliptic fibrations}

In characteristic $2$ and $3$, smooth projective surfaces may admit quasi-elliptic fibrations: 
i.e.~there is a morphism to a projective curve such that the general fibre is a cuspidal rational curve.
In other characteristics, the smoothness assumption for the surface prevents the general fibre from being singular.
A detailed account of the theory of quasi-elliptic surfaces can be found in the book by Cossec and Dolgachev \cite{CD}.

Although many concepts from the theory of elliptic surfaces carry over to quasi-elliptic surfaces,
there are a few exceptional properties:
\begin{itemize}
 \item 
Any quasi-elliptic surface $X\to C$ with section admits a purely inseparable cover of degree $p$ by a ruled surface.
In particular, if $C\cong\PP^1$, then $X$ is unirational and thus $\rho(X)=b_2(X)$.
\item
$X$ admits only certain additive fibre types depending on the characteristic.
\item
There is a different Euler number formula than Thm.~\ref{Thm:Euler_number}, since the general fibre has $e(F)=2$.
Accordingly, there is a different notion of discriminant to detect singular fibres.
\item
The sections admit the structure of a finite  abelian group $\MW(X)$.
This group is in fact $p$-elementary, but its length may exceed two.
\item
The structure theorems from section \ref{s:NS} apply.
For instance, they give 
\[
\rho(X)=\mbox{rank}(T(X)),\;\; \MW(X)=T(X)'/T(X).
\]
\end{itemize}
Thus any ``genus one fibration'' with a section of infinite order is necessarily elliptic.
To verify this, it suffices to compute the height of the section as in \ref{ss:height}.

\begin{Example}
Consider a Kummer surface of product type as in Ex.~\ref{Ex:Kummer}.
If the cubics $f, g$ are defined over some number field and reduce nicely modulo (a prime above) $2$ without multiple zeroes, then (\ref{eq:Kummer}) defines a quasi-elliptic K3 surface with section over $\bar\F_2$.
In contrast to the usual situation, there are five singular fibres of type $\I_0^*$, namely at the zeroes of $g(t)$ including  $\infty$ and at the (double) zero of the formal derivative of $g(t)$.
\end{Example}

\subsection{}
\label{ss:L1}

It will be very convenient to have a characterisation of elliptic fibrations on K3 surfaces in terms of lattice theory.
Recall that any fibre $F$ of an elliptic surface has self-intersection $F^2=0$.
On a K3 surface $X$, the converse statement also holds true.
For ease of notation, we shall employ the following convention throughout the next two paragraphs:

\smallskip

\textbf{Disclaimer for \ref{ss:L1}, \ref{ss:Kondo}:}
In characteristic $2$ and $3$, the notion of elliptic fibration includes quasi-elliptic fibrations.

\smallskip

This disclaimer is necessary since from the abstract information given in the propositions and corollaries, it is in general impossible to detect whether the fibration is elliptic or quasi-elliptic.

\begin{Proposition}[{\cite[\S3, Thm.~1]{PSS}}]
\label{Prop:D^2=0}
A K3 surface $X$ admits an elliptic fibration (not necessarily with section) if and only if there is a divisor $0\neq D\in\NS(X)$ with $D^2=0$.
\end{Proposition}

The main idea to prove the proposition is to apply reflections and inversions to $D$ until it becomes effective.
Then the linear system $|D|$ induces the elliptic fibration.

\begin{Corollary}
\label{Cor:ell-K3}
Any K3 surface with Picard number $\rho\geq 5$ admits an elliptic fibration.
\end{Corollary}

By assumption, $\NS(X)$ is an indefinite lattice of rank at least five. 
It is well-known that any such quadratic form represents zero.
Hence the corollary follows from Prop.~\ref{Prop:D^2=0}.
The general situation on algebraic K3 surfaces is as follows:
$$
\begin{array}{lcl}
\rho\geq 5 && \text{elliptic fibration}\\
\rho=2,3,4 && \text{depends on $X$}\\
\rho=1 && \text{no elliptic fibration}
\end{array}
$$

\subsection{}
\label{ss:Kondo}
The previous considerations did not address the question of a section.
However, one can easily obtain a more precise description from \cite[\S3, Thm.~1]{PSS}:

\begin{Proposition}
\label{Prop:PSS}
Let $X$ be a K3 surface.
Assume that $D$ is an effective divisor on $X$ that has the same type as a singular fibre of an elliptic fibration.
Then $X$ admits a unique elliptic fibration with $D$ as a singular fibre.
Moreover, any irreducible curve $C$ on $X$ with $D.C=1$ induces a section of the elliptic fibration. 
\end{Proposition}

Yet another equivalent way to phrase this abstractly is to ask for an embedding of the hyperbolic plane $U$ into $\NS(X)$:
\[
U\hookrightarrow \NS(X).
\]
If $\NS(X)=U \oplus M$, then $M$ is related to the induced elliptic fibration by \cite[Lem.~2.1]{Kondo}: Let $F$ denote any fibre. Then
\[
 M \cong F^\bot/\langle F\rangle \subset \NS(X).
\]
More specifically, if $M$ decomposes into an orthogonal sum of root lattices associated to Dynkin diagrams, then they correspond to singular fibres of the fibration by a result of Kond\=o \cite[Lem.~2.2]{Kondo}.
This last statement can also be viewed as a special case of Lem.~\ref{Lem:frame-root}.

\begin{Example}
\label{Ex:Kummer2}
On the Kummer surface of $E\times E'$ in Ex.~\ref{Ex:Kummer} (char$(k)\neq 2$), we have singled out 24 rational curves on the two isotrivial fibrations: the components of the four $\I_0^*$ fibres and the torsion sections.
This configuration is also known as the double Kummer pencil.
In Fig.~\ref{Fig:2Km}, the double components and two-torsion sections are represented by the vertical and horizontal lines.

One can easily find further divisors of Kodaira type.
For instance in \cite{SI}, a $\II^*$ fibre is found together with a section and two orthogonal effective divisors which yield further non-reduced singular fibres (see \ref{ss:SI}).
The next figure sketches an elliptic fibration with two $\IV^*$ fibres which are marked in blue and red.
For this fibration, all other curves from the double Kummer pencil serve as sections.
\end{Example}

\begin{figure}[ht!]
\begin{center}
\setlength{\unitlength}{1.08mm}
\begin{picture}(60, 47)
%
%

{\color{red}
\put(7, 10){\line(1, 0){41}}}

{\color{blue}
\put(7, 20){\line(1, 0){41}}
\put(7, 30){\line(1, 0){41}}
\put(7, 40){\line(1, 0){41}}}
%
%
{\color{blue}
\put(10, 5.1){\line(0, 1){4.4}}
\put(10, 10.6){\line(0, 1){8.9}}
\put(10, 20.6){\line(0, 1){8.9}}
\put(10, 30.6){\line(0, 1){8.9}}
\put(10, 40.6){\line(0, 1){4.4}}}

{\color{red}
\put(20, 5.1){\line(0, 1){4.4}}
\put(20, 10.6){\line(0, 1){8.9}}
\put(20, 20.6){\line(0, 1){8.9}}
\put(20, 30.6){\line(0, 1){8.9}}
\put(20, 40.6){\line(0, 1){4.4}}
\put(30, 5){\line(0, 1){4.4}}
\put(30, 10.6){\line(0, 1){8.9}}
\put(30, 20.6){\line(0, 1){8.9}}
\put(30, 30.6){\line(0, 1){8.9}}
\put(30, 40.6){\line(0, 1){4.4}}
\put(40, 5.1){\line(0, 1){4.4}}
\put(40, 10.6){\line(0, 1){8.9}}
\put(40, 20.6){\line(0, 1){8.9}}
\put(40, 30.6){\line(0, 1){8.9}}
\put(40, 40.6){\line(0, 1){4.4}}}
%
%

\put(8, 6){\line(1,1){5.5}}
{\color{blue}
\put(8, 16){\line(1,1){5.5}}
\put(8, 26){\line(1,1){5.5}}
\put(8, 36){\line(1,1){5.5}}}
%
{\color{red}
\put(18, 6){\line(1,1){5.5}}}
\put(18, 16){\line(1,1){5.5}}
\put(18, 26){\line(1,1){5.5}}
\put(18, 36){\line(1,1){5.5}}
%

{\color{red}
\put(28, 6){\line(1,1){5.5}}}
\put(28, 16){\line(1,1){5.5}}
\put(28, 26){\line(1,1){5.5}}
\put(28, 36){\line(1,1){5.5}}
%

{\color{red}
\put(38, 6){\line(1,1){5.5}}}
\put(38, 16){\line(1,1){5.5}}
\put(38, 26){\line(1,1){5.5}}
\put(38, 36){\line(1,1){5.5}}
%
\end{picture}
\end{center}
\vskip -.3cm
\caption{Double Kummer pencil}\label{fig:Km}
\label{Fig:2Km}
\end{figure}
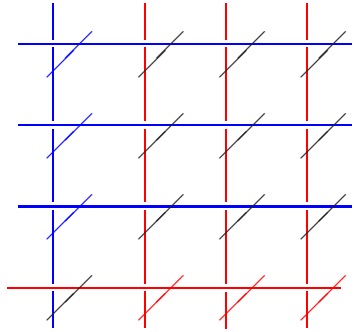

\subsection{}
With the tools exploited in the preceeding sections, we shall aim at two different kinds of classifications of elliptic K3 surfaces:
\begin{itemize}
 \item 
Configurations of singular fibres on all elliptic K3 surfaces with section;
\item
Elliptic fibrations with section on a given K3 surfaces.
\end{itemize}

These classifications will be sketched throughout the next sections.
They will require some additional techniques from lattice theory which will be introduced along the way.

\subsection{Configurations on elliptic K3 surfaces: Existence}

In Example \ref{Ex:RES-K3}, we have already seen two methods to produce elliptic K3 surfaces: by quadratic base change or quadratic twisting from rational elliptic surfaces.
However, since these surfaces do always have Picard number $\rho\geq 10$, these ideas do not suffice to cover all elliptic K3 surfaces.

Abstractly, we could also argue with Prop.~\ref{Prop:PSS} and try to find lattice polarised K3 surfaces, i.e.~with any given shape of the N\'eron-Severi group as in Prop.~\ref{Prop:lp}.
In practice, however, this may be a difficult undertaking -- often it is achieved using the theory of elliptic fibrations as developped in this section and not vice versa!

Thirdly, we have the deformation argument of \ref{ss:deformation} at our hands.
Like for rational elliptic surfaces, this argument reduces the existence problem to checking which configurations of extremal elliptic K3 surfaces can exist.
Afterwards, we will review lattice theoretic means to prove non-existence for the remaining configurations.

\subsection{Extremal elliptic K3 surfaces}
\label{ss:extr-K3}

Recall that an extremal elliptic surface has maximal Picard number, yet finite Mordell-Weil group.
Contrary to rational elliptic surfaces, we encounter on the level of K3 surfaces that the notions of extremal elliptic K3 surfaces differ in characteristic zero and $p>0$:
In characteristic zero, the maximal Picard number is $\rho=20$.
Such K3 surfaces are called \textbf{singular} in the sense of exceptional, as they are lying isolated, but dense in the moduli space of K3 surfaces.

\smallskip

Singular K3 surfaces play an improtant role in the arithmetic of K3 surfaces.
We will briefly explore their properties in \ref{ss:singular_K3}.
Note that we have already seen singular K3 surfaces in Example \ref{Ex:Kummer}:
Kummer surfaces for isogenous elliptic curves with CM.
However, we have also realised that the obvious elliptic fibrations on these Kummer surfaces are \emph{not} extremal:
they have Mordell-Weil rank two.

In contrast, extremal elliptic K3 surfaces in positive characteristic are very rare.
Since they are supersingular by definition, they can be traced back to rational elliptic surfaces by purely inseparable base change. Confer \cite{Ito} for a full treatment.

\subsection{Dessins d'enfant}
\label{ss:dessin}

In this section, we work over $\C$ (or a number field).
For an extremal elliptic fibration over $\PP^1$, all the ramification of the j-map occurs at three point $0, \infty$ and $1$ (or $12^3$ depending on the normalisation).
Hence we have in fact three equivalent representations:
\begin{enumerate}[(1)]
 \item through the j-map, a Belyi map of degree $n$;
\item
through the monodromy representation $\sigma_0, \sigma_1, \sigma_\infty\in S_n$;
\item
through the dessin d'enfant $j^{-1}([0,1])$ associated to the Belyi map $j$.
\end{enumerate}

In essence, the first maps are geometric (or arithmetic) objects while the second representations are a priori purely combinatorial.
In this respect, the dessins d'enfant could be regarded as intermediate between the other two (see for instance \cite{LZ} for details).

\smallskip

In order to prove the existence of an extremal configuration of elliptic fibres, it suffices to derive one of the three representations with the corresponding properties.
In \cite{MP2}, Miranda and Persson exhibited monodromy representations for all semi-stable extremal elliptic K3 surfaces.
Recently, Beukers and Montanus derived explicit defining equations for all these surfaces in \cite{BeuMon}.
A crucial ingredient for this explicit work was the fact that the number of inequivalent dessins d'enfant can be determined by representation theoretic means.

\smallskip

Here we shall only sketch one example, 
which is the first case in the list of  Miranda and Persson  \cite{MP2}.
Prior to \cite{BeuMon} its defining equation was determined by one of us in \cite{ShCR}.
In fact, extremal elliptic surfaces are closely related to the extreme case of the abc-theorem  and in particular to Davenport's inequality (cf.~\cite{Sh-abc}).
Over $\C$, one can thus prove existence and finiteness theorems for semi-stable extremal elliptic surfaces  over $\PP^1$ for any fixed arithmetic genus  $\chi \geq 1$.

\begin{Example}
\label{Ex:19}
An elliptic K3 surface can maximally have singular fibres of type $\I_{19}$
 or $\I_{14}^*$ (if one does not restrict to semistable case). 
In characteristic zero, this bound follows from Cor.~\ref{Cor:ST}.
For positive characteristic ($\neq 2$), the proof in \cite{S-max} combines the theory of supersingular K3 surfaces with Mordell-Weil lattices.

As for existence, we shall only prove this over some number field by drawing an appropriate dessin d'enfant.
Explicitly we only visualise the preimages of $0$ under $j$ which appear as vertices with three adjacent edges.

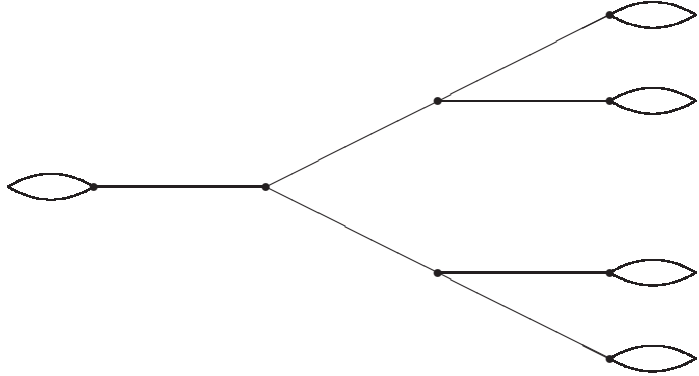
\begin{figure}[ht!]
\setlength{\unitlength}{.45in}
\begin{picture}(10,5)(0,0.5)

%
%

\thinlines

\multiput(2,3)(2,0){2}{\circle*{.1}}
\put(2,3){\line(1,0){2}}
\put(4,3){\line(2,1){4}}
\put(4,3){\line(2,-1){4}}
\multiput(6,4)(2,1){2}{\circle*{.1}}
\multiput(6,2)(2,-1){2}{\circle*{.1}}
\put(6,4){\line(1,0){2}}
\put(6,2){\line(1,0){2}}
\put(8,4){\circle*{.1}}
\put(8,2){\circle*{.1}}

\qbezier(1,3)(1.5,3.3)(2,3)
\qbezier(1,3)(1.5,2.7)(2,3)

\qbezier(8,5)(8.5,5.3)(9,5)
\qbezier(8,5)(8.5,4.7)(9,5)

\qbezier(8,4)(8.5,4.3)(9,4)
\qbezier(8,4)(8.5,3.7)(9,4)

\qbezier(8,2)(8.5,2.3)(9,2)
\qbezier(8,2)(8.5,1.7)(9,2)

\qbezier(8,1)(8.5,1.3)(9,1)
\qbezier(8,1)(8.5,0.7)(9,1)


\end{picture}
\caption{Dessin d'enfant for the $\I_{19}$ fibration}
\label{Fig:I19}
\end{figure}

Note that the dessin d'enfant is symmetric, as the (unique) elliptic surface can in fact be defined over $\Q$ \cite{ShCR}. 
Thus it also reduces nicely mod $p$ for almost all primes $p$ (cf.~\cite{SS2} where also uniqueness is proven for any characteristic $\neq 2$). 
\end{Example}

\subsection{Configurations on elliptic K3 surfaces: Non-existence}

In order to prove that elliptic K3 surfaces do not admit certain configurations of singular fibres, we will use similar techniques as in section \ref{s:s-s}; 
in particular, we will again build the quotient by torsion sections.
However, there is one substantial difference: $\NS(X)$ need not be unimodular (and unless $\rho(X)=2$, so that $\NS(X)=U$, it will not be unimodular in general).

Hence we have to develop different criteria.
The configuration of singular fibres is encoded in the trivial lattice $T(X)$ of an elliptic surface $X$ with section.
If $e(X)=12 \chi(X)$ is a multiple of $24$ (e.g.~if $X$ is K3), then the zero section has even self-intersection $\bar O^2=-\chi(X)$.
Hence zero section and general fibre generate an even unimodular lattice, the hyperbolic plane $U$. 
In particular, $\NS(X)$ is an even lattice by Lem.~\ref{Lem:O,F_purp}.
Hence we can apply techniques from the theory of even lattices as developped by Nikulin \cite{N}.
In the context of elliptic fibrations, these techniques have first been employed systematically by Miranda and Persson \cite{MP2}.

\subsection{Discriminant group and discriminant form}
\label{ss:discs}

Let $L$ denote an even integral non-degenerate lattice $L$ with quadratic form $(\cdot,\cdot)$.
We define its dual $L^\vee$ by 
\[
L^\vee = \{x\in L\otimes \Q; (x,y)\in\Z\; \forall\, y\in L\}.
\]
Then the \textbf{discriminant group} $G_L$ of $L$ is the finite abelian group
\[
 G_L = L^\vee/L.
\]
\begin{Example}
For a Dynkin diagram $W$, let $V$ be the corresponding type of singular fibre from \ref{ss:Dynkin}
(regardless of the ambiguity in case $W=A_1$ or $A_2$). 
Then the discriminant group $G_W$ of $W$ equals $G(V)$ in the notation of Lemma \ref{Lem:group_sing}.
\end{Example}

We now consider the following set-up: $N$ is an even unimodular lattice (think $\Lambda$).
We assume that the non-degenerate lattice $L$ embeds primitively into $N$.
Let $M$ denote the orthogonal complement of $L$ in $N$:
\[
 M=L^\bot \subset N.
\]

\begin{Proposition}[{\cite[Prop.~1.6.1]{N}}]
\label{Prop:disc_grp}
Under the above assumptions, $G_L\cong G_M$.
\end{Proposition}

Most of the time, the above result is applied to the full N\'eron-Severi group of a given surface $X$ (with $H^2(X,\Z)$ even).
Here, however, we are mainly concerned with the configuration of singular fibres.
Hence we have to take the subtlety into account that the trivial lattice $T(X)$ of an elliptic surface need not be primitive in $\NS(X)$.
By \ref{ss:torsion1},  we obtain the primitive closure $T(X)'$  after saturating $T(X)$ by adding the torsion sections.

The basic idea now is to compare the length of $G_{T(X)}$ against the maximum possible length of the discriminant group of $(T(X)')^\bot$. The latter equals $b_2(X)-\mbox{rank}(T(X))$. 
We now specialise to the K3 case where $\Lambda$ is in fact even of rank $b_2(X)=22$:

\smallskip

\textbf{Length criterion:}
Let $X$ be an elliptic K3 surface with section.
If the $p$-length of $G_{T(X)}$ exceeds $22-\mbox{rank}(T(X))$, then $X$ admits a $p$-torsion section.

\smallskip

But then we can consider the quotient by (translation by) the $p$-torsion section.
Whenever the Euler numbers of the resulting singular fibres do not sum up to $24$, we establish a contradiction.

\begin{Lemma}
\label{Lem:9-tuples}
The configurations $[1, 2^7, 9], [2^7,3,7], [2^7,5^2]$ do not exist.
\end{Lemma}

\begin{proof}
By the length criterion, each of these configurations implies the existence of a 2-torsion section. But then the odd fibres imply that the quotient has Euler number at least $7\cdot 1 + 2\cdot 10=27$, contradiction.
\end{proof}

For all other $n$-tuples ($n\geq 9$), Miranda and Persson \cite{MP2} prove existence through the deformation argument from \ref{ss:deformation}.
On the other hand, all 6-, 7- and 8-tuples deforming to either of those from Lemma \ref{Lem:9-tuples} cannot exist. 
Proving non-existence for other $n$-tuples with $n=6,7,8$ requires some further criteria (see \cite{MP2}).
Here we shall only introduce the discriminant form which will be of later use.

\smallskip

The \textbf{discriminant form} $q_L$ of an even, integral and non-degenerate lattice $L$ is defined as the induced quadratic form on the discriminant group:
\begin{eqnarray*}
q_L: G_L & \to & \Q/2\Z\\
x & \mapsto & x^2 \mod 2\Z.
\end{eqnarray*}

\begin{Proposition}[{\cite[Prop.~1.6.1]{N}}]
\label{Prop:disc_form}
Under the assumptions of Prop.~\ref{Prop:disc_grp}, $q_L = -q_M$.
\end{Proposition}

Nikulin moreover showed that the genus of an even lattice (i.e.~its isogeny class) is determined by signature and discriminant form \cite[Cor.~1.9.4]{N}.
There are specific cases where the genus of a lattice consists of a single class.
For definite lattices, this typically requires that both rank and (the absolute value of the) discriminant are relatively small (cf.~\cite[\S 15, Cor.~22]{CS}).
For indefinite lattices, the situation differs significantly:

\begin{Proposition}[Kneser {\cite{Kneser}}, Nikulin {\cite[Cor.~1.13.3]{N}}]
\label{Prop:KN}
Let $L$ be an indefinite even lattice.
If the length of $G_L$ does not exceed $\mbox{rank}(L)-2$, then the genus of $L$ consists of a single class.
\end{Proposition}

\subsection{}
\label{ss:Shimada}

There is also a purely lattice-theoretic approach towards configurations of singular fibres.
Namely we consider lattice polarised K3 surfaces and apply Prop.~\ref{Prop:lp}.
If the trivial lattice of a given configuration embeds primitively into $\Lambda$, then this suffices to prove existence by Kond\=o's argument from \ref{ss:Kondo}.

The argument becomes more involved as soon as the configuration forces sections by the length criterion.
But then Lemma \ref{Lem:frame-root} shows how to recover the singular fibres from the frame.
This approach has been pursued with great success by Shimada and Zhang:
\begin{itemize}
 \item 
In \cite{Shimada}, Shimada classifies all possible configurations on elliptic K3 surfaces including the torsion subgroups of $\MW$. 
For the extra criterion implementing torsion sections, confer \cite[Thm.~7.1]{Shimada}.
\item
In \cite{SZ}, Shimada and Zhang classify all extremal elliptic K3 surfaces.
This classification extends \cite{MP2} in two directions: the authors also treat configurations that are not semi-stable, and they determine Mordell-Weil groups
 and transcendental lattices (cf.~\ref{ss:singular_K3}).
\end{itemize}

\subsection{Elliptic fibrations}
By the Kodaira-Enriques classification of algebraic surfaces (cf.~\cite{BHPV}),
only K3 surfaces and abelian surfaces have trivial canonical bundle.
In consequence, any other surface can admit at most one elliptic fibration with section:

\begin{Lemma}
\label{Lem:two-fibr}
Assume that the surface $S$ admits two distinct elliptic fibrations with section which are not of product type. Then $S$ is a K3 surface. 
\end{Lemma}

\begin{proof}
If $K_S$ is not trivial, then the canonical bundle formula (Thm.~\ref{Thm:can}) implies that any two fibres of the two fibrations are algebraically equivalent. 
Hence the fibrations are isomorphic.
If $K_S$ is trivial, then $S$ cannot be abelian, since the fibrations are not of product type; thus $S$ is K3.
\end{proof}

If we let go the condition of a section,
then there are also Enriques surfaces to be considered.
Enriques surfaces (outside characteristic two) 
are quotients of K3 surfaces by fixed point-free involutions.
They can be studied purely in terms of their universal covers -- which return the original K3 surfaces.
In particular an elliptic fibration on an Enriques surface (without section) induces an elliptic fibration 
on the covering K3 surface (which may well admit a section).
For general details, see \cite{CD}.
Remarkably, several classification problems for elliptic fibrations and beyond have first been solved for Enriques surfaces (see \cite{BP} for an instance).


\subsection{Elliptic fibrations on a K3 surface}
\label{ss:Nishi}

On a single K3 surface, the elliptic fibrations with section can be classified by several means.
To our knowledge, Oguiso was the first to establish a complete classification of elliptic fibrations on complex Kummer surfaces for $E\times E'$ as in Example \ref{Ex:Kummer} in case $E, E'$ are not isogenous \cite{Oguiso}.
Oguiso pursued an essentially geometric technique
based on finiteness arguments of Sterk \cite{Sterk} that also allowed him to detect (infinitely many) non-isomorphic fibrations with the same configuration of singular fibres.
On the other hand, Sterk showed in \cite{Sterk} that any K3 surface $X$ does only have finitely many elliptic fibrations up to $\mbox{Aut}(X)$.

\smallskip

Another approach towards this problem is based on lattice theoretic ideas. 
An essential ingredient is a gluing technique which can be traced back to Witt and Kneser \cite{Kne}.
From this technique and the classifcation of definite unimodular lattices of rank 24 in \cite{Nie},
Nishiyama developed an elegant method to classify all elliptic fibrations on a given K3 surface in \cite{Nishi}.
This technique moreover builds on a converse of Nikulin's results in \ref{ss:discs}:

\begin{Theorem}[{\cite[Cor.~1.6]{N}}]
\label{Thm:Nik}
Let $L, M$ be even non-degenerate integral lattices such that
\[
 G_L\cong G_M,\;\;\; q_L=-q_M.
\]
Then there exists a unimodular overlattice $N$ of $L\oplus M$ such that both $L$ and $M$ embed primitively into $N$ and are each other's orthogonal complement:
\[
 L = M^\bot\subset N,\;\;\; M=L^\bot\subset N.
\]
\end{Theorem}

Of course, we can apply the above theorem to the N\'eron-Severi lattice of any K3 surface $X$.
What makes elliptic fibrations special is the fact that the frame $W(X)$ is even and negative-definite of rank $\rho(X)-2$ by Lem.~\ref{Lem:O,F_purp}.
Hence we can try to embed the frame into a negative-definite lattice.
This brings the advantage that negative-definite lattices of small rank have been fully classified.
Note that in the present case, $W(X)$ has rank at most $18$, so it suffices to consider unimodular negative-definite lattices of rank $24$.
Luckily these so-called \textbf{Niemeier lattices} admit a fairly  simple classification in terms of their root lattices:

\begin{Theorem}[{\cite[Satz 8.3]{Nie}}]
A negative-definite even unimodular lattice $N$ is determined by its root lattice $N_{\mbox{root}}$ up to isometry.
There are 24 possibilities for $N$.
\end{Theorem}

Given an elliptic K3 surface $X$, Nishiyama aims at embedding the frames of all elliptic fibrations into Niemeier lattices.
For this purpose he first determines an even negative-definite lattice $M$ such that
\[
 q_M = -q_{\NS(X)},\;\;\; \mbox{rank}(M)+\rho(X)=26.
\]
By Thm.~\ref{Thm:Nik}, $M \oplus W$ has a Niemeier lattice as an overlattice for each frame $W$ of an elliptic fibration on $X$.
Thus one is bound to determine the (inequivalent) primitive embeddings of $M$ into Niemeier lattices.
To achieve this, it is essential to consider the root lattices involved (as shown in \cite{Nishi}).
In each case, the orthogonal complement of $M$ gives the corresponding frame $W$.
Then one can read off Mordell-Weil lattice (plus the torsion in $\MW$) and singular fibres from $W$ by
\[
 T(X) = U \oplus W_{\mbox{root}},\;\;\; 
\MWL(X) = W/W_{\mbox{root}}'
\]

\subsection{Transcendental lattice}

The existence of a negative-definite lattice $M$ as above was subsequently proven by Nishiyama in \cite{Nishi-Saitama}.
The argument starts from the \textbf{transcendental lattice} of $X$, 
i.e.~the orthogonal complement of $\NS(X)$ in $H^2(X,\Z)$ with respect to cup-product:
\[
\T(X) = \NS(X)^\bot \subset H^2(X,\Z).
\]
In general, $\T(X)$ is an even lattice of rank $r=22-\rho(X)$ and signature $(2,20-\rho(X))$.
Let $t=r-2$.
By \cite[Thm.~1.12.4]{N}, $\T(X)[-1]$ admits a primitive embedding into the following indefinite unimodular lattice:
\[
 \T(X)[-1] \hookrightarrow U^t  \oplus  E_8.
\]
Then we define $M$ as the orthogonal complement of $\T(X)[-1]$ in $U^t \oplus E_8$.
By construction, this lattice is negative definite of rank $t+6=r+4=26-\rho(X)$ with discriminant form
\[
 q_M = -q_{\T(X)[-1]} = q_{\T(X)} = -q_{\NS(X)}.
\]
Hence $M$ takes exactly the shape required for Nishiyama's technique in \ref{ss:Nishi}.

\subsection{}

It remains to discuss when the above technique guarantees the existence of an elliptic fibration with fixed frame $W$ on the given K3 surface $X$.

\smallskip

First it is instructive to note that in any case there is some K3 surface with the elliptic fibration specified by $W$.
Apply Thm.~\ref{Thm:Nik} to $\T(X)$ and $U\oplus W$ to derive a primitive embedding of $U\oplus W$ into the K3 lattice $\Lambda$.
By Prop.~\ref{Prop:lp}, there is a family of K3 surfaces with N\'eron-Severi group $\NS=U\oplus W$.
In particular, each of these K3 surfaces admits an elliptic fibration with section and frame $W$ by \ref{ss:Kondo}.

\smallskip

Now let us turn to the specific K3 surface $X$.
Assume that we have found a candidate frame $W$ for an elliptic fibration on $X$  by Nishiyama's technique.
By construction, $\NS(X)$ and $U\oplus W$ have the same signature and discriminant form.
In particular, they lie in the same genus (cf.~\ref{ss:discs}).
We need to decide whether this genus consists of a single class.
In order to apply Prop.~\ref{Prop:KN} to the complex K3 case,
the decisive property is that $q_{\NS(X)}=-q_{\T(X)}$ by Prop.~\ref{Prop:disc_form}.
Thus $\NS(X)$ and $\T(X)$ have the same length, 
and the lattice of bigger rank (unless both have rank $11$) satisfies the condition of Prop.~\ref{Prop:KN}.

\begin{Lemma}[{\cite[Cor.~2.9, 2.10]{Mo}}]
Let $X$ be a complex K3 surface.
\begin{enumerate}[(a)]
\item
If $\rho(X)\geq 12$, then $\NS(X)$ is uniquely determined by signature and discriminant form.
\item
If $\rho(X)\leq 10$, then $\T(X)$ is uniquely determined by signature and discriminant form.
\end{enumerate}
\end{Lemma}

Applied to the above situation arising from Nishiyama's technique,
the lemma implies that $\NS(X)$ and $U\oplus W$ are isomorphic 
if $\rho(X)\geq 12$.
By \ref{ss:Kondo}, $X$ thus admits an elliptic fibration with frame $W$.

\subsection{}

Nishiyama applied this technique to several K3 surfaces.
For instance, he determined all elliptic fibrations on the complex Kummer surfaces from Example \ref{Ex:Kummer} for the following cases (thus also recovering Oguiso's result \cite{Oguiso}).
$$
\begin{array}{ccc}
\hline
E, E' & \rho(X) & \text{no.~of fibrations}\\
\hline
E\cong E', j(E)=0 & 20 & 30\\
E\cong E', j(E)=12^3 & 20 & 25\\
E\cong E' \text{ without CM} & 19 & 34\\
E\not\sim E' & 18 & 11\\
\hline
\end{array}
$$
Subsequently, the classification begun by Oguiso in \cite{Oguiso} was taken up by Kuwata and Shioda.
In \cite{KS} they derive defining equations and elliptic parameters for all fibrations in the last case.

\subsection{}

Given a K3 surface $X$ one can also ask the general question whether $X$ admits any elliptic fibration with section at all.
Over $\C$ there is a uniform answer if the Picard number is big enough:

\begin{Lemma}
Every complex K3 surface of Picard number at least $13$ admits an elliptic fibration with section.
\end{Lemma}

The claim follows directly from a result of Nikulin \cite[Cor.~1.13.5]{N}.
In detail, Nikulin derived the following implication that relates the rank of  an indefinite even integral lattice $L$ with its length, i.e.~the minimum number of generators of the discriminant group $G_L=L^\vee/L$:
\[
\mbox{rank}(L)\geq \mbox{length}(L) +3 \Longrightarrow U \hookrightarrow L.
\]
In the present situation, the length of $\NS(X)$ is trivially bounded by both the rank of $\NS(X)$ and the rank of $\T(X)$.
The latter is $22-\rho(X)\leq 9$ by assumption.
Hence the inequality of rank against length holds for $\NS(X)$, and the embedding of $U$ into $\NS(X)$ gives an elliptic fibration with section.

\section{Arithmetic of elliptic K3 surfaces}
\label{s:arith}

Throughout this paper, we have tried to emphasise the importance of elliptic fibrations for arithmetic considerations.
In this section, we point out several applications to K3 surfaces.
We start by reviewing singular K3 surfaces.
This naturally leads to questions of Mordell-Weil ranks (\ref{ss:Kuwata}) and modularity (\ref{ss:mod}).
We shall also briefly look into connections with Shimura curves, abelian surfaces with real multiplication and rational points.

\subsection{Singular K3 surfaces}
\label{ss:singular_K3}

Over $\C$, the maximal Picard number for a K3 surface $X$ is $\rho(X)=20$.
The K3 surfaces attaining this bound are called \textbf{singular}.
We have seen specific examples in two instances:
\begin{itemize}
\item 
extremal elliptic K3 surfaces (\ref{ss:extr-K3}, Example \ref{Ex:19});
\item
Kummer surfaces of isogenous CM-elliptic curves (Example \ref{Ex:Kummer}).
\end{itemize}
Note that the first examples constitute a finite set up to isomorphism.
Hence we shall concentrate on the Kummer surfaces.

\smallskip

For a K3 surface $X$ over $\C$, recall the   lattice $\T(X)$, i.e.~the orthogonal complement of $\NS(X)$ in $H^2(X,\Z)$ with respect to cup-product:
\[
\T(X) = \NS(X)^\bot \subset H^2(X,\Z).
\]
In general, $\T(X)$ is an even lattice of signature $(2,20-\rho(X))$.
For a singular K3 surface, the transcendental lattice is positive definite of rank two.
Thus a singular K3 surface $X$ is uniquely determined by $\T(X)$ up to isomorphism due to the Torelli Theorem.
The intersection form on $\T(X)$  provides a useful link to class group theory -- and thereby to elliptic curves with complex multiplication (CM).

\subsection{}
Let us first consider \textbf{singular abelian surfaces}, i.e.~complex abelian surfaces $A$ with maximal Picard number $\rho(A)=4$. 
One defines the transcendental lattice $\T(A)$ as above.
In \cite{SM}, Shioda and Mitani showed that every singular abelian surfaces is isomorphic to the product of isogenous CM elliptic curves.
The curves can be given explicitly in terms of $\T(A)$ as follows.
The intersection form on $\T(A)$ is expressed uniquely up to conjugation in $SL_2(\Z)$ by a quadratic form
\begin{eqnarray}\label{eq:Q}
Q(A)=
\begin{pmatrix}
2a & b\\
b & 2c
\end{pmatrix},
\;\;\;\;\; a,b,c\in\Z.
\end{eqnarray}
We denote its discriminant by $d=b^2-4ac<0$.
By \cite{SM}, $A$ is isomorphic to the product of the following elliptic curves $E, E'$ (written as complex tori):
\begin{eqnarray}\label{eq:E}
E=E_\tau,\;\;\tau = \dfrac{-b+\sqrt{d}}{2a},\;\;\;\;\; E'=E_{\tau'},\;\;\tau' = \dfrac{b+\sqrt{d}}2.
\end{eqnarray}

By construction, $E$ and $E'$ are isogenous with CM in $K=\Q(\sqrt{d})$.
In particular, this result implies that every singular abelian surface can be defined over the Hilbert class field $H(d)$ (a certain abelian Galois extension of $K$ with prescribed ramification and Galois group $Cl(d)$), since the same holds for $E$ and $E'$ by Shimura's work \cite[\S 7]{Shimura}.

\subsection{}

The first approach towards singular K3 surfaces was to consider \textbf{Kummer surfaces} of singular abelian surfaces as in Example \ref{Ex:Kummer}.
In general, the transcendental lattice of a Kummer surface $\Km(A)$ is obtained from the abelian surface $A$ by multiplying the intersection form $Q(A)$ on $\T(A)$ by a factor of $2$
by \cite{PSS}:
\[
 \T(\Km(A)) = \T(A)[2].
\]
Hence any singular K3 surface $X$ with $\T(X)$ two-divisible (as an even lattice) can be realised as a Kummer surface for appropriate $E, E'$.
However, the transcendental lattice of a singular K3 surface need not be two-divisible as we will see for the following extremal elliptic K3 surface.

\begin{Example}
Consider the elliptic K3 surface $X$ with an $\I_{19}$ fibre from Example \ref{Ex:19}.
By Thm.~\ref{Thm:E-NS}, $X$ has N\'eron-Severi group
\[
 \NS(X) = U  \oplus  A_{18}.
\]
As its orthogonal complement in $H^2(X,\Z)$, the transcendental lattice $\T(X)$ has discriminant $d=-19$.
As $\Q(\sqrt{-19})$ has class number one, $Q(X)$ is uniquely represented by the quadratic form
\[
 Q(X) = \begin{pmatrix}
         2 & 1\\ 1 & 10
        \end{pmatrix}.
\]
In particular, $Q(X)$ is not two-divisible.
Hence $X$ cannot be isomorphic to a Kummer surface over $\C$.
\end{Example}

\subsection{Shioda-Inose structure}
\label{ss:SI}

This failure of the Kummer map to be surjective was overcome by Shioda and Inose by means of elliptic fibrations on Kummer surfaces \cite{SI}.
They started with a Kummer surface of product type as in Example \ref{Ex:Kummer}.
Explicitly they identified a divisor $D$ of type $\II^*$ in the double Kummer pencil as mentioned in Example \ref{Ex:Kummer2}.
This divisor is marked blue in the following diagram of the double Kummer pencil:

\begin{figure}[ht!]
\begin{center}
\setlength{\unitlength}{1.08mm}
\begin{picture}(60, 47)
%
%

{\color{yellow}
\put(7, 10){\line(1, 0){41}}}

{\color{blue}
\put(7, 20){\line(1, 0){41}}
\put(7, 30){\line(1, 0){41}}}
\put(7, 40){\line(1, 0){41}}
%
%
{\color{blue}
\put(10, 5.1){\line(0, 1){4.4}}
\put(10, 10.6){\line(0, 1){8.9}}
\put(10, 20.6){\line(0, 1){8.9}}
\put(10, 30.6){\line(0, 1){8.9}}
\put(10, 40.6){\line(0, 1){4.4}}
\put(21, 5.1){\line(0, 1){4.4}}
\put(21, 10.6){\line(0, 1){8.9}}
\put(21, 20.6){\line(0, 1){8.9}}
\put(21, 30.6){\line(0, 1){8.9}}
\put(21, 40.6){\line(0, 1){4.4}}}

{\color{red}
\put(30, 5){\line(0, 1){4.4}}
\put(30, 10.6){\line(0, 1){8.9}}
\put(30, 20.6){\line(0, 1){8.9}}
\put(30, 30.6){\line(0, 1){8.9}}
\put(30, 40.6){\line(0, 1){4.4}}
\put(40, 5.1){\line(0, 1){4.4}}
\put(40, 10.6){\line(0, 1){8.9}}
\put(40, 20.6){\line(0, 1){8.9}}
\put(40, 30.6){\line(0, 1){8.9}}
\put(40, 40.6){\line(0, 1){4.4}}}
%
%

\put(8, 6){\line(1,1){5.5}}
{\color{blue}
\put(8, 16){\line(1,1){5.5}}
\put(8, 26){\line(1,1){5.5}}
\put(8, 36){\line(1,1){5.5}}}
%
{\color{blue}
\put(18, 6){\line(1,1){5.5}}
\put(18, 16){\line(1,1){5.5}}}
\put(18, 26){\line(1,1){5.5}}
\put(18, 36){\line(1,1){5.5}}
%

{\color{red}
\put(28, 6){\line(1,1){5.5}}}
\put(28, 16){\line(1,1){5.5}}
\put(28, 26){\line(1,1){5.5}}
{\color{red}

\put(28, 36){\line(1,1){5.5}}
%

\put(38, 6){\line(1,1){5.5}}}
\put(38, 16){\line(1,1){5.5}}
\put(38, 26){\line(1,1){5.5}}

{\color{red}
\put(38, 36){\line(1,1){5.5}}}
%
\end{picture}
\end{center}
\vskip -.3cm
\caption{Singular fibres in the double Kummer pencil}\label{fig:SI}
\end{figure}
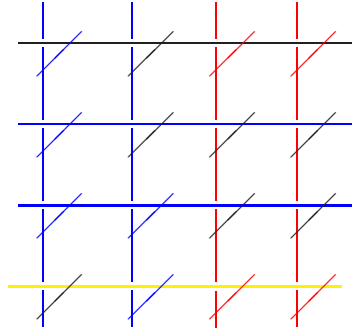

By Prop.~\ref{Prop:PSS}, the Kummer surface admits an elliptic fibration with $D$ as a singular fibre and section corresponding to the rational curve marked in yellow.
As the red marked curves are perpendicular to $D$, they constitute further fibre components.
Thus there are two more reducible fibres.
Shioda-Inose prove that their types can only be $\I_0^*, \I_1^*$ or possibly once $\IV^*$.

Consider the quadratic base change $X$ of this elliptic fibration on $\Km(E\times E')$ ramified at the last two singular fibres.
Since these fibres are non-reduced, the base change has again $e(X)=24$.
Thus $X$ is a K3 surface, and it is singular since the surfaces are isogenous.
Shioda and Inose proceeded to prove that $X$ has the same transcendental lattice as the original abelian surface $A$.
This concludes the construction of a singular K3 surface with any given transcendental lattice.

Note that the deck transformation corresponding to the quadratic base change defines a Nikulin involution on $X$:
It has exactly eight isolated fixed points and leaves the holomorphic 2-form invariant.
The fixed points sit in the two branch fibres; their resolution yields the corresponding fibres on the Kummer surface.
This general construction -- abelian surface and K3 surface with the same transcendental lattice 
such that Kummer quotient and Nikulin involution yield the same Kummer surface -- is nowadays referred to as \textbf{Shioda-Inose structure}. 
$$
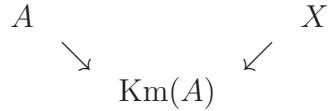
\begin{array}{ccccc}
A &&&& X\\
& \searrow && \swarrow &\\
&& \text{Km}(A) &&
\end{array}
$$
The terminology was introduced by Morrison who worked out lattice theoretic criteria to decide which K3 surfaces of Picard number $\rho\geq 17$ admit a Shioda-Inose structure \cite{Mo}.
Recently Kumar extended Shioda-Inose's construction for Kummer surfaces of product type to Jacobians of genus two curves \cite{Kumar}.
This time, elliptic fibrations with one fibre of type $\\II^*$ and $\III^*$ each play the central role.

\subsection{Inose fibration}

The construction of Shioda-Inose can be defined over some explicit number field:
\begin{enumerate}[(1)]
 \item 
The elliptic curves $E, E'$ are defined over the Hilbert class field $H(d)$.
\item
The Kummer quotient respects the base field of the abelian surface (cf.~Example \ref{Ex:Kummer}).
\item
The field of definition of the divisor $D$ of type $\II^*$ only involves the $2$-torsion of $E$ and $E'$.
Thus $D$ and with it the elliptic fibration $|D|$ are defined over $H(d)(E[2], E'[2])$.
\item
The quadratic base change can be defined over the same field, since the ramification points are either both rational or conjugate quadratic as there cannot be any further singular fibres of the same type by the Euler number $e(\Km(A))=24$.
\end{enumerate}
The Shioda-Inose structure has some interesting consequences for the arithmetic of a singular K3 surface $X$. 
For instance, over some extension of $H(d)$ the zeta function of $X$ can be expressed in terms of Hecke characters \cite[Thm.~6]{SI}.

Subsequently Inose exhibited an explicit model for the elliptic fibration on $X$ induced by the Shioda-Inose structure.
By construction, this fibration has two singular fibres of type $\II^*$.
Hence it admits a Weierstrass equation of the following shape (with $\II^*$'s at $0,\infty$):
\begin{eqnarray}\label{eq:Inose}
X:\;\;\; y^2 = x^3 - 3\,A\,t^4\,x + t^5\,(t^2 - 2\, B\,t+1).
\end{eqnarray}
This fibration is often called \textbf{Inose fibration}.
Inose showed in \cite{Inose} that $A$ and $B$ are related to the j-invariants of $E, E'$ as follows:
\[
 A^3 = j(E)\,j(E')/12^6,\;\;\; B^2 = (1-j(E)/12^3)\,(1-j(E')/12^3). 
\]
This gives a model of $X$ over a degree six extension of $H(d)$.
In fact, the Inose fibration can be twisted in such a way that the coefficients are already defined over $H(d)$ (cf.~\cite[Prop.~4.1]{S-fields}).

The remaining singular fibres of Inose's fibration are easily computed from $E$ and $E'$.
We also list the Mordell-Weil ranks and the discriminants in the case where $X$ is singular, i.e.~$E$ and $E'$ are isogenous with CM.
In the non-CM case, the $\MW$-rank drops by one if $E\sim E'$, resp.~by two if $E\not\sim E'$ (exactly as in Example \ref{Ex:Kummer}).

\begin{table}[ht!]
$$
\begin{array}{ccccc}
\hline
E, E' & j & \text{sing fibres} & \mbox{rank}(\MW) & d\\
\hline
E\cong E' & j(E)=0 & \IV & 0 & -3\\
& j(E)=12^3 & 2\,\I_2 & 0 & -4\\
& j(E)\neq 0, 12^3 & \I_2, 2\, \I_1 & 1 & \neq -3, -4\\
\hline
E \not\cong E' & j(E)\,j(E')=0 & 2 \, \II & 2 & -3\,N^2 (N>1)\\
& j(E)\,j(E')\neq 0 & 4\,\I_1 & 2 & -16, -28 \text{ or } h(d)>1\\
\hline
\end{array}
$$
\caption{Mordell-Weil ranks of Inose's fibration (CM case)}
\label{Table:Inose}
\end{table}

\subsection{Automorphisms}

Except for two cases, Inose's fibration always has positive Mordell-Weil rank.
Hence these singular K3 surfaces have infinite automorphisms groups. 
Shioda--Inose made this observation and extended it to all singular K3 surfaces:

\begin{Lemma}
 Any singular K3 surface has infinite automorphism group.
\end{Lemma}

To complete the proof of the lemma, Shioda and Inose pursued an explicit geometric approach for the two remaining singular K3 surfaces \cite[\S 5]{SI}.
Later Vinberg calculated the automorphism groups of these two ``most algebraic'' K3 surfaces completely \cite{Vinberg}.
Alternatively, we could use the classification of elliptic fibrations on these K3 surfaces along the lines of \ref{ss:Nishi}.
In \cite{Nishi}, Nishiyama showed that there are exactly six elliptic fibrations for $d=-3$ and 13 for $d=-4$.
In either case, $X$ admits an elliptic fibration with Mordell-Weil rank one, so the automorphism group is infinite. 

\subsection{}
\label{ss:cyclic-base}

We shall now start with the Inose fibration (\ref{eq:Inose}) and apply cyclic base changes ramified at the $\II^*$ fibres.
For instance, a quadratic base change yields another K3 surface $X'$ with two singular fibres of type $\IV^*$ instead of $\II^*$.
Shioda showed in \cite{Sandwich} that $X'$ returns in fact the Kummer surface $\Km(E\times E')$ from the Shioda-Inose structure -- we have identified the singular fibres in the double Kummer pencil in~Example \ref{Ex:Kummer2}.
Thus $X$ is sandwiched by this very Kummer surface.

In general, all such cyclic base changes of degree $n$ up to six return isogenous K3 surfaces $X^{(n)}$. By \cite{Shioda}, the base change results in a multiplication of the intersection form on the transcendental lattice by the degree:
\[
\T(X^{(n)}) = \T(X)[n]\;\;\; (n=1,\hdots,6).
\]

\subsection{Mordell-Weil ranks of elliptic K3 surfaces}
\label{ss:Kuwata}

By Cor.~\ref{Cor:ST}, a complex elliptic K3 surface could have any Mordell-Weil rank from 0 to 18, since $\rho\leq 20$.
Cox proved in  \cite{Cox} that all these ranks actually occur. 
This result is a consequence of the surjectivity of the period map;
however, Cox did not give any explicit examples. 

To overcome this lack of examples, Kuwata considered the cyclic base changes of Inose's fibration from \ref{ss:cyclic-base}.
He noticed that since $\rho$ stays constant under each base change, these surfaces provide explicit  examples of complex elliptic K3 surfaces with any Mordell-Weil rank from 0 to 18 except for 15 \cite{Kuwata}. 
Here we produce a table of the fibre configurations and Mordell-Weil ranks (Table \ref{Table:Kuwata}).

\begin{table}[ht!]
$$
\begin{array}{c||cc||c|c|c}
\hline
& \multicolumn{2}{c}{E\cong E'} 
& E\not\cong E' & E\sim E' & E\not\sim E'\\
\text{degree} & \text{config} & \mbox{MW-rank} & \text{config} & \mbox{MW-rank} & \mbox{MW-rank}\\
\hline
1 & 2\,\II^*, \I_2, 2\,\I_1 & \begin{cases} 1\\0\end{cases} & 2\,\II^*, 4\,\I_1 & \begin{cases} 2\\1\end{cases} & 0\\
2 & 2\,\IV^*, 2\,\I_2, 4\, \I_1 & \begin{cases} 4\\3\end{cases} & 2\,\IV^*, 8\,\I_1 & \begin{cases} 6\\5\end{cases} & 4\\
3 & 2\,\I_0^*, 3\,\I_2, 6\,\I_1 &\begin{cases} 7\\6\end{cases} & 2\,\I_0^*, 12\, \I_1 & \begin{cases} 10\\9\end{cases} & 8\\
4 & 2\,\IV, 4\,\I_2,8\,\I_1 & \begin{cases} 10\\9\end{cases} & 2\,\IV, 16\,\I_1 & \begin{cases} 14\\13\end{cases} & 12\\
5 & 2\,\II, 5\,\I_2, 10\,\I_1 & \begin{cases} 13\\12\end{cases} & 2\,\II, 20\,\I_1 & \begin{cases} 18\\17\end{cases} & 16\\
6 &  6\,\I_2, 12\,\I_1 & \begin{cases} 12\\11\end{cases} & 24\,\I_1 & \begin{cases} 18\\17\end{cases} & 16\\
\hline
\end{array}
$$
\caption{Mordell-Weil ranks of cyclic base changes of Inose's fibration}
\label{Table:Kuwata}
\end{table}

The configurations are valid under the assumptions that in the first case $j(E)\neq 0, 12^3$, and in the second case $j(E)\,j(E')\neq 0$.
The Mordell-Weil ranks are computed from the Picard number given in Example \ref{Ex:Kummer} by Cor.~\ref{Cor:ST}.
In the first two cases, we list both the Mordell-Weil ranks in the CM case and in the non-CM case.
By inspection, the ranks cover all numbers from $0$ to $18$ except for $15$.

\subsection{}
Kloosterman filled the gap in Kuwata's list by exhibiting an elliptic K3 surface with Mordell-Weil rank $15$ \cite{Kl}.
Here we briefly sketch his construction.

Let $S$ be a rational elliptic surface with a fibre of type $\III^*$ and three $\I_1$'s.
These surfaces have Mordell-Weil rank one by Cor.~\ref{Cor:ST} and come in a one-dimensional family (write down the general Weierstrass form of a rational elliptic surface with an $\III^*$ fibre).

We construct an auxiliary K3 surface $X'$ by applying a general quadratic base change to $S$.
Then the trivial lattice $T(X')$ has rank 16, as it equals $U \oplus E_7^2$.
For general $X'$, the N\'eron-Severi group is thus generated by $T(X')$ and the pull-back of the $\MW$-generator of $S$.
Since quadratic base changes over $\PP^1$ have two moduli, $X'$ lives in a three-dimensional family with $\rho(X')\geq 17$.
Hence the general member actually has $\rho=17$.

Finally, we exhibit the cyclic base change of degree 4 to $X'$ that ramifies at the two $\III^*$ fibres.
In general, this yields another elliptic K3 surface $X$ with only $\I_1$ fibres.
Since $\rho(X)=\rho(X')$ by construction, $X$ has Mordell-Weil rank $r\geq 15$. Again we have $r=15$ in general.

Within this three-dimensional family, it remains to single out a member with $\rho=17$.
Kloosterman achieved this as follows:
he first determined a surface (over $\Q$) where by means of good reduction mod $p$, the Lefschetz fixed point formula implied $\rho\leq 18$ after point counting over $\F_q$ for $q=p,p^2,p^3$;
then he compared the reduction modulo another good prime $q$ to prove $\rho=17$ by an extension of a technique pioneered by van Luijk in \cite{vL} (related to the Tate conjecture, see \ref{s:Tate}).


\subsection{Explicit Mordell-Weil lattices}

Sometimes it is even possible to determine the Mordell-Weil lattices abstractly or in terms of generators.
For the base changes $X^{(n)}$ of Inose's pencil,
this was achieved in \cite{Shioda}.

Kloosterman's example above initiates a little detour to yet another area:
symplectic automorphisms of K3 surfaces,
i.e.~automophisms that leave the holomorphic two-form invariant.
In consequence, the desingularisation of the quotient is again a K3 surface.
Recall that we have seen one example in form of Nikulin involutions.

Symplectic automorphism groups $G$ on K3 surfaces $X$ have been classified completely starting with Nikulin for the abelian case \cite{N0}.
One crucial property is that the transcendental lattice $\T(X)$ is $G$-invariant.
Lattice theoretically this has the following effect.
Let $\Omega_G$ be defined as orthogonal complement
\[
\Omega_G = (H^2(X,\Z)^G)^\bot \subset H^2(X,\Z).
\]
Then $\Omega_G$ is in fact a primitive negative-definite sublattice of $\NS(X)$ that does not depend on $X$.
To compute $\Omega_G$, it thus suffices to work with one particular K3 surface $X$.
In practice, it is often convenient to work with elliptic K3 surfaces since here we have several means of prooducing symplectic automorphisms:
notably 
as quotients by translation by torsion sections and by base change (possibly composed with the hyperelliptic involution).
This approach has been pursued very successfully by Garbagnati and Sarti \cite{GS}.

Returning to Kloosterman's example above, it has the dihedral group on four elements, $D_4$, as group of symplectic automorphisms.
In \cite{Garba} Garbagnati determined $\Omega_{D_4}$ explicitly.
For the general member of Kloosterman's family, it follows that $\MWL\cong\Omega_{D_4}(-1)$.

\subsection{Mordell-Weil rank in characteristic $p$}

The corresponding $\MW$ rank problem in positive characteristic can be approached by similar methods.
For instance, the Kummer construction from Example \ref{Ex:Kummer} is well-defined in any odd characteristic $p>2$. 
The major difference is the occurence of supersingular K3 surfaces, but the Picard number is expressed by the same formula as in characteristic zero:
\begin{eqnarray}
\label{eq:rho-p}
\;\;\;\;\;\;\;\;\;\;
\rho(\Km(E\times E'))=18+\mbox{rank}(\mbox{Hom}(E,E')) \in\{18,19,20,22\}.
\end{eqnarray}
Note that over $\bar\F_p$, the Tate conjecture predicts an even parity of the Picard number (cf.~\ref{s:Tate}).
Hence $\rho=19$ cannot occur over $\bar\F_p$.

In fact, in characteristic $p>3$ one can still consider Inose's fibration (\ref{eq:Inose}) with the same Picard number as the above Kummer surface.
With Tables \ref{Table:Inose}, \ref{Table:Kuwata} adjusted according to (\ref{eq:rho-p}), we obtain elliptic K3 surfaces covering many $\MW$ ranks in characteristic $p>3$.

In particular, $X^{(5)}$ and $X^{(6)}$ yield K3 surfaces with $\MW$ rank 20 over $\bar\F_p$ as soon as there are at least two distinct supersingular elliptic curves over $\F_p$ (since maximal $\MW$ rank requires all fibres to be irreducible, so $E$ and $E'$ cannot be isomorphic).
By the well-known formulas, the latter holds true for $p=11$ and $p>13$.

\begin{Example}
Choose $E, E'$ with j-invariants $0$ resp.~$12^3$.
Then 
\[
 X^{(5)}: \;\;\; y^2 = x^3 + t\,(t^{10} +1) \;\;\; \mbox{and} \;\;\; X^{(6)}: \;\;\; y^2 = x^3 + t^{12} +1.
\]
These fibrations have $\MW$ rank 20 over $\bar\F_p$ if and only if both $E$ and $E'$ are supersingular mod $p$, i.e.~if $p\equiv -1\mod 12$.
\end{Example}

We note that all these supersingular K3 surfaces have Artin invariant $\sigma=1$ (i.e. $\mbox{disc}\; \NS(X)=-p^2$).
By a result of Ogus \cite{Ogus}, such a K3 surface is unique up to isomorphism. 
Through $X^{(5)}$ and $X^{(6)}$, this K3 surface is endowed with two distinct elliptic fibrations of maximal rank 20 for any pair of non-isomorphic supersingular elliptic curves over $\F_p$.
In particular, the number of distinct $\MW$ rank 20 fibrations grows with the characteristic (cf.~\cite{Sh-MW-p}).

\subsection{Mordell-Weil ranks of elliptic surfaces over $\Q$}
\label{ss:MW-Q}

So far we have mainly been concerned with N\'eron-Severi group and Mordell-Weil group of elliptic surfaces over algebraically closed ground fields $k$, thus regarding them as purely geometric objects.
As soon as we do not assume $k$ to be algebraically closed, arithmetic problems enter much more prominently.
In fact, these problems can be very interesting in their own right, and
often elliptic fibrations provide a genuine tool to study these problems.
We shall discuss a few of them throughout the remainder of this section.

\smallskip

Consider Kuwata's and Kloosterman's explicit examples of elliptic K3 surfaces over $\C$ for any given Mordell-Weil rank.
In any case, $\NS$ and $\MW$ are defined over some finite extension of the ground field the details of which we ignored so far.
However, the analogous $\MW$-rank question for fixed ground field is much more delicate.

Here we outline part of the solution to the problem over $\Q$.
The question of Mordell-Weil rank 18 over $\Q$ was formulated by Shioda in \cite{Shioda-20}. 
Elkies gave a negative answer in \cite{Elkies}, \cite{Elkies-rank}.
His line of  argument combined results specific to singular K3 surfaces with lattice theory.

\smallskip

By Cor.~\ref{Cor:ST}, Mordell-Weil rank 18 implies that an elliptic K3 surface is singular. 
Moreover the trivial lattice is only $U$, i.e.~all fibres are irreducible. 
In consequence, the Mordell-Weil lattice is even and integral, but has no roots.
This is a very special property, shared for instance by the Leech lattice.

Thanks to the Shioda-Inose structure, any singular K3 surface $X$ of discriminant $d<0$ can be defined over the ring class field $H(d)$ by means of Inose's fibration (\ref{eq:Inose}).
Here  $X$ may admit a model over some smaller field, but $H(d)$ is always preserved through its Galois action on $\NS(X)$ (cf.~\cite{S-NS}).
If the Mordell-Weil rank over $\Q$ is 18, this implies that $d$ has class number one, so $|d|\leq 163$. 
Because of the absence of roots, $\MWL$ would thus  break the density records for sphere packings in $\R^{18}$. 
By gluing up to a  Niemeier lattice as in \ref{ss:Nishi}, Elkies is able to establish a contradiction.

As for lower rank, Elkies found an elliptic K3 surface with Mordell-Weil rank 17 over $\Q$ (and necessarily $\rho=19$) \cite{Elkies-rank}. 
One can show that all intermediate ranks are also attained over $\Q$. 


\subsection{Modularity}
\label{ss:mod}

Modularity usually refers to varieties over $\Q$.
Here one asks for a relation to modular forms through the zeta function or $L$-series.
For a singular K3 surfaces, Inose's fibration (\ref{eq:Inose}) defines a model over the ring class field $H(d)$.
In \ref{ss:MW-Q}, we have reviewed an obstruction to descending $X$ to smaller number fields: 
the Galois action on the N\'eron-Severi group.
On singular abelian surfaces, there is a similar, but stronger obstruction to the corresponding problem stemming from the Galois representation of $H^1(A)$ (cf.~\cite[Lemma 6]{ES}).
For singular K3 surfaces, lattice theory imposes another substantially weaker condition
related to the genus of the transcendental lattice (cf.~\cite{Shimada-T}, \cite{S-fields}).

Specifically, we can ask for singular K3 surfaces over $\Q$.
By \cite{S-fields}, this requires that the genus of $\T(X)$ consists of a single class;
thus we are concerned with imaginary quadratic fields of class group exponent two.
Over some extension, the zeta functions is expressed through particular Hecke characters by \cite[Thm.~6]{SI}.
Over $\Q$, this relation specialises to classical modularity by a result of Livn\'e:

\begin{Theorem}[Livn\'e {\cite{L}}]
\label{Thm:mod}
Every singular K3 surface $X$ over $\Q$ is modular:
the Galois representation on $\T(X)$ is associated to a Hecke eigenform of weight 3.
\end {Theorem}

This result can be viewed as one of the first higher-dimensional analogues of the modularity of elliptic curves after Wiles 
\cite{Wiles}.
The converse was recently established by Elkies and Sch\"utt:

\begin{Theorem}[Elkies, Sch\"utt {\cite{ES}}]\label{Thm:geom}
Every known Hecke eigenform of weight 3 with eigenvalues $a_p\in\Z$ is associated to a singular K3 surface over $\Q$.
\end{Theorem}

The solution was first reduced to a finite problem by a combination of a finiteness result of Weinberger for imaginary quadratic fields with class group exponent two \cite{Wb}, and the classification of the above Hecke eigenforms up to twisting in \cite{S-CM}.
Then it remained to find singular K3 surfaces over $\Q$ for the 65 discriminants from Weinberger's list.
In support of exhibiting these surfaces, elliptic K3 surfaces enter through two key features:

First parametrise suitable one- and two-dimensional families of elliptic K3 surfaces over $\Q$ with Picard number $\rho\geq 19$ resp.~$18$.
Essentially, this is achieved through explicit calculations with extended Weierstrass forms and their discriminant, as it suffices to consider families with $\NS$ spanned by the trivial lattice.

Then determine a singular member of the family with additional section(s) such that the discriminant equals some given $d<0$.
By the height formula (\ref{eq:disc-NS-E}), the discriminant translates into information on the intersection behaviour of the section(s).
In particular, one can directly read off the fibre components that have to be met by the section(s).
Often these extra information reduce the complexity enough as to facilitate an explicit computation of the singular member of the family and the section(s).

\subsection{Shimura curves}
\label{ss:Shim}

There is another reason to be interested in one-dimensional families of (elliptic) K3 surfaces with $\rho\geq 19$ as in the proof of Thm.~\ref{Thm:geom}:
the parametrising curve is always a modular curve or a Shimura curve associated to a quaternion algebra over $\Q$.
In fact, as the explicit knowledge of Shimura curves is still fairly limited, K3 families enable us to draw interesting consequences for Shimura curves.
This approach was pursued by Elkies in \cite{E-Shimura};
he derived explicit defining equations and CM points of certain Shimura curves and also determined the corresponding QM-abelian surfaces.

The reason why K3 surfaces are easier to parametrise lies in the N\'eron-Severi group:
it has rank 19, so the level of the quaternion algebra (which gives the discriminant of $\NS$) is spread over substantially more divisor classes.

\smallskip

In practice, Elkies in \cite{E-Shimura} pursues a similar approach as in \ref{ss:mod}:
he determines a one-dimensional family of elliptic K3 surface with $\rho\geq 19$ and given discriminant;
then he computes CM points along the lines of \ref{ss:mod}.

The main step now consists in finding the corresponding abelian surfaces.
This can be achieved through a Shioda-Inose structure.
Namely we need an elliptic fibration with section and singular fibres of type $\II^*$ and $\III^*$.
Then we can apply Kumar's construction \cite{Kumar} to find the Igusa-Klebsch invariants of a family of genus 2-curves.
The Jacobians of these curves yield the corresponding abelian surfaces.

In practice, it might actually not be so easy to parametrise the required elliptic fibrations directly.
However, often the K3 surfaces admit other fibrations which are more readily found.
By Prop.~\ref{Prop:PSS}, it then suffices to identify perpendicular divisors of type $\tilde E_7$ and $\tilde E_8$ in $\NS(X)$ and compute their linear systems.
In \cite{E-Shimura}, Elkies achieves this in several steps.
Here we treat a related example.

\begin{Example}
Consider a K3 surface $X$ given by Inose's fibration (\ref{eq:Inose}) with two $\II^*$ fibres.
In the intersection graph consisting of zero section and $\tilde E_8$'s, we directly identify a divisor of type $\tilde D_{12}$ by omitting the vertices coresponding to the far double components of the $\II^*$ fibre.
The graph is sketched in Figure \ref{Fig:D16}.

\begin{figure}[ht!]
\setlength{\unitlength}{.45in}
\begin{picture}(10,3.7)(-0.5,-0.3)
\thicklines
\put(0,2){\line(1,0){7}}
\multiput(0,2)(1,0){8}{\circle*{.1}}
\put(2,2){\line(0,1){1}}
  \put(2,3){\circle*{.1}}
    \put(0.05,2.45){\makebox(0,0)[l]{$\tilde{E}_8$}}

\put(0,1){\line(1,0){7}}
\multiput(0,1)(1,0){8}{\circle*{.1}}
\put(2,1){\line(0,-1){1}}
  \put(2,0){\circle*{.1}}
    \put(0,0.55){\makebox(0,0)[l]{$\tilde{E}_8$}}
  
\put(8,1.5){\circle{.2}}
\put(8,1.5){\circle*{.1}}
\put(7,2){\line(2,-1){1}}
\put(7,1){\line(2,1){1}}
  \put(8.25,1.15){\makebox(0,0)[l]{$O$}}

\thinlines
\put(0.75,-0.25){\framebox(8.05,3.5){}}

\end{picture}
\caption{An $\tilde{D}_{16}$ divisor supported on the zero-section and components of the singular fibres}
\label{Fig:D16}
\end{figure}
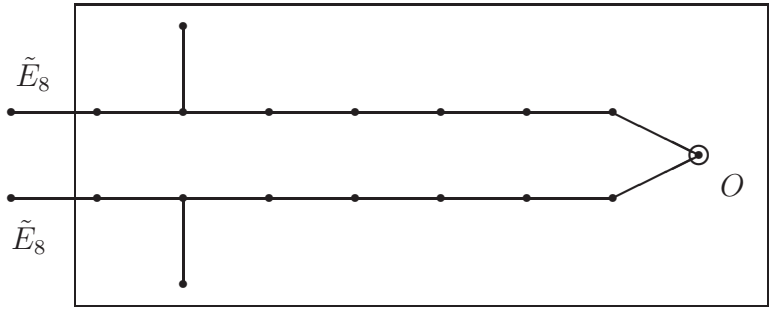

By Prop.~\ref{Prop:PSS}, $X$ admits an elliptic fibration with $\I_{12}^*$ fibre and sections induced by the omitted far components.
By itself, this fibration might be not so easy to parametrise.
However, the fibration can be deduced directly from (\ref{eq:Inose}) by choosing $u=x/t^2$ as an elliptic parameter.
One obtains the Weierstrass form
\[
 X:\;\;\; y^2 = t^2\,u^3 - 3\,A\,t^2\,u + t\,(t^2 - 2\, B\,t+1)
\]
over $k(u)$ in $(t,y)$-coordinates with $\I_{12}^*$ fibre at $u=\infty$ and two-torsion section $(0,0)$.
\end{Example}

\subsection{Real multiplication abelian surfaces}

Similar techniques can be applied to study abelian surfaces $A$ with real multiplication (RM):
here an order in a real quadratic field $K=\Q(\sqrt{D})$ embeds into the endomorphism ring of $A$.
A general RM abelian surface $A$ has Picard number $\rho(A)=2$, but only in the very first cases the general $A$ is known.
Elkies and Kumar again argue with K3 surfaces:
as in \ref{ss:Shim} the main idea is to exhibit a two-dimensional family of K3 surfaces with $\rho\geq 18$ and the right discriminant.
Then compute an appropriate elliptic fibration to find the Shioda-Inose structure.
Elkies reported on some results  for $D=33, 53$ in \cite{Elkies-RM}.

\subsection{Rational points}

Let $X$ be a variety over a number field $K$ (or over the function field of a curve). The rational points  are said to be  \emph{potentially dense} on $X$ if there is some finite extension $K'/K$ such that $X(K')$ is dense. Behind this notion stands the expectation that potential density should be a geometric property, depending only on the canonical class $K_X$.
Thanks to Faltings' theorem \cite{Fa},
this concept applies to curves: 
potential density holds true exactly for curves of genus zero and one.

The same is true for surfaces with $K_X$ negative: Over a finite extension, they are all rational. 
On the other hand, the Lang-Bombieri conjecture rules out potential density for projective varieties of general type.

Among surfaces with $K_X\equiv 0$, potential density has been proved for abelian and Enriques surfaces. 
In this context, elliptic fibrations enter naturally since positive Mordell-Weil rank over $K$ implies the density of $K$-rational points.
In this spirit, a result by Bogomolov and Tschinkel covers a great range of K3 surfaces \cite{BT} -- among them all elliptic K3 surfaces:

\begin{Theorem}[Bogomolov, Tschinkel]\label{Thm:BT}
Let $X$ be a K3 surface over a number field. Assume that $X$ has an elliptic fibration or infinite automorphism group. Then the rational points  are potentially dense on $X$.
\end{Theorem}

Note that the elliptic fibration of the theorem is not required to have a section.
In consequence, Cor.~\ref{Cor:ell-K3} implies that every K3 surface over some number field with $\rho\geq 5$ shares the property of potential density.


\subsection{Hasse principle}

The question of rational points is closely related to the Hasse principle: If a variety $X$ has a $\Q$-rational point, then it has $\Q_v$-rational points at every place $v$. Hence the set of adelic points $X(\A)$ is nonempty. The converse implication is known as the Hasse principle:
\[
X(\A)\neq\emptyset \;\, \stackrel{?}{\Longrightarrow} \;\, X(\Q)\neq \emptyset.
\]
The Hasse principle holds for conics, but is not true in general. 
In 1970, Manin \cite{Manin} discovered that the failure of the Hasse principle can often be explained through the Brauer group $\Br(X)$. 
He defined a subset $X(\A)^{\Br}\subset X(\A)$ that contains $X(\Q)$, but can be empty even if $X(\A)$ is not. This case is referred to as Brauer-Manin obstruction. 

\smallskip

For K3 surfaces, it is  an open problem whether the Brauer-Manin obstruction is the only obstruction to the Hasse principle. 
In the context of elliptic surfaces (not necessarily with section), Colliot-Th\'el\`ene, Skorobogatov and Swinnerton-Dyer pioneered an elaborate technique encompassing these problems \cite{CTSSD}.
There are direct applications to K3 surfaces;
subject to standard hypotheses, they imply the existence of rational points on certain K3 surfaces (cf.~\cite[pp.~585, 625/626]{CTSSD}).

The techniques of \cite{CTSSD} apply not only to K3 surfaces, but to several types of surfaces endowed with an elliptic fibration.
They have been further devolped by Wittenberg in \cite{Wit}.

\section{Ranks of elliptic curves}
\label{s:ranks}

In this section, we shall give a brief discussion of possible Mordell-Weil ranks of elliptic surfaces.
In order for this problem to make sense, the base curve $C$ has to be fixed;
otherwise one can always apply a base change to a curve of higher genus to split some multisection and thus increase the $\MW$ rank.
Here we will almost exclusively  be concerned with the case $C\cong \PP^1$.

\subsection{Noether-Lefschetz loci}
\label{ss:N-L}

Before discussing Mordell-Weil ranks, we have to look into Picard numbers in generality.
Let $S\to\PP^1$ denote an elliptic surface with section over $\C$.
Upon fixing the arithmetic genus $\chi(S)=n$, the moduli space of these surfaces has dimension $10\,n-2$ by inspection of a globally minimal Weierstrass form as in \ref{ss:min}:
the degrees $4\,n+1$ resp.~$6\,n+1$ of $a_4$ and $a_6$ minus 4 normalisations by M\"obius transformations and scaling.
Here we are not concerned with rational elliptic surfaces, since then always $\rho=10$.
Hence we shall assume that $\chi(S)>1$.

\smallskip

Recall that $\NS(S)$ is generated by horizontal and vertical divisors by Cor.~\ref{Cor:ST}.
Hence an increase of the Picard number within a family of elliptic surfaces can be achieved by two means:
\begin{itemize}
 \item 
by degenerating the singular fibres in order to produce an additional fibre component;
\item
by forcing an additional, i.e.~independent section.
\end{itemize}
From the discriminant $\Delta$, it is visible that the first degeneration causes just one condition on the parameters of the family: 
the degenerate surfaces ought to live in one-dimensional subspaces of the moduli space.
Kloosterman gave a rigorous proof of this fact in \cite{Kl-NL} based on the model as hypersurface in weighted projective space from \ref{ss:IP^1}.

\begin{Theorem}
\label{Thm:NL}
Consider non-isotrivial complex elliptic surfaces over $\PP^1$ with section and fixed arithmetic genus $\chi\geq 2$.
The surfaces with Picard number $\rho\geq r\; (r\geq 2)$ have $(10\, \chi-r)$-dimensional Noether-Lefschetz locus
\end{Theorem}

To some extent, isotrivial elliptic surfaces can be dealt with as well (cf.~\cite[Cor.~1.2]{Kl-NL}).
The K3 case of $\chi=2$ follows from the period map (\ref{ss:period}).
For $r=3$, the result was previously proven by Cox \cite{Cox-NL} as we shall discuss in the next section.

\subsection{Section heuristics}

While the above fibre degeneration argument seems to work in any characteristic, 
the problem of additional sections is more subtle;
in particular, it depends on the characteristic.

\smallskip

First we discuss a heuristic argument:
Let $P=(X(t), Y(t))$ be an integral section on an integral model of the elliptic surface $S\to\PP^1$ in the notion of \ref{ss:int-surf}.
Then $X(t)$ has degree $2\chi$ and $\deg(Y(t)) = 3\chi$ in general.
In total, these polynomials have $5\chi+2$ coefficients;
on the other hand, the Weierstrass form has degree $6\chi$ in $t$.
Hence there are $(\chi-1)$ more equations (at $t^i,\, i=0,\hdots,6\,\chi$) than coefficients available.
These heuristics suggest that increasing the Mordell-Weil rank by one should cost $\chi-1$ moduli dimensions.

So far these arguments have been made explicit only in the first case by Cox \cite{Cox-NL}:

\begin{Theorem}
\label{Thm:Cox}
Consider non-isotrivial complex elliptic surfaces over $\PP^1$ with section and fixed arithmetic genus $\chi\geq 2$.
The surfaces with Mordell-Weil rank $\geq 1$ have $(9\, \chi-1)$-dimensional moduli.
\end{Theorem}

For higher $\MW$ rank. Thm.~\ref{Thm:NL} gives an upper bound for the dimension of the Noether-Lefschetz locus;
however this bound cannot be expected to be optimal, as suggested by Thm.~\ref{Thm:Cox}.

\subsection{Mordell-Weil rank over $\C$}

The above heuristics show why it is hard to produce elliptic surfaces of high $\MW$ rank.
At this point, the main hope seemingly consists in finding a canonical (geometric) construction as in \ref{ss:Kuwata}.
In particular, it is unclear whether Mordell-Weil ranks of complex elliptic surfaces over $\PP^1$ can be arbitrary large:

\begin{Question}
Is there an upper bound for the Mordell-Weil ranks of complex elliptic surfaces over $\PP^1$?
\end{Question}

One should point out that Lapin once claimed that the $\MW$ rank is unbounded \cite{Lapin}.
However, some gaps were discovered in his proof (cf.~\cite{Schoen}), so the question remains open.

\subsection{}
To our knowledge, the current record for Mordell-Weil ranks of complex elliptic surfaces over $\PP^1$ is $68$.
it is attained by some of the isotrivial elliptic surfaces
\begin{eqnarray}\label{eq:m}
 S:\;\;\; y^2 = x^3 + t^m + 1\;\;\;\; (m\in \N).
\end{eqnarray}

\begin{Theorem}[Shioda (1991)]
The elliptic surface $S$ over $\C$ has Mordell-Weil rank $r\leq 68$.
Moreover
\[
 r=68\;\; \Longleftrightarrow\;\; m\equiv 0\mod 360.
\]
\end{Theorem}

The proof relies on the fact that $S$ is a Delsarte surface, i.e.~covered by a Fermat surface.
Since the Lefschetz number $\lambda(S)$ is a birational invariant, it can be computed from the Galois cover.
Then $b_2(S)=\lambda(S)+\rho(S)$ gives the Picard number.
The Mordell-Weil rank is easily obtained from Cor.~\ref{Cor:ST} since there are no reducible fibres except possibly at $\infty$.

This argument was exhibited in detail by Shioda in \cite{Shioda-AJM} where the previous rank record of $56$ was achieved by another class of 
Delsarte elliptic surfaces
 which is non-isotrivial (i.e. with non-constant $j$-invariant).

The lattice structure of the Mordell-Weil lattice for the elliptic surface (\ref{eq:m}) for any $m$ has been  determined by Usui \cite{Usui}. For example, in case $r=68$, its determinant is equal to $2^{136} 3^{102} 5^{40}$.

\subsection{Mordell-Weil ranks in positive characteristic}

The approach through Delsarte surfaces is also very useful in positive characteristic.
In fact, it enables us to show that $\MW$-ranks are not bounded.

The reason for this lies in unirationality.
Namely, by \cite{KaS} the Fermat surface of degree $n$ is unirational in charateristic $p$ if and only if
\[
 \exists \, \nu\in\N:\;\;\; p^\nu \equiv -1\mod n.
\]
Since unirationality implies supersingularity ($b_2=\rho$), one can thus construct supersingular elliptic surfaces in abundance:
by controlling the trivial lattice, one can achieve arbitrarily large $\MW$ ranks.
For instance, the elliptic surface in (\ref{eq:m}) is unirational in characteristic $p$ if $p\equiv -1\mod 3m$ (or if this holds for some power of $p$).

The first proof that $\MW$-ranks are unbounded in any positive characteristic goes back to Tate  and Shafarevich \cite{S-Tate} who used supersingular Fermat curves.
Their proof is based on the validity of the Tate conjecture of abelian varieties over finite fields \cite{Tate-Endo}
that we will briefly discuss below in \ref{s:Tate}.
In particular, the construction of Tate and Shafarevich works over fixed finite fields.
This contrasts with the geometric reasoning in \cite{KaS} over an algebraically closed field.

\smallskip

Mordell-Weil lattices in positive characteristic had a great impact on the sphere packing problem.
Namely they enabled Elkies and one of us independently to derive lattices with greater density than previously known (cf.~\cite[pp.~xviii-xxii]{CS}, \cite{Elki}, \cite{Shio-SP}).

\subsection{Tate conjecture}
\label{s:Tate}

Another feature of positive characteristic is that the Noether-Lefschetz loci behave in a different way.
This observation can be easily checked for explicit examples;
in general it depends on the validity of the \textbf{Tate conjecture}.

\smallskip

The Tate conjecture \cite{Tate-C} states that over a finite field $\F_q$, the Picard number of an algebraic surface $X$ can be read off from the induced action of the Frobenius morphism Frob$_q$ on $H^2(X)$.
Namely, consider the subgroup $N$ of $\NS(X)$ generated by divisor class defined over $\F_q$.
Conjecturally, the rank of $N$ equals the multiplicity of $q$ as an eigenvalue of Frob$_q^*$ on $H^2(X)$ (for a suitable \'etale cohomology with $\ell$-adic coefficients).
The Tate conjecture is only known for specific classes of surfaces, such as abelian surfaces and products of curves \cite{Tate-Endo}, elliptic K3 surfaces with section \cite{ASD} and Fermat surfaces \cite{KaS}.

One particular consequence of the Tate conjecture concerns the (geometric) Picard number of $X$ (i.e.~over $\bar\F_q$):
it equals the number of eigenvalues of Frob$_q^*$ which have the shape $q$ times a root of unity.
By the Weil conjectures, all other eigenvalues come in complex conjugate pairs.
Hence the Tate conjecture implies that the difference $b_2(X)-\rho(\bar X)$ has even parity.
In other words, divisors come always in pairs on surfaces over finite fields!

\smallskip

The Tate conjecture has the following impact on an elliptic surfaces over $\bar\F_q$:
Suppose we degenerate or specialise a given surface over $\bar\F_q$ so that there is an additional fibre component (or section).
Then the Tate conjecture 
forces the occurrence of another independent divisor.
Generally this divisor will appear as an independent section.

\subsection{Specialisation}

When working over number fields, there is one particular reason to be interested in elliptic surfaces with high Mordell-Weil rank:
specialisation yields elliptic curves with high rank.
Namely, by a result of 
N\'eron and Silverman \cite{Silv}, the Mordell-Weil group of the general member in a one-dimensional family of abelian varieties injects into the Mordell-Weil groups of the specialisations except at finitely many places  over any given (finite) number field.

Applied to the elliptic K3 surface of $\MW$ rank 17 mentioned in \ref{ss:MW-Q},
specialisation yields infinitely many elliptic curves over $\Q$ with rank at least 17.
This construction was accelerated by Elkies in \cite{Elkies} as follows.
First he exhibited a base change from $\PP^1$ to an elliptic curves $E$ of positive rank such that the $\MW$ rank increased to 18.
Then he studied the fibres  of the base changed elliptic surface over $\Q$ -- generally elliptic curves of rank at least 18 over $\Q$.
Among them, Elkies found one particular fibre that attains the current record of an elliptic curve over $\Q$ with rank at least 28.

In a way, this new record curve puts an interesting twist to the previous achievements:
the first record curves were based on geometric considerations, in particular cubic pencils and rational elliptic surfaces.
This approach culminated  in the famous work of N\'eron (1954) to construct infinitely many elliptic curves over $\Q$ of rank at least 11 (cf. Serre \cite{Serre}).
N\'eron's technique was  later clarified  with
the method of Mordell-Weil lattices in \cite{Sh-Neron}.
Then, starting from an elliptic curve over $\Q(t)$ of rank at least 12 
by Mestre  and  others, 
extensive computer searches played the dominant role in subsequent years, leading to the previous record curves of rank (at least) 24.
It is now with Elkies' work that geometric arguments are reclaiming their prominent role.

Similar techniques have been pursued for elliptic curves over $\Q$ with fixed torsion subgroup (cf.~\cite{Elkies-rank}).
Note that here the Mordell-Weil ranks of the elliptic surfaces are much smaller compared to Picard number and Euler number of the elliptic surface, since the torsion sections force reducible singular fibres by \ref{ss:quot-sect}.

\subsection*{Acknowledgements} 
We thank the organisers of the conference ``Algebraic Geometry in East Asia'' for inviting us to participate in this fruitful meeting and contribute this paper to the proceedings.
The first author is indepted to I.~Shimada for supporting and arranging his visit.

Most of this paper was written while the first author held a position at University of Copenhagen,
and the first draft was almost finished during a visit of the second author.
We are grateful to the Department of Mathematics for their support.

We would like to express our gratitude to all our colleagues for discussions and collaborations throughout the years.
Particular thanks go to the students who attended our lectures on the subject of this survey. 
It was on one of these occasions that we first met in Milano in 2006, and we would like to take this opportunity to thank B.~van Geemen for making this initial meeting possible.

We are grateful to I.~Dolgachev, R.~Kloosterman, D.~Lorenzini, and H.~Partsch for their helpful comments.
Our thanks go to the referee for many comments that helped us improve the paper.


%




\end{document}